\documentclass[preprint,11pt]{elsarticle}
\usepackage{amsmath, amssymb,amsthm, fullpage, color, relsize, enumerate, enumitem, hyperref, mathtools}
\usepackage{latexsym}
\usepackage{bm,bbm} 
\usepackage{graphicx,adjustbox}
\usepackage{chngcntr}
\usepackage{import}
\usepackage{comment}
\usepackage{multirow,multicol}
\usepackage{subfig}
\usepackage{caption}
\numberwithin{equation}{section}
\usepackage{pgfplots}
\pgfplotsset{compat=newest} 
\pgfplotsset{plot coordinates/math parser=false}

\DeclareMathOperator{\tr}{tr}

\theoremstyle{definition}
\newtheorem{defn}{Definition}[section]

\newtheorem{prop}{Proposition}[section]

\newtheorem{lem}{Lemma}[section]
\newtheorem{thm}{Theorem}[section]
\newtheorem{cor}{Corollary}[section]
\newtheorem{rem}{Remark}[section]
\newtheorem{assu}{Assumption}[section]

\newtheorem{clawithinpf}{Claim}
\makeatletter
\@addtoreset{clawithinpf}{proof}
\makeatother

\newcommand{\suchthat}{\;\ifnum\currentgrouptype=16 \middle\fi|\;}

\newcommand\blfootnote[1]{%
	\begingroup
	\renewcommand\thefootnote{}\footnote{#1}%
	\addtocounter{footnote}{-1}%
	\endgroup
}

\newcommand{\bb}[1]{\mathbb{#1}} 
 
\newcommand{\cc}[1]{\mathcal{#1}} 
\newcommand{\ff}[1]{\mathfrak{#1}} 
\newcommand{\zero}{\boldsymbol{0}}

\newcommand{\cov}{\mathop{\mathrm{cov}}} 
\newcommand{\var}{\mathop{\mathrm{var}}} 
\newcommand{\diam}{\mathop{\mathrm{diam}}} 
\newcommand{\dist}{\mathop{\mathrm{dist}}}

\newcommand{\RHS}{\mathop{\mathrm{RHS}}} 
\newcommand{\LHS}{\mathop{\mathrm{LHS}}}  

\newcommand{\argmax}{\operatornamewithlimits{\rm argmax}}

\newcommand{\brac}[1]{\left[#1\right]}
\newcommand{\set}[1]{\left\{#1\right\}}
\newcommand{\abs}[1]{\left\lvert #1 \right\rvert}
\newcommand{\paren}[1]{\left(#1\right)}

\newcommand{\brapar}[1]{\left[#1\right)}
\newcommand{\Bigparen}[1]{\Big(#1\Big)}

\newcommand{\Bigset}[1]{\Big\{#1\Big\}}

\newcommand{\InnerProd}[2]{\langle #1,#2 \rangle}

\newcommand{\cp}[1]{\overset{#1}{\rightarrow}}    
\newcommand{\eqd}{\overset{d}{=}}  
\newcommand{\RelNum}[2]{\overset{#1}{#2}}

\newcommand{\LpNorm}[2]{\left\| #1\right\|_{\ell_{#2}}}
\newcommand{\SpNorm}[2]{\left\| #1\right\|_{\cc{S}_{#2}}}
\newcommand{\norm}[2]{\left\|#1\right\|_{#2}}
\newcommand{\OpNorm}[3]{\left\| #1\right\|_{#2\rightarrow#3}}
\newcommand{\RNum}[1]{\uppercase\expandafter{\romannumeral #1\relax}}

\pgfplotsset{every axis/.append style={
		scaled y ticks = false, 
		scaled x ticks = false, 
		y tick label style={/pgf/number format/.cd, fixed, fixed zerofill,
			,precision=1},
		x tick label style={/pgf/number format/.cd, fixed, fixed zerofill,
			precision=1}
	}
}

\newlength\figureheight 
\newlength\figurewidth

\journal{Journal of Multivariate Analysis}

\begin{document}

\begin{frontmatter}



\title{On the consistency of inversion-free parameter \\
estimation for Gaussian random fields \tnoteref{t1} \blfootnote{Email: \textit {\{hksh,clayscot,xuanlong\}@umich.edu}}}


\author[1]{Hossein Keshavarz}
\author[2,1]{Clayton Scott}
\author[1,2]{XuanLong Nguyen}

\address[1]{Department of Statistics, University of Michigan}
\address[2]{Department of Electrical Engineering and Computer Science, University of Michigan}
\tnotetext[t1]{This research is partially supported by NSF grant ACI-1047871. Additionally, CS is partially supported by NSF grants 1422157, 1217880, and 0953135, and LN by NSF CAREER award DMS-1351362, NSF CNS-1409303, and NSF CCF-1115769.}

\begin{abstract}
Gaussian random fields are a powerful tool for modeling environmental processes. For high dimensional samples, classical approaches for estimating the covariance parameters require highly challenging and massive computations, such as the evaluation of the Cholesky factorization or solving linear systems. Recently, Anitescu, Chen and Stein \cite{M.Anitescu} proposed a fast and scalable algorithm which does not need such burdensome computations. The main focus of this article is to study the asymptotic behavior of the algorithm of Anitescu {\rm et al.} (ACS) for regular and irregular grids in the increasing domain setting. Consistency, minimax optimality and asymptotic normality of this algorithm are proved under mild differentiability conditions on the covariance function. Despite the fact that ACS's method entails a non-concave maximization, our results hold for any stationary point of the objective function. A numerical study is presented to evaluate the efficiency of this algorithm for large data sets.
\end{abstract}


\end{frontmatter}


\section{Introduction}

Gaussian processes have plethora of applications, ranging from the modeling of environmental processes in \emph{geostatistics} (e.g., \cite{N.Cressie, AE.Gelfand}) to supervised regression and classification in \emph{machine learning} \cite{J.Bernardo,W.Chu,CE.Rasmussen}, and the simulation of complex computer models \cite{J.Sacks}. The versatility of the correlation structure of Gaussian processes provides a tractable and powerful tool for the modeling of large and highly dependent environmental variables. As a common approach in the field of spatial statistics, the covariance functions of Gaussian processes are assumed to belong to a parametric family. High precision estimates of the covariance parameters are pivotal for interpolating Gaussian processes which is the ultimate goal in many geostatistical problems \cite{N.Cressie, ML.Stein2}.

In the last two decades, there has been extensive research regarding the statistical and computational facets of the estimation of the Gaussian processes' covariance parameters. Maximum likelihood estimation (MLE) was the earliest favored algorithm in the geostatistics community, e.g., Mardia {\rm et al.} \cite{KV.Mardia} and Ying \cite{Z.Ying}. However, solving systems of linear equations is inevitable to evaluate the Gaussian likelihood. Notwithstanding the recent advances toward scalable and efficient solution of the system of linear equation ( e.g., iterative \emph{Krylov} subspace method or block preconditioned conjugate gradient algorithm \cite{DP.OLeary}) which moderately reduces the computational and memory costs of the direct evaluation of the precision matrix, obtaining the MLE of unknown covariance parameters using such linear systems solvers is still a strenuous task, especially for a generic Gaussian spatial process observed at numerous and possibly irregularly spaced locations. Approximating the likelihood function by tapering the covariance matrix is another class of algorithms aiming to reduce the numerical burden of MLE (see Kaufman {\rm et al.} \cite{CG.Kaufman}). Due to the sparsity of the tapered covariance matrix, its inverse can be computed in a faster and more stable way. Recent studies \cite{J.Du,CG.Kaufman,D.Wang} demonstrate the consistency and asymptotic normality of this algorithm under some mild conditions on the taper function.

Because of the obstacles of solving system of linear equations for massive data, which is necessary for tapered and exact MLE, it is of great interest to develop estimation techniques without requiring such extensive computations. Such class of algorithms, which will be referred to as \emph{inversion-free}, are based upon optimizing a loss function whose form (and its derivatives) is independent of the precision matrix of data. The first attempt toward such a goal has been done by Anitescu, Chen and Stein \cite{M.Anitescu} in $2014$ (referred here to as ACS). Their proposed procedure is faster and more stable than likelihood based algorithms. In \cite{M.Anitescu}, the covariance parameters are estimated by computing the global maximizer of a non-concave program. Simulation studies verify the efficiency of ACS's approach in the case that the covariance matrix has a bounded condition number. The main purpose of this paper is to appraise the asymptotic properties of the ACS's algorithm such as consistency, minimax optimality and asymptotic normality. The developed theory in this paper shows that ACS's algorithm has the same asymptotic rate of convergence as the MLE. In practice, the solution of ACS's optimization problem may also serve as the starting point (initial guess) of a likelihood maximization procedure.

In geostatistics, there are two common asymptotic regimes: \emph{increasing-domain} and 
\emph {fixed-domain}, the latter sometimes referred to as infill asymptotics (see \cite{ML.Stein2}, Section $3.3$ or \cite{N.Cressie}, p. $480$). In the former setting, the minimum distance among the sampling points is bounded away from zero and more samples are collected by increasing the diameter of the spatial domain. In the latter regime, the data are sampled in a fixed and bounded domain, and the observations get denser as the sample size $n$ increases.

Zhang \cite{H.Zhang} showed that not all the covariance parameters are consistently estimable in the fixed domain regime. Strictly speaking, there is no asymptotically consistent algorithm for estimating the \emph{non-micro ergodic} covariance parameters, which do not asymptotically affect the interpolation mean square error (see \cite{ML.Stein2} for a precise definition). On the other hand, it is known in the literature that subject to some mild regularity conditions, maximizing the likelihood provides a strongly consistent and asymptotically normal estimate for all the covariance parameters in the increasing domain setting \cite{F.Bachoc,KV.Mardia}.

Increasing domain asymptotic analysis of covariance estimation has two significant benefits. First, unlike the infill asymptotic setting, the geometry of the spatial sampling has a crucial impact on the asymptotic distribution of the parameter estimate. Thus, this regime is a natural asymptotic framework for assessing the role of irregularity of spatial sampling on the covariance parameter estimation \cite{F.Bachoc}. This claim can be verified by a deeper look at the asymptotic distribution of the microergodic parameter estimates in the fixed domain (see e.g., \cite{ML.Stein2,Z.Ying} for MLE and \cite{J.Du,CG.Kaufman,D.Wang} for tapered MLE). Another significant characteristic of increasing domain regime is that the covariance matrix has a universally bounded condition number as $n$ grows under some mild regularity conditions. This feature of the covariance matrix plays a major role in our asymptotic analysis. Although in many geostatistical applications in a fixed bounded domain the condition number of the covariance matrix increases at least linearly with respect to $n$, preconditioning filters is commonly used to uniformly control the condition number independent of $n$ \cite{J.Chen,ML.Stein2}. Therefore, we believe that our developed increasing domain asymptotics can be useful for the fixed domain analysis of preconditioned inversion-free algorithms.

\paragraph{Outline of main results.}
This paper studies the increasing domain asymptotic behaviour of ACS's estimation algorithm introduced in \cite{M.Anitescu}. Specifically, suppose that $\ff{G}$ is a zero mean stationary Gaussian process in $\bb{R}^d$ with covariance function $\cov\paren{\ff{G}\paren{s},\ff{G}\paren{s'}} = R\paren{s-s',\eta}$ in which $\eta\in\Omega$ denotes the vector of unknown covariance parameters. One realization of $\ff{G}$ has been observed on a \emph{$d-$dimensional perturbed regular lattice} of $n = N^d$ points, which will be formally defined in Section \ref{ProbForm}. 
The specific contributions of this work are given as follow:

\begin{enumerate}[label = (\alph*)]
\item Assuming the polynomial decay of $R\paren{s,\eta}$ and its gradient (with respect to $\eta$) in terms of the Euclidean norm of $s$, and under some mild identifiability condition on $R$, we prove that the global maximizer of ACS's method consistently estimates $\eta$. Furthermore, the estimation error is of order $\sqrt{n^{-1}\ln n}$ which is shown to be minimax optimal up to some $\sqrt{\ln n}$ term. 
\item As the proposed loss function in Anitescu {\rm et al.} \cite{M.Anitescu} is not jointly concave in $\eta$, finding its global maximizer is challenging. For a large enough sample size and under an additional condition regarding the polynomial decay of the second derivative of $R\paren{s,\eta}$ with respect to $\eta$, we show that any \emph{stationary point} of this non-concave program is concentrated around the true $\eta$ with radius of order $\sqrt{n^{-1}\ln n}$. 
\item The asymptotic normality of the stationary points of the aforementioned algorithm will be substantiated under some mild restriction on the third derivative of $R$ with respect to $\eta$. 
\end{enumerate}

Furthermore, the Appendix contains several easy-to-reference nonasymptotic results on the global and local behaviour of the quadratic forms of Gaussian processes which may come in handy for the analysis of certain problems in statistics and machine learning.

\paragraph{Plan of the paper.}
In Section \ref{ProbForm}, we formulate ACS's inversion-free estimation method and precisely introduce the geometry of the sampling points. Section \ref{Analysis} expresses the necessary assumptions and studies the asymptotic properties of the estimation algorithm. Section \ref{ConvRate} presents the convergence rate of the global and local maximizers of the optimization problem introduced in Section \ref{ProbForm}. We investigate the minimax optimality and the asymptotic normality of the local maximizers in Section \ref{MinMaxOptimandAsympNorm}. The objective of Section \ref{SimulRes} is to assess the performance of ACS's algorithm and verify the developed theory using simulation studies on synthetic data. Section \ref{Discus} serves as the conclusion and discusses the future directions. Section \ref{Proofs} presents the proof of the main results. Finally, the Appendix contains some auxiliary technicalities on the nonasymptotic behaviour of the quadratic forms of Gaussian processes and of large covariance matrices with polynomially decaying off-diagonal entries, which are essential in Section \ref{Proofs}.

\paragraph{Notation.}
For any $m\in\bb{N}$, $I_m$ and $\zero_m$ respectively denote the $m$ by $m$ identity matrix and all zeros column vector of length $m$. Moreover, $\wedge$ and $\vee$ stand for the minimum and maximum operators. For two matrices of the same size $M$ and $M'$, $\InnerProd{M}{M'}{}\coloneqq \sum_{i,j} M_{ij}M'_{ij}$ denotes their usual inner product. We use the following matrix norms on $M\in\bb{R}^{m\times n}$. For any $1\leq p<\infty$, $\LpNorm{M}{p}\; \coloneqq \paren{\sum_{i,j} \abs{M_{ij}}^p}^{1/p}$ stands for the element-wise $p-$norm of $M$. The usual operator norm (largest singular value of $M$) is represented by $\OpNorm{M}{2}{2}$. Associated to any finite set $\cc{D}\subset\bb{R}^d$ and $s\in\cc{D}$, we define $\cc{D}-s \coloneqq \set{s'-s:\;s'\in\cc{D}}$. Moreover, we write $\cc{D}\paren{s,r} \coloneqq \set{s'\in\cc{D}:\;\LpNorm{s'-s}{2}\leq r}$ and $\cc{D}^c\paren{s,r} = \cc{D}\setminus \cc{D}\paren{s,r}$, for any non-negative $r$. $\cc{S}^m$ stands for the $m-$dimensional unit sphere with respect to the Euclidean norm, i.e., $\cc{S}^m\coloneqq \set{v\in\bb{R}^{m+1}:\; \LpNorm{v}{2}=1}$. For a random sequence $x_n$ and a deterministic positive sequence $a_n$, we write $x_n=\cc{O}_{\bb{P}}\paren{a_n}$ when $x_n$ is bounded below by $a_n$ asymptotically, i.e., $\lim\limits_{n\rightarrow\infty} \Pr\paren{\abs{x_n}\geq Ca_n} = 0$ for some $C>0$.  For two sets  $\Omega_1,\Omega_2\subset\bb{R}^m$, $\dist\paren{\Omega_1,\Omega_2} \coloneqq \inf_{\omega_i\in\Omega_i,\;i=1,2} \LpNorm{\omega_1-\omega_2}{2}$ represents their mutual distance with respect to the Euclidean norm. Moreover, for $\cc{A}\subset\bb{R}^m$ and $r>0$, $\cc{N}_r\paren{\cc{A}}$ denotes a subset of $\cc{A}$ (of minimal size) such that for any $a\in\cc{A}$, $\dist\paren{\set{a},\cc{N}_r\paren{\cc{A}}}\leq r$. The cardinality of such set will be referred as the \emph{covering number} of $\cc{A}$. Given spatial points $\set{s_1,\ldots,s_n}\in\bb{R}^d$ and the covariance function $R\paren{\cdot,\eta}$ parametrized by $\eta = \paren{\eta_1,\ldots,\eta_m}$, the associated covariance matrix and its derivatives are defined as
\begin{equation*}
R_n\paren{\eta} = \brac{R\paren{s_i-s_j,\eta}}^n_{i,j=1},\quad\frac{\partial}{\partial\eta_r}R_n\paren{\eta} = \brac{\frac{\partial}{\partial\eta_r}R\paren{s_i-s_j,\eta}}^n_{i,j=1},\;\;\forall\;r=1,\ldots,m.
\end{equation*}
The higher order derivatives can be defined in an analogous way. For two random vectors $v_1$ and $v_2$, the expression $v_1\eqd v_2$ means that they have the same distribution. Lastly, $D\paren{\bb{P}_1 \parallel \bb{P}_2}$ indicates the \emph{Kullback-Leibler} divergence of two distributions $\bb{P}_i,\;i=1,2$.

\section{Problem set up and ACS's estimation algorithm}\label{ProbForm}
Consider a mean zero and \emph{stationary} (real valued) Gaussian process $\ff{G}:\bb{R}^d\mapsto\bb{R}$ whose covariance function belongs to a parametric family $\cc{C}_{R,\Omega}\coloneqq \set{R\paren{\cdot,\eta}:\eta\subset\Omega}$. In other words, there exists $\eta_0\in\Omega$ for which 
\begin{equation}\label{CovStrct}
{\rm E} \ff{G}\paren{s}\ff{G}\paren{s'} = R\paren{s-s',\eta_0},\quad \forall\;s,s'\in\bb{R}^d.
\end{equation}
Moreover, there is $m\in\bb{N}$ such that $\Omega$ is a \emph{compact} $\paren{m+1}$ dimensional subset of $\bb{R}^{m+1}$ with respect to the Euclidean topology. Thus, $\cc{C}_{R,\Omega}$ is assumed to be a \emph{finite dimensional} class. For analytical convenience, we consider an alternative formulation for the unknown parameters of the covariance function given as
\begin{equation*}
\eta_0 = \paren{\phi_0,\theta_0},\quad \phi_0\in\cc{I},\;\theta_0\in\Theta.
\end{equation*}
In this new representation, $\phi_0$ is a strictly positive scalar denoting the variance of $\ff{G}$ and the $m-$dimensional vector $\theta_0$ stands for the other parameters of $R$. Moreover, $\cc{I}\subset\paren{0,\infty}$ is a bounded interval and $\Theta\subset\bb{R}^m$ is compact. For instance in isotropic Matern or powered exponential classes, $\theta_0$ is a positive vector representing the \emph{range parameter} and \emph{fractal index}. Finally, \eqref{CovStrct} can be rewritten as $R\paren{s-s',\eta_0} = \phi_0 K\paren{s-s',\theta_0}$, in which $K\paren{\cdot,\theta_0}$ indicates the correlation function parametrized by $\theta_0$.

The objective is to estimate $\eta_0$ observing one realization of $\ff{G}$ at a deterministic set of spatial locations $\cc{D}_n = \set{s_1,\ldots,s_n}\subset\bb{R}^d$. It is beneficial to emphasize that our asymptotic analysis lies in the increasing domain regime in which the diameter of $\cc{D}_n$ tends to infinity as $n\rightarrow\infty$. The collected samples form a zero mean Gaussian column vector $Y = \brac{\ff{G}\paren{s_1},\ldots,\ff{G}\paren{s_n}}^\top$ of length $n$. Before proceeding further, let us precisely introduce the geometric structure of $\cc{D}_n$.

\begin{assu}\label{ObsLocAssu}
Suppose that there is $N\in\bb{N}$ such that $n = N^d$. There exists $\delta\in\brapar{0,1/2}$ for which $\cc{D}_n$ is a \emph{$d-$dimensional $\delta-$perturbed regular lattice} (with unit grid size). Namely, 
\begin{equation*}
\cc{D}_n = \set{v_i+\delta p_i:\;v_i\in\cc{V}_{N,d},\; p_i\in\brac{-1,1}^d}^n_{i=1},
\end{equation*}
in which $\cc{V}_{N,d} \coloneqq \set{v_1,\ldots,v_n} = \set{1,\ldots,N}^d$ denotes the $d-$dimensional regular lattice.
\end{assu}

The condition $\delta\in\brapar{0,1/2}$ guarantees the existence of a strictly positive minimum distance $\paren{1-2\delta}$ between the distinct points in $\cc{D}_n$. The scalar $\delta$ quantifies the amount of irregularity in $\cc{D}_n$. In the case of $\delta = 0$, $\cc{D}_n$ forms a regular lattice and the irregularity can be more apparent as $\delta$ increases. Although the absence of randomness in Assumption \ref{ObsLocAssu} may appear problematic at first sight, our theoretical contributions are not restricted to any further set of strong conditions on $p_i$'s. For instance, the presented results in the next section hold almost surely if $p_i$'s are independent (or even dependent) draws of a distribution supported on $\brac{-1,1}^d$ which is absolute continuous with respect to the Lebesgue measure.. 

Now we present ACS's estimation algorithm introduced in \cite{M.Anitescu}. Define,
\begin{equation}\label{InvFreeOptProbEtaForm}
\hat{\eta}_n = \argmax_{\eta\in\Omega} F_n\paren{Y,\eta},\quad\mbox{where}\quad F_n\paren{Y,\eta}\coloneqq \frac{1}{n}\set{Y^\top R_n\paren{\eta}Y - \frac{1}{2}\LpNorm{R_n\paren{\eta}}{2}^2 }.
\end{equation}
Note that $F_n\paren{Y,\eta}$ does not depend on the Cholesky factorization of $R_n\paren{\eta}$ and regardless of the choice of covariance function, it can be evaluated in $\cc{O}\paren{n^2}$ operations, even for irregularly spaced samples which is an improvement over the conventional likelihood function. The optimization algorithm in \eqref{InvFreeOptProbEtaForm} can be reformulated as 
\begin{equation}\label{InvFreeOptProb}
\paren{\hat{\phi}_n,\hat{\theta}_n} = \argmax_{\paren{\phi,\theta}\in\cc{I}\times \Theta} F_n\paren{Y,\phi,\theta},\quad\mbox{where}\quad F_n\paren{Y,\phi,\theta}\coloneqq \frac{1}{n}\set{\phi Y^\top K_n\paren{\theta}Y -\frac{\phi^2}{2}\LpNorm{K_n\paren{\theta}}{2}^2 }.
\end{equation}
Despite the fact that $F_n\paren{Y,\phi,\theta}$ has a simple quadratic (concave) form of $\phi$, its dependence to $\theta$ is fairly complicated . For instance $F_n$ is not a concave function of $\theta$ even for the classic case of isotropic exponential covariance. So, accurate approximation of its global maximizer can be computationally expensive. 

\begin{rem}\label{Rem2.1}
We conclude this section mentioning two characteristics of $F_n\paren{Y,\phi,\theta}$ that can provide a theoretical clue for generalizing ACS's loss function to a broader class of inversion-free losses. The first property is also critical for the theoretical analysis in the next section.
\begin{enumerate}[leftmargin=*]
\item As stated in \cite{M.Anitescu}, the true parameter $\eta_0$ is a stationary point of the expected value of $F_n\paren{Y,\eta}$. That is, 
\begin{equation*}
{\rm E} \set{\frac{\partial}{\partial\eta_j}F_n\paren{Y,\eta}\mid_{\eta=\eta_0}} = 0,\quad\forall\; j=1,\ldots,\paren{m+1}.
\end{equation*}
Roughly speaking, $\eta_0$ is located in a small neighborhood of a stationary point of \eqref{InvFreeOptProbEtaForm} if the gradient of $F_n\paren{Y,\eta}$ is smooth enough and concentrated around its expected value. The next fact, which has not been stated in \cite{M.Anitescu}, reveals a profound connection of \eqref{InvFreeOptProbEtaForm} to MLE.
\item For brevity, define $H_n\paren{Y,\eta}\coloneqq F_n\paren{Y,\eta}-F_n\paren{Y,\eta_0}$ and $L_n\paren{\eta,\eta_0} \coloneqq R^{1/2}_n\paren{\eta_0}R^{-1}_n\paren{\eta}R^{1/2}_n\paren{\eta_0}$. Also, let $\tilde{\eta}_n$ denotes the MLE of $\eta$. Obvious calculations lead to
\begin{align*}
&\hat{\eta}_n = \argmax_{\eta\in\Omega}H_n\paren{Y,\eta},\\
&\tilde{\eta}_n = \argmax_{\eta\in\Omega} H'_n\paren{Y,\eta}, \quad\mbox{where}\quad H'_n\paren{Y,\eta}\coloneqq\frac{1}{n}\Bigset{-\log\det L_n\paren{\eta,\eta_0} - n + Y^\top R^{-1}_n\paren{\eta} Y}.
\end{align*}
Notice that ${\rm E} H_n\paren{Y,\eta_0} = {\rm E} H'_n\paren{Y,\eta_0} = 0$. Under Assumption \ref{ObsLocAssu} and using similar techniques as Appendix \ref{AuxRes}, one can guarantee the existence of a universal scalar $C\in\paren{0,\infty}$ such that 
\begin{equation}\label{Minorizer}
{\rm E} H_n\paren{Y,\eta} \leq C\;{\rm E} H'_n\paren{Y,\eta},\quad\forall \eta\in\Omega.
\end{equation}
Namely, in the increasing domain regime, the objective function proposed in \cite{M.Anitescu} can be viewed as an approximate \emph{minorizing surrogate} of the likelihood function in the expected value sense (it forms a perfect minorizer whenever $C = 1$ in \eqref{Minorizer}).
\end{enumerate}
\end{rem}

\section{Main results}\label{Analysis}

We establish the asymptotic characteristics of the estimation algorithm in \eqref{InvFreeOptProbEtaForm}. Section \ref{ConvRate} examines the consistency of the global maximizer and the stationary points of \eqref{InvFreeOptProbEtaForm} under some sufficient conditions on $\Omega$ and the correlation function $K\paren{\cdot,\theta}$. The near minimax optimality and the asymptotic normality of the stationary points will be covered in Section \ref{MinMaxOptimandAsympNorm} 

\subsection{Consistency and the convergence rate}\label{ConvRate}

The following assumptions are assumed on the parameter space $\Omega = \cc{I}\times\Theta$ and the correlation function $K\paren{\cdot,\theta}$ for studying the asymptotic behaviour of the global maximizer of \eqref{InvFreeOptProbEtaForm}. Similar but slightly stronger conditions have been used in \cite{F.Bachoc} for the increasing domain asymptotic analysis of MLE.

\begin{assu}\label{Covar&ParamSpAssu}
The following conditions are satisfied by $\Omega$ and $K$.
\begin{enumerate}[label=(A\arabic*),leftmargin=*]
\item $\Theta$ and $\cc{I}$ are \emph{compact connected} subsets of $\bb{R}^m$ and $\paren{0,\infty}$, respectively. 
\item There are bounded scalars $M>0$ and $r_1>1$ such that for any $s\in\cc{D}_n$,
\begin{equation}\label{IdentifCond}
\max_{s'\in\cc{D}_n\paren{s,r_1} }\abs{K\paren{s'-s,\theta_2}-K\paren{s'-s,\theta_1}} \geq M\LpNorm{\theta_2-\theta_1}{2},\quad\forall\;\theta_1,\theta_2\in\Theta.
\end{equation}
\item For some $q\in\set{1,2,3}$, there exists a positive scalar $C_{K,\Theta}$ such that
\begin{equation*}
\max_{\theta\in\Theta}\paren{\abs{K\paren{s,\theta}} \vee \abs{\frac{\partial }{\partial\theta_{j_1}}\ldots\frac{\partial }{\partial\theta_{j_q}} K\paren{s,\theta}}}\leq \frac{C_{K,\Theta}}{1+\LpNorm{s}{2}^{d+1}},\quad\forall\;s\in\bb{R}^m,
\end{equation*}
for any $j_1,\ldots,j_q\in\set{1,\ldots,m}$.
\end{enumerate}
\end{assu}

Condition $\paren{A2}$, assuring the \emph{identifiability} of $\theta$ from the $K$, holds for the standard class of correlation functions such as Matern, powered exponential and rational quadratic. A detailed look at $\paren{A2}$ is postponed to the end of this section. Before commenting on $\paren{A3}$, let us define the family of geometric anisotropic covariance functions.

\begin{defn}\label{GeoAnisGPDefn}
Let $\ff{G}:\bb{R}^d\mapsto\bb{R}$ be a zero mean stationary Gaussian process in $\bb{R}^d$. Then $\ff{G}$ is called \emph{geometric anisotropic} if
\begin{equation}\label{GeoAnis}
R\paren{s-s',\eta_0} \coloneqq {\rm E} \ff{G}\paren{s}\ff{G}\paren{s'} = \phi_0 K\paren{\sqrt{\paren{s-s'}^\top A_0\paren{s-s'}}},\quad \forall\;s,s'\in\bb{R}^d,
\end{equation}
for $\phi_0 > 0$, \emph{symmetric positive definite} matrix $A_0\in\bb{R}^{d\times d}$, $\eta_0 = \paren{\phi_0, A_0}$ and a correlation function $K$. Specifically if $A_0 = \theta^{-1}_0I_d$ for some strictly positive $\theta_0$, then $\ff{G}$ is said to be an \emph{isotropic} Gaussian process. 
\end{defn}

For geometric anisotropic processes, $K$ is either assumed to be a fully known function (in this case $\eta_0 = \paren{\phi_0, A_0}$ in which $\phi_0\in\cc{I}$ and $A_0\in\Theta$, denotes the unknown parameters of covariance function) or known up to some strictly positive scalar $\nu_0$, usually refers to as the \emph{fractal index}. In the latter case, $\eta_0 = \set{\phi_0, \theta_0 = \paren{A_0,\nu_0}}$. Now, we mention some commonly used class of covariance functions, with unknown fractal index, satisfying $\paren{A3}$ with $q=1$ (appearing in the statement of the first main result in this section). It is supposed in the following Remark that 
\begin{equation}\label{EigValsUppLwBnd}
\Lambda_{\min,\Theta}\leq \min_{A_0\in\Theta} \frac{1}{\OpNorm{A^{-1}_0}{2}{2}} \leq \max_{A_0\in\Theta} \OpNorm{A_0}{2}{2} \leq \Lambda_{\max,\Theta},
\end{equation}
for strictly positive and bounded scalars $\Lambda_{\min,\Theta}$ and  $\Lambda_{\max,\Theta}$. Namely, all eigenvalues of $A_0$ are universally bounded away from zero and infinity.

\begin{rem}\label{Rem3.1}
Any compactly supported correlation function, such as \emph{spherical or Wendland family} \cite{H.Wendland} on $\bb{R}^d$ trivially admits $\paren{A3}$. Assumption $\paren{A3}$ with $q=1$ also holds for some classical families of geometric anisotropic covariances such as:
\begin{enumerate}[label=(\alph*),leftmargin=*]
\item \emph{Matern:} The Gaussian process $\ff{G}$ has Matern covariance function if it fulfills \eqref{GeoAnis} with
\begin{equation}\label{MaternCovFunc}
K\paren{r} = \frac{2^{1-\nu_0}}{\Gamma\paren{\nu_0}}r^{\nu_0}\ff{K}_{\nu_0}\paren{r},
\end{equation}
in which $\nu_0$ is an unknown, strictly positive scalar lies in a compact space. Moreover, $\Gamma\paren{\cdot}$ and $\ff{K}_{\nu_0}\paren{\cdot}$ represent the Gamma function and the modified Bessel function of the second kind, respectively. The parametric Matern family satisfies $\paren{A3}$ provided condition \eqref{EigValsUppLwBnd}.
\item \emph{Powered exponential:} A covariance function in this class satisfies \eqref{GeoAnis} with $K\paren{r} = e^{-r^{\nu_0}}$ and $\nu_0\in\paren{0,2}$. Like Matern class, assuming \eqref{EigValsUppLwBnd}, any member of powered exponential family fulfills $\paren{A3}$ with $q=1$.
\item \emph{Rational quadratic:} The elements of this class are of the form \eqref{GeoAnis} with $K\paren{r} = \paren{1+r^2}^{-\paren{\frac{d}{2}+\nu_0}}$ and $\nu_0>0$. For the case of known fractal index, $\paren{A3}$ with $q=1$ is valid, if \eqref{EigValsUppLwBnd} holds. Note that for unknown $\nu_0$ the exact same statement is satisfied under a slightly stronger condition of $\nu_0 > 1/2$  
\end{enumerate}
\end{rem}

Parts $\paren{b}$ and $\paren{c}$ of Remark \ref{Rem3.1} are verifiable by straightforward algebra and differentiation rules. In order to demonstrate part $\paren{a}$, see \cite{M.Abramowitz} for the derivative properties of the Bessel function of the second kind (with respect to the entries of $A_0$) and see Lemma \ref{MaternDecayDiffwrtNu} for the asymptotic behaviour of the partial derivatives of the Matern covariance with respect to $\nu_0$. Now, we state the first significant result of this section regarding the consistency of the global maximizer of \eqref{InvFreeOptProb} under Assumption \ref{Covar&ParamSpAssu} and perturbed regular lattice sampling. 

\begin{thm}\label{GlobMinRate}
Suppose that Assumptions \ref{ObsLocAssu} and \ref{Covar&ParamSpAssu} with $q=1$ hold for $\cc{D}_n$, $\Omega$ and $K$. Then the maximizer of \eqref{InvFreeOptProb} satisfies 
	
\begin{equation}\label{ConsistencyGlobMaxEq}
\Pr\paren{ \LpNorm{\hat{\theta}_n-\theta_0}{2}\vee \abs{\frac{\hat{\phi}_n}{\phi_0}-1} \geq C\sqrt{\frac{\ln n}{n}}}\rightarrow 0,\quad\mbox{as}\; n\rightarrow\infty,
\end{equation}
for some constant $C$ (which depends on $\cc{D}_n$, $\Omega$ and $K$).
\end{thm}

\begin{rem}
Let $\phi_{\min}$ and $\phi_{\max}$ denote the smallest and largest element in $\cc{I}$. Obviously, $\phi_{\min}$ and $\phi_{\max}$ are well defined and finite due to $\paren{A1}$. Moreover,
\begin{equation*}
\paren{1 \wedge \phi_{\min}}\paren{\LpNorm{\hat{\theta}_n-\theta_0}{2}\vee \abs{\frac{\hat{\phi}_n}{\phi_0}-1}}\leq\LpNorm{\hat{\eta}_n-\eta_0}{2}\leq \sqrt{1+\phi^2_{\max}}\paren{\LpNorm{\hat{\theta}_n-\theta_0}{2}\vee \abs{\frac{\hat{\phi}_n}{\phi_0}-1}}.
\end{equation*}
Thus \eqref{ConsistencyGlobMaxEq} is a stronger statement than $\LpNorm{\hat{\eta}_n-\eta_0}{2} = \cc{O}_{\bb{P}}\paren{\sqrt{n^{-1}\ln n}}$ and they are equivalent when $\phi_{\min} > 0$ (which is true under $A1$). 
\end{rem}

An analogous consistency result has been proved recently by Bachoc \cite{F.Bachoc} for the MLE and cross validation estimator. Based upon Theorem \ref{GlobMinRate}, the asymptotic rate of ACS's algorithm has not been sacrificed for increasing the speed and memory efficiency comparing to the MLE. 

Finally, we concisely address the role of the identifiability condition $\paren{A2}$ in Theorem \ref{GlobMinRate}. Actually, $\paren{A2}$ plays a decisive role in translating consistent estimation of the correlation matrix (in the relative sense) to $\eta_0$. Strictly speaking, $\paren{A2}$ is required to deduce \eqref{ConsistencyGlobMaxEq} from the probabilistic statement 
\begin{equation*}
\frac{1}{\sqrt{n}}\LpNorm{K_n(\hat{\theta}_n)-K_n\paren{\theta_0}}{2} = \cc{O}_{\bb{P}}\paren{\sqrt{n^{-1}\ln n}}.
\end{equation*}

The rest of this section is devoted to the analysis of the stationary points of \eqref{InvFreeOptProb}. Solving the unique root of the derivative of $F_n\paren{Y,\phi,\theta}$ with respect to $\phi$, yields a closed form formula for $\hat{\phi}_n$ in terms of $\hat{\theta}_n$, $Y$ and the correlation function, namely
\begin{equation}\label{ClosedFormFormulaVar}
\hat{\phi}_n = \frac{Y^\top K_n(\hat{\theta}_n) Y}{ \LpNorm{K_n(\hat{\theta}_n)}{2}^2 }.
\end{equation}
Moreover, $\hat{\theta}_n$ can be obtained using
\begin{equation}\label{ObjFuncTheta}
\hat{\theta}_n = \argmax_{\theta\in\Theta}\; G_n\paren{Y,\theta},\quad\mbox{where}\quad G_n\paren{Y,\theta} = \frac{Y^\top K_n\paren{\theta} Y}{ \LpNorm{K_n\paren{\theta}}{2} }.
\end{equation}
Note that for large $n$, computing the global maximizer of \eqref{ObjFuncTheta} can be less intensive than \eqref{InvFreeOptProbEtaForm} due to searching over a smaller space $\Theta$. We first visually assess the key properties of $F_n$ in some simple scenarios. In Figure \ref{Fig:Fig1}, $G_n\paren{Y,\theta}$ (which is a univariate function of scalar $\theta$) has been plotted versus $\theta$ for the two dimensional ($d = 2$) isotropic Matern covariance function in two different scenarios. In the left panel the isotropic Gaussian process $\ff{G}$ has been generated with the parameters $\paren{\phi_0,\theta_0,\nu_0} = \paren{1,4,0.5}$ and has been sampled in a randomly perturbed regular lattice with $\delta = 0.2$ and of size $N = 100$. In the right panel, the covariance parameters are given by $\paren{\phi_0,\theta_0,\nu_0} = \paren{1,6,1.5}$ and the GP is sampled at a randomly generated perturbed regular lattice with $N = 100$ and $\delta = 0.2$. As is apparent from Figure \ref{Fig:Fig1}, for these two parsimonious scenarios $G_n\paren{Y,\theta}$ is not a concave function of $\theta$ and has a single \emph{inflection point}. However, $G_n\paren{Y,\theta}$ has only one stationary point which coincides with its global maximizer. In the following, we rigorously study the large sample behaviour of the stationary points of $G_n\paren{Y,\theta}$ (as well as $F_n$) in a generic case. We initiate our analysis by stating the sufficient conditions on $K$ and $\Omega$.

\begin{figure}
\centering
\setlength\figureheight{5cm} 
\setlength\figurewidth{6.5cm} 
%
%
\begin{tikzpicture}

\pgfplotsset{every axis/.append style={
		scaled y ticks = false, 
		scaled x ticks = false, 
		y tick label style={/pgf/number format/.cd, fixed, fixed zerofill,
			,precision=1},
		x tick label style={/pgf/number format/.cd, fixed, fixed zerofill,
			precision=0}
	}
}

\begin{axis}[%
width=\figurewidth,
height=\figureheight,
scale only axis,
xmin=0,
xmax=8,
xlabel={$\theta$},
xmajorgrids,
ymin=0.9,
ymax=3.6,
ylabel={$\frac{G_n\paren{Y,\theta}}{\sqrt{n}}$},
ymajorgrids,
name=plot1,
title={$\theta_0 = 4, \nu_0 = 0.5, \delta = 0.3$}
]
\addplot [color=black,solid,line width=1pt,mark size=1.7pt,mark=o,mark options={solid},forget plot]
  table[row sep=crcr]{  0.2666667 0.9490951 \\
  	0.5333333 1.4580007 \\
  	0.8000000 2.0400403 \\
  	1.0666667 2.4862664 \\
  	1.3333333 2.7997902 \\
  	1.6000000 3.0183310 \\
  	1.8666667 3.1713383 \\
  	2.1333333 3.2784417 \\
  	2.4000000 3.3526541 \\
  	2.6666667 3.4028051 \\
  	2.9333333 3.4350491 \\
  	3.2000000 3.4537872 \\
  	3.4666667 3.4622389 \\
  	3.7333333 3.4628069 \\
  	4.0000000 3.4573156 \\
  	4.2666667 3.4471697 \\
  	4.5333333 3.4334639 \\
  	4.8000000 3.4170589 \\
  	5.0666667 3.3986364 \\
  	5.3333333 3.3787386 \\
  	5.6000000 3.3577985 \\
  	5.8666667 3.3361618 \\
  	6.1333333 3.3141046 \\
  	6.4000000 3.2918471 \\
  	6.6666667 3.2695635 \\
  	6.9333333 3.2473912 \\
  	7.2000000 3.2254375 \\
  	7.4666667 3.2037848 \\
  	7.7333333 3.1824956 \\
  	8.0000000 3.1616156 \\
};
\end{axis}

\pgfplotsset{every axis/.append style={
		scaled y ticks = false, 
		scaled x ticks = false, 
		y tick label style={/pgf/number format/.cd, fixed, fixed zerofill,
			,precision=1},
		x tick label style={/pgf/number format/.cd, fixed, fixed zerofill,
			precision=0}
	}
}

\begin{axis}[%
width=\figurewidth,
height=\figureheight,
scale only axis,
xmin=0,
xmax=12,
xlabel={$\theta$},
xmajorgrids,
ymin=2,
ymax=13,
ylabel={$\frac{G_n\paren{Y,\theta}}{\sqrt{n}}$},
ymajorgrids,
at=(plot1.right of south east),
anchor=left of south west,
title={$\theta_0 = 6, \nu_0 = 1.5, \delta = 0.2$}
]
\addplot [color=black,solid,line width=1.0pt,mark size=1.7pt,mark=o,mark options={solid},forget plot]
  table[row sep=crcr]{ 0.400000  2.288095 \\
0.800000  4.641630 \\
1.200000  6.511755 \\
1.600000  8.027198 \\
2.000000  9.232499 \\
2.400000 10.172364 \\
2.800000 10.890693 \\
3.200000 11.427524 \\
3.600000 11.817711 \\
4.000000 12.090713 \\
4.400000 12.270911 \\
4.800000 12.378159 \\
5.200000 12.428415 \\
5.600000 12.434372 \\
6.000000 12.406046 \\
6.400000 12.351292 \\
6.800000 12.276247 \\
7.200000 12.185701 \\
7.600000 12.083388 \\
8.000000 11.972227 \\
8.400000 11.854504 \\
8.800000 11.732019 \\
9.200000 11.606194 \\
9.600000 11.478159 \\
10.000000 11.348818 \\
10.400000 11.218895 \\
10.800000 11.088977 \\
11.200000 10.959538 \\
11.600000 10.830964 \\
12.000000 10.703571 \\
};
\end{axis}
\end{tikzpicture}%
\caption{The above figures exhibit $n^{-1/2}G_n\paren{Y,\theta}$ for the isotropic Matern covariance function (with known $\nu_0$). In the left panel, $\theta_0 = 4$, $\nu_0 = 0.5$ and the spatial samples form a two dimensional randomly perturbed regular lattice of size $N = 100$ with $\delta = 0.3$. In the right panel, $\theta_0 = 6$, $\nu_0 = 1.5$ and $\cc{D}_n$ is a randomly chosen two dimensional perturbed regular lattice with $N = 100$ and $\delta = 0.3$.}
\label{Fig:Fig1}
\end{figure}
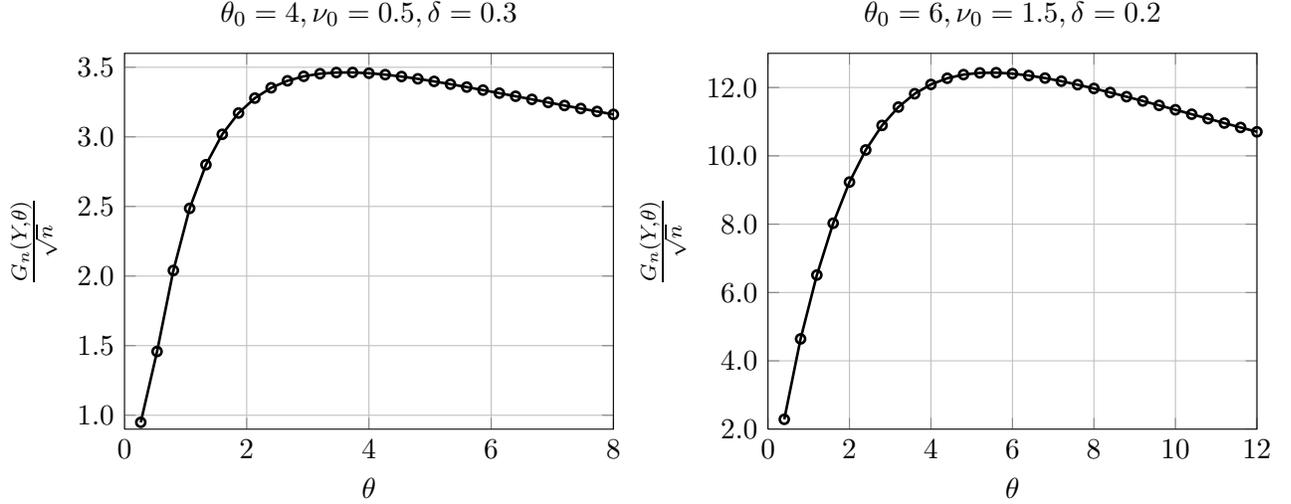

\begin{assu}\label{Covar&ParamSpAssu2}
$\paren{A1}$ holds for $\Omega$, and $K$ fulfills $\paren{A2}$ and $\paren{A3}$ with $q=2$ in Assumption \ref{Covar&ParamSpAssu}.
\end{assu}

\begin{rem}
The analysis of the stationary points of \eqref{InvFreeOptProb} requires a slightly stronger conditions than that of the global maximizer in Assumption \ref{Covar&ParamSpAssu}. The main distinction is the polynomial decay of the second order derivative of $K$ with respect to $\theta$. Note that the new condition on the second derivative of $K$ is not too restrictive. For instance the same analysis as Remark \ref{Rem3.1} validates this condition for all  covariance families introduced in Remark \ref{Rem3.1} (with a larger constant $C_{K,\Theta}$).
\end{rem}

\begin{thm}\label{LocMinRate}
Suppose that $\cc{D}_n$ admits Assumption \ref{ObsLocAssu} and Assumption \ref{Covar&ParamSpAssu2} holds for $\Omega$ and $K$. Then any stationary point of the optimization problem \eqref{InvFreeOptProb} satisfies
	
\begin{equation}\label{ConsistencyLocMaxEq}
\lim\limits_{n\rightarrow\infty}\Pr\paren{ \LpNorm{\hat{\theta}_n-\theta_0}{2}\vee \abs{\frac{\hat{\phi}_n}{\phi_0}-1} \geq C\sqrt{\frac{\ln n}{n}}}=0,
\end{equation}
for an appropriately chosen constant $C>0$ depending on $\cc{D}_n$, $\Omega$ and $K$.
\end{thm}

Theorem \ref{LocMinRate} shows that any stationary point of $F_n$ is concentrated in a small neighborhood of $\paren{\phi_0,\theta_0}$, with high probability. In other words, $F_n\paren{Y,\phi,\theta}$ shows a similar behaviour as Figure \ref{Fig:Fig1} in the general case. In addition, the comparison between \eqref{ConsistencyGlobMaxEq} and \eqref{ConsistencyLocMaxEq} reveals that stationary points converge to $\paren{\phi_0,\theta_0}$ with the same rate as the global maximizer.

We conclude this section by illustrating how restrictive the identifiability assumption $\paren{A2}$ can be for the frequently used classes of the covariance functions. We first introduce a slightly stronger identifiability condition than $\paren{A2}$, which will be referred to as $\paren{A4}$. Note that a slightly modified version of $\paren{A4}$ has been first introduced in \cite{F.Bachoc} for studying the increasing domain asymptotics of the maximum likelihood and cross validation algorithms. That is to say, these identifiability conditions are not exclusive to ACS method and pop up in the asymptotic analysis of other algorithms.

\begin{prop}\label{IdentifProp}
$\paren{A2}$ is satisfied, whenever
\begin{enumerate}[label=(A\arabic*),leftmargin=*,start=4]
\item $\paren{a}$ There are positive scalars $r_2 > 1$ and $M_2$ such that for any $\eta\in\Omega$ and $\lambda\in\cc{S}^{m}$, 
\begin{equation*}
\min_{s\in\cc{D}_n}\max_{s'\in\cc{D}_n\paren{s,r_2}}\abs{ \sum\limits_{j=1}^{m+1} \lambda_j \frac{\partial}{\partial \eta_j} R\paren{s-s',\eta} } \geq M_2.
\end{equation*}
$\paren{b}$ The following inequality holds for any distinct pair of points $\eta_1,\eta_2\in\Omega$ 
\begin{equation*}
\min_{s\in\cc{D}_n}\max_{s'\in\cc{D}_n\paren{s,r_2}} \abs{ R\paren{s-s',\eta_2} - R\paren{s-s',\eta_1} } > 0.
\end{equation*}
\end{enumerate}
\end{prop}

Clearly, $\paren{A4.b}$ is necessary for any algorithm consistently estimating $\eta$ and it can be verified for all typical classes of geometrical anisotropic covariance function. However, understanding the role and restrictiveness of $\paren{A4.a}$ is more subtle than that of $\paren{A4.b}$. 
Note that unlike $\paren{A3}$, all the introduced identifiability conditions not only depend on the choice of the covariance function but also to the observed locations $\cc{D}_n$. It may be excessive to seek the class of covariances satisfying $\paren{A4.a}$ for any perturbed lattice $\cc{D}_n$. So, a more pertinent question is: which class of covariance functions do almost surely satisfy $\paren{A4.a}$ for a randomly generated perturbed lattice? The following result responds to question by rigorously characterizing a broad subclass of the geometrically anisotropic covariances (as defined in Definition \ref{GeoAnisGPDefn}) fulfilling $\paren{A4.a}$. 

\begin{prop}\label{A5GeoAnisot}
Let $\ff{P}$ be a distribution in $\brac{-1,1}^d$ which is absolutely continuous with respect to the Lebesgue measure. Suppose that $R\paren{\cdot,\eta}:\bb{R}^d\mapsto\brapar{0,\infty}$ is a geometrically anisotropic covariance function with a known $\nu_0$ (if exists). Then, $\paren{A4}$ almost surely holds if
\begin{enumerate}[label = (\alph*)]
\item $K:\brapar{0,\infty}\mapsto\brapar{0,\infty}$ is a nonzero, differentiable and strictly decreasing function ($K$ may only have right derivative at zero).
\item $\cc{D}_n$ is a randomly generated $\delta-$perturbed lattice associated to $\ff{P}$. That is, $p_i$ are independent draws of $\ff{P}$ in Assumption \ref{ObsLocAssu}.
\end{enumerate}
\end{prop}

\begin{cor}
Let $\cc{D}_n$ be $d-$dimensional regular lattice (associated to $\delta=0$) and assume that $R\paren{\cdot,\eta}:\bb{R}^d\mapsto\brapar{0,\infty}$ is a geometrically anisotropic covariance function with known $\nu_0$. Then, $\paren{A4}$ holds if $K:\brapar{0,\infty}\mapsto\brapar{0,\infty}$ admits the condition $\paren{a}$ in Proposition \ref{A5GeoAnisot}.
\end{cor}

Although the conditions of Proposition \ref{A5GeoAnisot} trivially hold for the non-compactly supported covariance function introduced in Remark \ref{Rem3.1}, deploying analogous proof techniques can lead to a similar result for compactly supported covariance function.

\begin{prop}\label{A5GeoAnisot2}
Suppose that $\ff{P}$, $R\paren{\cdot,\eta}$ and $\cc{D}_n$ satisfy the same conditions as Proposition \ref{A5GeoAnisot}. Then, $\paren{A4}$ almost surely holds if there exists a large enough positive scalar $r_0$ for which
\begin{itemize}
\item $K:\brapar{0,\infty}\mapsto\brapar{0,\infty}$ is a nonzero, differentiable and strictly decreasing function in the interval $\brac{0,r_0}$ and $K\paren{r} = 0$ for any $r>r_0$.
\end{itemize}
\end{prop}

Although the required conditions on the covariance function's formulation, in Propositions \ref{A5GeoAnisot} and \ref{A5GeoAnisot2}, are very minimal, we assume that the fractal index $\nu_0$ (if exists) is fully known. However in the following result, $\nu_0$ is one of the unknown parameters to be estimated. Here the central emphasis is on the powered exponential and rational quadratic classes, as their partial derivative with respect to the fractal index have a somewhat simple closed form that can be handled without great difficulty. 

\begin{prop}\label{A5GeoAnisotUnknwnNu}
Let $\ff{P}$ be a distribution in $\brac{-1,1}^d$ which is absolutely continuous with respect to the Lebesgue measure. Suppose that $R\paren{\cdot,\eta}:\bb{R}^d\mapsto\brapar{0,\infty}$ is a geometrically anisotropic covariance function. Then, $\paren{A4}$ almost surely holds if
\begin{enumerate}[label = (\alph*)]
\item $K:\brapar{0,\infty}\mapsto\brapar{0,\infty}$ is either a powered exponential or rational quadratic covariance functions in Remark \ref{Rem3.1} with unknown $\nu_0 >0$.
\item $\cc{D}_n$ is a randomly generated $\delta-$perturbed lattice associated to $\ff{P}$. That is, $p_i$ are independent draws of $\ff{P}$ in Assumption \ref{ObsLocAssu}.
\end{enumerate}
\end{prop}

\begin{rem}
A prudent look at the proof of Proposition \ref{A5GeoAnisotUnknwnNu} reveals that the following property (which is satisfied by the powered exponential and rational quadratic families) has the crucial role.
\begin{equation}\label{CrucLocProperty}
\frac{\partial K}{\partial \nu} \Big\lvert_{\set{ r = \zero_d, \;\theta}} = 0,\quad\;\forall\; \theta\in\Theta.
\end{equation}
For the Matern class, not only $\frac{\partial K}{\partial \nu}$ not satisfy \eqref{CrucLocProperty}, it does not have a tractable algebraic form. We believe that $\paren{A4}$ holds true for the geometric anisotropic Matern family with unknown $\nu$, even though it is beyond the reach of our current proof technique.
\end{rem} 

\subsection{Minimax optimality and Asymptotic normality}\label{MinMaxOptimandAsympNorm}

Now, we further investigate the asymptotic statistical properties of ACS's algorithm. Near minimax optimality and asymptotic normality are respectively presented in Theorems \ref{MinMaxThm} and \ref{AsympDistThm}.

\begin{thm}\label{MinMaxThm}
Suppose that Assumptions \ref{ObsLocAssu} and \ref{Covar&ParamSpAssu} hold for $\cc{D}_n$, $\Omega$ and $K$. Then there exist $n_0\in\bb{N}$ and a bounded scalar $C>0$ such that
\begin{equation*}
\sup_{\eta_0\in\Omega} \Pr\paren{\LpNorm{\hat{\eta}_n-\eta_0}{2}\geq \frac{C}{\sqrt{n}}}\geq \frac{1}{4},
\end{equation*}
for any estimator $\hat{\eta}_n$ and any $n\geq n_0$.
\end{thm}

Theorem \ref{MinMaxThm} reveals that the established bounds in Theorems \ref{GlobMinRate} and \ref{LocMinRate} are sharp up to order $\sqrt{\ln n}$. This means that for the perturbed regular lattice sampling scheme, no algorithm can achieve a significantly better rate than the one considered in this paper.

\begin{thm}\label{AsympDistThm}
Suppose that $\cc{D}_n$ is a perturbed lattice introduced in Assumption \ref{ObsLocAssu}. Furthermore, $\paren{A1}$, $\paren{A3}$ with $q=3$ and $\paren{A4}$ are fulfilled by $\Omega$ and $R$. There is a positive definite matrix $\Sigma\in \bb{R}^{\paren{m+1}\times \paren{m+1}}$ with bounded operator norm such that
\begin{equation}\label{AsympNormEq}
\sqrt{n}\paren{\hat{\eta}_n-\eta_0} \cp{d} \cc{N}\paren{\zero_{m+1},\Sigma}. 
\end{equation}
\end{thm}

The exact formulation of $\Sigma$ has been omitted in this section due to its complicated algebraic form. We refer the reader to the proof of Theorem \ref{AsympDistThm} in Section \ref{Proofs} for further details. It is worthwhile to mention that the entries of $\Sigma$ heavily depend on the configuration of points in $\cc{D}_n$, which is a major disparity between fixed and increasing domain asymptotics. Comparing Theorems \ref{LocMinRate} and \ref{AsympDistThm}, here we impose a slightly stronger differentiability condition (polynomially decaying of the third derivative) for establishing asymptotic normality. This condition has been formerly introduced in \cite{F.Bachoc} and holds for the geometrically anisotropic covariances in Remark \ref{Rem3.1}.

\section{Simulation studies}\label{SimulRes}

The relatively large-scaled numerical studies in this section give a fairly comprehensive appraisal of the statistical and computational performance of the optimization problem \eqref{InvFreeOptProb}. Despite the popularity of \emph{R} language among the statisticians, running the iterative programs such as loops in R is much slower than that of C$++$ (around $250$ times slower according to some studies \cite{SB.Aruoba}). Taking advantage of the \emph{Rcpp} package and hybrid programming techniques in R can considerably expedite the execution time (up to $50$ times in our simulation studies). In order to get the maximum speed, the \emph{open MP} application programming interface has been used to exploit the multi-threaded programing technology. All the numerical experiments in this section have been executed on $12$ processors, except for the second simulation study ($n = 10^6$) which has been implemented on $60$ cores.

Generating high dimensional samples from a Gaussian process on an irregularly spaced grid is the foremost challenge that we confronted in our synthetic data simulations. Applying the traditional method based upon the Cholesky decomposition of the covariance matrix is almost infeasible in the case that $n \approx 10^5$ or larger. Hence, we use the considerably faster spectral method (pp. $203-205$, \cite{N.Cressie}) for generating stationary Gaussian processes. For completeness, this algorithm will be concisely presented here. Strictly speaking, the objective is to simulate a real valued zero mean stationary Gaussian process $\ff{G}$ in $\bb{R}^d$ with the covariance function $\phi_0K\paren{\cdot,\theta_0}$ over a $\delta-$perturbed lattice $\cc{D}_n = \set{s_1,\ldots,s_n}$. For the purpose of generating a realization of $\ff{G}$ on a perturbed grid, without loss of generality we can assume that the samples are all of unit variance, i.e., $\phi_0 = 1$. We also assume that $\ff{G}$ is geometric anisotropic. Recalling from Definition \ref{GeoAnis}, there is a symmetric positive definite matrix $B_0\in\bb{R}^{d\times d}$ which represents the symmetric square root of $A_0$, such that $K\paren{r,\theta_0} = K\paren{\LpNorm{B_0r}{2}}$. Throughout this section $d=2$ and $K$ is either the Matern or rational quadratic covariance function which have been previously introduced in Remark \ref{Rem3.1}.

Let $p\in\bb{N}$ be a large enough number and $\set{\xi_k}^p_{k=1}$ be i.i.d. uniform random variables on $\brac{-\pi,\pi}$. Let $\hat{K}:\bb{R}^d\mapsto\bb{R}$ denotes the spectral density of $\ff{G}$ defined by
\begin{equation*}
\hat{K}\paren{\omega} \coloneqq \paren{2\pi}^{-d}\int_{\bb{R}^d} K\paren{r,\theta_0} \cos\paren{\InnerProd{\omega}{r} } dr = \paren{2\pi}^{-d}\int_{\bb{R}^d} K\paren{\LpNorm{B_0r}{2}} \cos\paren{\InnerProd{\omega}{r} } dr.
\end{equation*}
The non-negative mapping $\hat{K}\paren{\cdot}$ is a density function in $\bb{R}^d$ (since it integrates to $K\paren{0,\theta_0} = 1$). Furthermore, let $\set{\omega_k}^N_{k=1}$ be independent draws from the density $\hat{K}\paren{\cdot}$. Now, define
\begin{equation}\label{GPApprox}
\ff{G}\paren{s} = \sqrt{\frac{2}{p}}\sum_{k=1}^{p} \cos\paren{\InnerProd{\omega_k}{s}+\xi_k},\quad\forall\; s\in\bb{R}^d.
\end{equation}
It is known that $\ff{G}$ is a geometric anisotropic process with $\cov\paren{\ff{G}\paren{s},\ff{G}\paren{s'}} = K\paren{\LpNorm{B_0\paren{s-s'}}{2}}$ for any pair $s,s'\in\bb{R}^d$ (p. $204$, \cite{N.Cressie}), converging in distribution to a Gaussian process as $p$ tends to infinity. Next, we explain how to generate the random variables $\set{\omega_k}^p_{k=1}$. The following fact which can be proved using the integration by substitution plays a principal role in our algorithm.

\begin{rem}
Let $\omega'\in\bb{R}^d$ be a draw from the following density function
\begin{equation*}
\hat{K}_I\paren{u} = \paren{2\pi}^{-d}\int_{\bb{R}^d} K\paren{\LpNorm{r}{2}} \cos\paren{\InnerProd{u}{r} } dr.
\end{equation*}
Then $\omega$ and $B_0\omega'$ have the same distribution, i.e., $\omega\eqd B_0\omega'$. Note that $\hat{K}_I$ is an isotropic function. Namely, there is a function $\Phi:\bb{R}\mapsto\brapar{0,\infty}$ for which $\hat{K}_I\paren{u} = \Phi\paren{\LpNorm{u}{2}}$. Moreover $\omega' \eqd \ff{r} \frac{\psi_d}{\LpNorm{\psi_d}{2}}$ in which $\psi_d$ is a standard $d-$ dimensional Gaussian vector and $\ff{r}$ is a non-negative random variable with the density function
\begin{equation*}
f_{\ff{r}}\paren{r} = \frac{d\pi^{d/2}}{\Gamma\paren{d/2+1}}r^{d-1}\Phi\paren{r} = \frac{2\pi^{d/2}}{\Gamma\paren{d/2}}r^{d-1}\Phi\paren{r}.
\end{equation*}
For instance in the case of $d=2$, we have $f_{\ff{r}}\paren{r} = 2\pi r \Phi\paren{r}$. Hence, $\omega\eqd \frac{\ff{r}}{\LpNorm{\psi_d}{2}}B_0\psi_d$.
\end{rem}

For Matern covariance function in two dimensional plane ($d=2$), generating independent samples of the random variable $\ff{r}$ is a straightforward task. In this case 
\begin{equation*}
f_{\ff{r}}\paren{r} = \frac{2\pi^{d/2}}{\Gamma\paren{d/2}}r^{d-1}\Phi\paren{r} = \frac{2\pi^{d/2}}{\Gamma\paren{d/2}}r^{d-1}\frac{\pi^{-d/2}\Gamma\paren{d/2+\nu} }{\Gamma\paren{\nu}}\paren{1+r^2}^{-\paren{\nu+d/2}} = \frac{2r\nu}{\paren{1+r^2}^{1+\nu}}.
\end{equation*}
Thus the cumulative distribution is of the form $\Pr\paren{\ff{r}\leq r} \coloneqq F_{\ff{r}}\paren{r} = 1-\paren{1+r^2}^{-\nu}$. So 
\begin{equation*}
\ff{r}\eqd F^{-1}_{\ff{r}}\paren{\ff{u}} = \sqrt{1-\paren{1-\ff{u}}^{-1/\nu}},
\end{equation*}
in which $\ff{u}$ is a uniform random variable in $\brac{0,1}$. One can find a closed from expression for $\ff{r}$ in terms of $\ff{u}$ for the rational quadratic covariance function, in the case that $\paren{\tau + \frac{1}{2}} \in\bb{N}$ (Recall $\tau$ from Remark \ref{Rem3.1}). In this case, $\Phi\paren{\cdot}$ has a form of the Matern covariance function \eqref{MaternCovFunc} (with different constants) due to the duality principle of the Fourier transform.

Throughout this section, $\ff{G}$ is assumed to be a zero mean Gaussian process in $\bb{R}^2$, whose covariance function is a member of either Matern or rational quadratic families with a known fractal index. In the first experiment, $\ff{G}$ is an isotropic spatial process. In other words, we set $A_0 = \theta^{-2}_0 I_2$ in the Definition \ref{GeoAnis}. $\theta_0$ is a strictly positive scalar known as the \emph{range parameter}. Furthermore, $\cc{D}_n$ is a randomly generated $\delta-$perturbed lattice of size $320^2$, i.e., $n = 102400 \approx 10^5$. The approximated realizations of $\ff{G}$ are generated using \eqref{GPApprox} with $p = 1.5\times 10^5$. To investigate the role of spatial irregularity in the computational and statistical performance of ACS's algorithm, we vary $\delta$ in the set $\set{0.1,0.3}$. The range parameter and the standard deviation, which is represented by $\sigma_0 = \sqrt{\phi_0}$, are respectively estimated solving the optimization problem \eqref{ObjFuncTheta} and closed form formula \eqref{ClosedFormFormulaVar}. The range parameter space is chosen as $\Theta=\brac{0.1,15}$. The single variable constrained optimization problem \eqref{ObjFuncTheta} is solved using the \emph{optimize} function in \emph{R} ,which exploits a combination of golden section search and successive parabolic interpolation. We stop the iteration of the solver when the relative change in the objective is below $10^{-3}$. Table \ref{Table1} displays the summary of our first simulation study.

\vspace{4mm}
\begin{adjustbox}{width=0.95\columnwidth}
\begin{tabular}{cc|c|c|c|c|}
\cline{3-6}
& & \multicolumn{2}{c|}{$\delta = 0.1$} & \multicolumn{2}{c|}{$\delta = 0.3$} \\ \hline
\multicolumn{1}{|c|}{\multirow{6}{*}{Matern} } &
\multicolumn{1}{|c|}{\multirow{2}{*}{$\nu_0 = 0.5$} } &
$\eta_0 = \paren{1,4}$ & $\eta_0 = \paren{1,7}$ & $\eta_0 = \paren{1,4}$ & $\eta_0 = \paren{1,7}$ \\
\multicolumn{1}{|c|}{}  
& & $\hat{\eta} = \paren{0.993,4.420}$ & $\hat{\eta} = \paren{0.978,7.565}$ & $\hat{\eta} = \paren{1.005,3.941}$ & $\hat{\eta} = \paren{1.028,6.553}$ \\ 
\cline{2-6} 
\multicolumn{1}{|c|}{} & \multicolumn{1}{|c|}{\multirow{2}{*}{$\nu_0 = 1.5$} } &
$\eta_0 = \paren{1,4}$ & $\eta_0 = \paren{1,7}$ & $\eta_0 = \paren{1,4}$ & $\eta_0 = \paren{1,7}$ \\
\multicolumn{1}{|c|}{} 
& & $\hat{\eta} = \paren{0.973,3.618}$ & $\hat{\eta} = \paren{1.023,6.568}$ & $\hat{\eta} = \paren{1.026,4.343}$ & $\hat{\eta} = \paren{1.005,7.102}$ \\ 
\cline{2-6}
\multicolumn{1}{|c|}{} & \multicolumn{1}{|c|}{\multirow{2}{*}{$\nu_0 = 2.5$} } &
$\eta_0 = \paren{1,4}$ & $\eta_0 = \paren{1,7}$ & $\eta_0 = \paren{1,4}$ & $\eta_0 = \paren{1,7}$ \\
\multicolumn{1}{|c|}{} 
& & $\hat{\eta} = \paren{1.014,4.186}$ & $\hat{\eta} = \paren{1.010,7.428}$ & $\hat{\eta} = \paren{1.044,4.037}$ & $\hat{\eta} = \paren{0.993,6.505}$ \\ \hline
\multicolumn{1}{|c|}{\multirow{4}{*}{Rational quadratic} } &
\multicolumn{1}{|c|}{\multirow{2}{*}{$\nu_0 = 0.5$} } &
$\eta_0 = \paren{1,4}$ & $\eta_0 = \paren{1,7}$ & $\eta_0 = \paren{1,4}$ & $\eta_0 = \paren{1,7}$ \\
\multicolumn{1}{|c|}{}  
& & $\hat{\eta} = \paren{1.017,4.100}$ & $\hat{\eta} = \paren{0.976,7.415}$ & $\hat{\eta} = \paren{1.013,3.916}$ & $\hat{\eta} = \paren{1.001,6.639}$ \\ 
\cline{2-6}
\multicolumn{1}{|c|}{} & \multicolumn{1}{|c|}{\multirow{2}{*}{$\nu_0 = 1.5$} } &
$\eta_0 = \paren{1,4}$ & $\eta_0 = \paren{1,7}$ & $\eta_0 = \paren{1,4}$ & $\eta_0 = \paren{1,7}$ \\
\multicolumn{1}{|c|}{} 
& & $\hat{\eta} = \paren{1.000,4.001}$ & $\hat{\eta} = \paren{1.014,7.210}$ & $\hat{\eta} = \paren{1.017,3.920}$ & $\hat{\eta} = \paren{0.994,7.063}$ \\ \hline
\end{tabular}
\end{adjustbox}
\captionof{table}{Estimation of $\eta_0 = \paren{\sigma_0,\theta_0}$ for the isotropic Matern and rational quadratic covariance functions, where $\cc{D}_n$ is a perturbed lattice of size $320^2$ with associated $\delta\in\set{0.1,0.3}$.}
\label{Table1}
\vspace{4mm}
The required CPU times for the numerical experiments in Table \ref{Table1} are approximately $30$ and $60$ minutes for the rational quadratic and Matern kernels, respectively. However evaluating the full MLE for a such large sample size is intractable. As is apparent from Table \ref{Table1}, the estimated parameters, $\hat{\eta}$, are in a close neighborhood of $\eta_0$. Moreover, estimating $\sigma_0$ has a significantly higher precision than that of the range parameters, since as the distant samples in $\cc{D}_n$ carry negligible information about $\theta_0$. Lastly, the condition number of the covariance matrix increases with the value of range parameter, leading to a higher estimation error $\LpNorm{\eta_0-\hat{\eta}}{2}$ for larger $\theta_0$. In the second simulation study which has the same set up as the first experiment, $\cc{D}_n$ is a irregular grid of size $1000^2 = 10^6$. We also set $p = 5\times 10^5$ in \eqref{GPApprox}. Table \ref{Table2} encapsulates the results of this experiment. The evaluation of $\hat{\eta}$ for this very high dimensional numerical study takes $8$ hours on $60$ cores with $4GB$ RAM. 
\vspace{1mm}
\begin{center}
\begin{tabular}{c|c|c|}
\cline{2-3}
& $\delta = 0.1,\; \nu_0 = 0.5$ & $\delta = 0.3,\; \nu_0 = 1.5$ \\ \hline
\multicolumn{1}{|c|}{\multirow{2}{*}{Rational quadratic} } &
$\eta_0 = \paren{1,4}$ & $\eta_0 = \paren{1,7}$\\
\multicolumn{1}{|c|}{}
& $\hat{\eta} = \paren{1.004,4.053}$ & $\hat{\eta} = \paren{0.999,6.948}$ \\ \hline
\end{tabular}
\captionof{table}{Estimation of $\eta_0 = \paren{\sigma_0,\theta_0}$ for the isotropic rational quadratic covariance functions, where $\cc{D}_n$ is a perturbed lattice of size $1000^2$ with associated $\delta\in\set{0.1,0.3}$.}
\label{Table2}
\end{center}

In the next set of experiments featuring Gaussian processes with isotropic covariance functions, we set $\cc{D}_n$ to be a two dimensional perturbed lattice of size $100^2$. For such a scenario, $\hat{\eta}$ can be estimated in a few minutes. Thus, we simulated $T = 100$ independent realizations and $\eta_0$ is estimated using the same procedure as previous studies, for each realization. The mean and root mean squared error (RMSE) have been computed across $T$ experiments. Table \ref{Table3} displays the average and RMSE for the standard deviation and range parameters for different values of $\eta$, $\delta$ and covariance kernels. The instances for which $\hat{\eta}$ hits the boundary points of $\Theta$ have been excluded in the procedure of calculating the mean and RMSE of the estimates.

Looking at the left and right panels of the Table \ref{Table3} reveals that the RMSE of $\hat{\sigma}$ and $\hat{\theta}$ are slightly larger for the higher values of $\delta$. Moreover, as we have discussed before, the RMSE and the average norm of $\paren{\hat{\eta}-\eta_0}$ directly depends on the range parameter. It is immediately clear that there is a considerable reduction in RMSE for rational quadratic kernel comparing to the Matern class. This observation may look surprising for the reader, as the condition number of the covariance matrix associated to Matern kernel is smaller than that of rational quadratic due to its faster decay. Thus at the first glance, it may not corroborate our developed theory regarding the consistency of ACS's estimation algorithm for covariance matrices with bounded condition number. However,  obtaining a highly accurate estimate of the dependence parameters are more difficult for a rapidly decaying covariance function, as more samples are almost independent. In the extreme case $\theta_0$ is unidentifiable if $K\paren{\cdot,\theta_0}$ is a compactly supported covariance function whose support size is strictly less than $\paren{1-2\delta}$ (In this case all the samples are independent).

\vspace{4mm}
\begin{adjustbox}{width=0.95\columnwidth}
\begin{tabular}{c|c|c|c|c|}
\cline{2-5}
& \multicolumn{2}{c|}{$\delta = 0.1$} & \multicolumn{2}{c|}{$\delta = 0.3$} \\ \hline
\multicolumn{1}{|c|}{\multirow{3}{*}{Matern covariance} } &
$\paren{\sigma_0,\theta_0} = \paren{1,4}$ & $\paren{\sigma_0,\theta_0} = \paren{1,7}$ & $\paren{\sigma_0,\theta_0} = \paren{1,4}$ & $\paren{\sigma_0,\theta_0} = \paren{1,7}$ \\
\multicolumn{1}{|c|}{}  
& $\hat{\theta} \pm \mbox{RSME} = 4.107 \pm 1.224$ & $\hat{\theta} \pm \mbox{RSME} = 7.259 \pm 2.462$ & $\hat{\theta} \pm \mbox{RSME} = 3.982 \pm 0.980$ & $\hat{\theta} \pm \mbox{RSME} = 6.814 \pm 2.233$ \\
\multicolumn{1}{|c|}{($\nu_0 = 0.5$)}  
& $\hat{\sigma} \pm \mbox{RSME} = 0.999 \pm 0.067$ & $\hat{\sigma} \pm \mbox{RSME} = 0.991 \pm 0.089$ & $\hat{\sigma} \pm \mbox{RSME} = 0.995 \pm 0.062$ & $\hat{\sigma} \pm \mbox{RSME} = 1.003 \pm 0.096$ \\ \hline
\multicolumn{1}{|c|}{\multirow{3}{*}{Matern covariance} } &
$\paren{\sigma_0,\theta_0} = \paren{1,4}$ & $\paren{\sigma_0,\theta_0} = \paren{1,7}$ & $\paren{\sigma_0,\theta_0} = \paren{1,4}$ & $\paren{\sigma_0,\theta_0} = \paren{1,7}$ \\
\multicolumn{1}{|c|}{}  
& $\hat{\theta} \pm \mbox{RSME} = 3.936 \pm 1.127$ & $\hat{\theta} \pm \mbox{RSME} = 6.588 \pm 2.060$ & $\hat{\theta} \pm \mbox{RSME} = 4.180 \pm 1.181$ & $\hat{\theta} \pm \mbox{RSME} = 6.519 \pm 2.127$ \\
\multicolumn{1}{|c|}{($\nu_0 = 1.5$)}  
& $\hat{\sigma} \pm \mbox{RSME} = 1.002 \pm 0.070$ & $\hat{\sigma} \pm \mbox{RSME} = 0.992 \pm 0.096$ & $\hat{\sigma} \pm \mbox{RSME} = 0.995 \pm 0.072$ & $\hat{\sigma} \pm \mbox{RSME} = 1.018 \pm 0.107$ \\ \hline
\multicolumn{1}{|c|}{\multirow{3}{*}{Rational quadratic} } &
$\paren{\sigma_0,\theta_0} = \paren{1,4}$ & $\paren{\sigma_0,\theta_0} = \paren{1,7}$ & $\paren{\sigma_0,\theta_0} = \paren{1,4}$ & $\paren{\sigma_0,\theta_0} = \paren{1,7}$ \\
\multicolumn{1}{|c|}{}  
& $\hat{\theta} \pm \mbox{RSME} = 3.889 \pm 0.599$ & $\hat{\theta} \pm \mbox{RSME} = 6.855 \pm 1.507$ & $\hat{\theta} \pm \mbox{RSME} = 4.032 \pm 0.647$ & $\hat{\theta} \pm \mbox{RSME} = 6.793 \pm 1.373$ \\
\multicolumn{1}{|c|}{covariance ($\nu_0 = 0.5$)}  
& $\hat{\sigma} \pm \mbox{RSME} = 1.002 \pm 0.062$ & $\hat{\sigma} \pm \mbox{RSME} = 0.986 \pm 0.082$ & $\hat{\sigma} \pm \mbox{RSME} = 0.992 \pm 0.046$ & $\hat{\sigma} \pm \mbox{RSME} = 0.990 \pm 0.069$ \\ \hline
\multicolumn{1}{|c|}{\multirow{3}{*}{Rational quadratic} } &
$\paren{\sigma_0,\theta_0} = \paren{1,4}$ & $\paren{\sigma_0,\theta_0} = \paren{1,7}$ & $\paren{\sigma_0,\theta_0} = \paren{1,4}$ & $\paren{\sigma_0,\theta_0} = \paren{1,7}$ \\
\multicolumn{1}{|c|}{}  
& $\hat{\theta} \pm \mbox{RSME} = 3.984 \pm 0.342$ & $\hat{\theta} \pm \mbox{RSME} = 7.160 \pm 1.010$ & $\hat{\theta} \pm \mbox{RSME} = 4.016 \pm 0.348$ & $\hat{\theta} \pm \mbox{RSME} = 7.127 \pm 1.116$ \\
\multicolumn{1}{|c|}{covariance ($\nu_0 = 1.5$)}  
& $\hat{\sigma} \pm \mbox{RSME} = 0.999 \pm 0.028$ & $\hat{\sigma} \pm \mbox{RSME} = 0.994 \pm 0.074$ & $\hat{\sigma} \pm \mbox{RSME} = 0.994 \pm 0.026$ & $\hat{\sigma} \pm \mbox{RSME} = 0.995 \pm 0.049$ \\ \hline
\end{tabular}
\end{adjustbox}
\captionof{table}{Mean and RMSE of $\hat{\eta}$ over $100$ independent experiments for the isotropic Matern and rational quadratic covariance functions, where $\cc{D}_n$ is a perturbed lattice of size $100^2$ with associated $\delta\in\set{0.1,0.3}$.}
\label{Table3}

\begin{figure}
\centering
\setlength\figureheight{5cm} 
\setlength\figurewidth{6.5cm} 
\subimport{/}{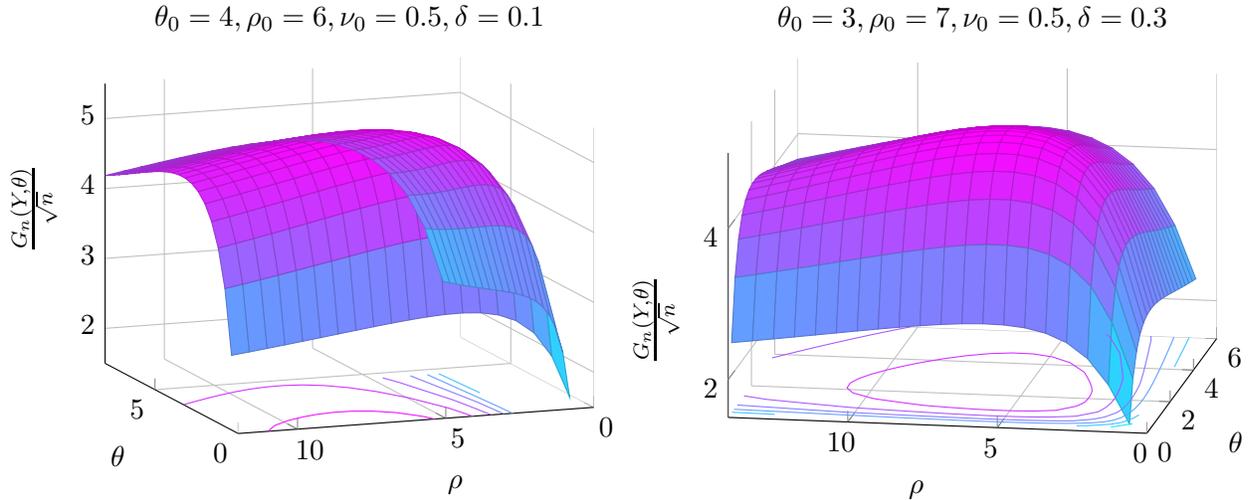}
\caption{The above figures exhibit $n^{-1/2}G_n\paren{Y,\theta}$ for geometric anisotropic Matern covariance function with $\nu_0 = 0.5$. The spatial samples form a two dimensional randomly $\delta-$perturbed regular lattice of size $N = 100$. In the left panel, $\paren{\theta_0,\rho_0} = \paren{4,6}$ and $\delta = 0.1$. In the right panel, $\paren{\theta_0,\rho_0} = \paren{3,7}$ and $\delta = 0.3$.}
\label{Fig:Fig2}
\end{figure}

\vspace{4mm}
Now we turn to investigate the precision and RMSE of estimation algorithm \eqref{InvFreeOptProb} for the geometric anisotropic covariance structure. Same as before, $\ff{G}$ is a zero mean stationary Gaussian process in $\bb{R}^2$ observed on a perturbed lattice of size $100^2$. $\ff{G}$ has a geometric anisotropic covariance kernel (Matern or rational quadratic) with 
\begin{equation*}
B_0 = \left(\begin{array}{cc}
\theta^{-1}_0 & 0\\
0 & \rho^{-1}_0
\end{array}\right)
, \quad\theta_0 = 4,\;\;\mbox{and}\;\; \rho_0 = 6.
\end{equation*}
The parameter space $\Theta = \set{\paren{\theta_0,\rho_0}\in\Theta}$ is a two dimensional box chosen as $\brac{0.1,15}^2$. The range parameters $\theta_0$ and $\rho_0$ are estimated by solving the optimization problem \eqref{ObjFuncTheta}. The Figure \ref{Fig:Fig2} exhibits the objective function in \eqref{ObjFuncTheta} and its contours for a Gaussian process with geometric anisotropic Matern covariance with $\nu_0 = 0.5$ which has been sampled on a $\delta-$perturbed regular grid. In the left and right panels of Figure \ref{Fig:Fig2}, the other parameters are respectively given by $\paren{\theta_0,\rho_0,\delta} = \paren{4,6,0.1}$ and $\paren{\theta_0,\rho_0,\delta} = \paren{3,7,0.3}$. It can be seen from Figure \ref{Fig:Fig2} that $G_n$ is a unimodal function with only one stationary point. We perform the maximization using the \emph{optim} function in \emph{R} and with \emph{L-BFGS-B} algorithm \cite{RH.Byrd} (box constrained BFGS). The maximum iteration and the initial guess of the \emph{L-BFGS-B} method are respectively $100$ and $\paren{2,2}$. The components of the gradient function are computed using the finite difference approximation with the step size of $10^{-3}$. We cease the iteration when the relative change in the objective function is below $10^{-5}$. The computation procedure of the average and RMSE of $\hat{\eta} = (\hat{\theta},\hat{\rho},\hat{\sigma})$ is exactly the same as the former simulation study. Table \ref{Table4} presents a summary of the final results of this simulation study. It is clear from Table \ref{Table4} that the RMSE for Matern covariance is significantly larger in comparison to rational quadratic class. Furthermore, increasing $\nu_0$ for each covariance kernel leads to a slightly larger RMSE.  

\vspace{4mm}
\begin{adjustbox}{width=0.85\columnwidth}
\begin{tabular}{c|c|c|}
\cline{2-3}
& $\delta = 0.1$ & $\delta = 0.3$ \\ \hline
\multicolumn{1}{|c|}{\multirow{4}{*}{Matern covariance ($\nu_0 = 0.5$)} } &
$\paren{\sigma_0,\rho_0,\theta_0} = \paren{1,6,4}$ & $\paren{\sigma_0,\rho_0,\theta_0} = \paren{1,6,4}$ \\
\multicolumn{1}{|c|}{}  
& $\hat{\sigma} \pm \mbox{RSME} = 0.988 \pm 0.096$ & $\hat{\sigma} \pm \mbox{RSME} = 0.993 \pm 0.097$ \\
\multicolumn{1}{|c|}{}  
& $\hat{\rho} \pm \mbox{RSME} = 6.042 \pm 1.885$ & $\hat{\rho} \pm \mbox{RSME} = 6.478 \pm 1.908$ \\ 
\multicolumn{1}{|c|}{}  
& $\hat{\theta} \pm \mbox{RSME} = 4.091 \pm 1.110$ & $\hat{\theta} \pm \mbox{RSME} = 4.038 \pm 1.272$ \\ \hline
\multicolumn{1}{|c|}{\multirow{4}{*}{Matern covariance ($\nu_0 = 1.5$)} } &
$\paren{\sigma_0,\rho_0,\theta_0} = \paren{1,6,4}$ & $\paren{\sigma_0,\rho_0,\theta_0} = \paren{1,6,4}$ \\
\multicolumn{1}{|c|}{}  
& $\hat{\sigma} \pm \mbox{RSME} = 0.993 \pm 0.108$ & $\hat{\sigma} \pm \mbox{RSME} = 0.984 \pm 0.104$ \\
\multicolumn{1}{|c|}{}  
& $\hat{\rho} \pm \mbox{RSME} = 05.965 \pm 1.981$ & $\hat{\rho} \pm \mbox{RSME} = 6.160 \pm 1.890$ \\ 
\multicolumn{1}{|c|}{}  
& $\hat{\theta} \pm \mbox{RSME} = 3.740 \pm 1.146$ & $\hat{\theta} \pm \mbox{RSME} = 3.970 \pm 1.243$ \\ \hline
\multicolumn{1}{|c|}{\multirow{4}{*}{Rational quadratic covariance ($\nu_0 = 0.5$)} } &
$\paren{\sigma_0,\rho_0,\theta_0} = \paren{1,6,4}$ & $\paren{\sigma_0,\rho_0,\theta_0} = \paren{1,6'4}$ \\
\multicolumn{1}{|c|}{}  
& $\hat{\sigma} \pm \mbox{RSME} = 0.992 \pm 0.071$ & $\hat{\sigma} \pm \mbox{RSME} = 0.989 \pm 0.076$ \\
\multicolumn{1}{|c|}{}  
& $\hat{\rho} \pm \mbox{RSME} = 5.978 \pm 1.241$ & $\hat{\rho} \pm \mbox{RSME} = 5.921 \pm 1.208$ \\ 
\multicolumn{1}{|c|}{}  
& $\hat{\theta} \pm \mbox{RSME} = 4.092 \pm 0.843$ & $\hat{\theta} \pm \mbox{RSME} = 4.037 \pm 1.064$ \\ \hline
\multicolumn{1}{|c|}{\multirow{4}{*}{Rational quadratic covariance ($\nu_0 = 1.5$)}} &
$\paren{\sigma_0,\rho_0,\theta_0} = \paren{1,6,4}$ & $\paren{\sigma_0,\rho_0,\theta_0} = \paren{1,6,4}$ \\
\multicolumn{1}{|c|}{}  
& $\hat{\sigma} \pm \mbox{RSME} = 0.996 \pm 0.036$ & $\hat{\sigma} \pm \mbox{RSME} = 0.998 \pm 0.036$ \\
\multicolumn{1}{|c|}{}  
& $\hat{\rho} \pm \mbox{RSME} = 6.116 \pm 0.821$ & $\hat{\rho} \pm \mbox{RSME} = 6.158 \pm 0.766$ \\ 
\multicolumn{1}{|c|}{}  
& $\hat{\theta} \pm \mbox{RSME} = 4.045 \pm 0.543$ & $\hat{\theta} \pm \mbox{RSME} = 4.150 \pm 0.524$ \\ \hline
\end{tabular}
\end{adjustbox}
\captionof{table}{Mean and RMSE of $\hat{\eta}$ over $100$ independent experiments for the geometric anisotropic Matern and rational quadratic covariance functions, where $\cc{D}_n$ is a perturbed lattice of size $100^2$ with associated $\delta\in\set{0.1,0.3}$.}
\label{Table4}

\section{Discussion}\label{Discus}

Investigation of the asymptotic properties of the non-likelihood based optimization algorithms for estimating covariance parameters has remained relatively intact. To our knowledge, this paper is the first asymptotic analysis of ACS's algorithm in the increasing domain regime. Notwithstanding the thorough study of the consistency, minimax optimality and asymptotic normality of the stationary points of ACS's loss function, there is much future work to be done to determine the computational and statistical strengths and weaknesses of this algorithm in either of the two frequently used asymptotic regimes. Here we mention a few among the many future directions which were beyond the scope of this paper.

\begin{enumerate}[label = (\alph*),leftmargin=*]
\item As indicated in Remark \ref{Rem2.1}, the inversion-free loss function can be viewed as an approximate minorizer for the likelihood loss (in the expected value sense) in the increasing domain setting. However, more work needs to be done to know how to precisely characterize a rich class of minorizers for the likelihood loss. We believe that responding to this question will provide a flexible class of fast and consistent estimators of covariance parameters.
\item Spatial statisticians usually cast doubt upon the benefits of increasing domain asymptotics as spatial processes are unlikely to be stationary over a large domain. However the developed theory in this manuscript has persuaded us that inversion free algorithm can consistently estimate microergodic covariance parameters, when applied on the preconditioned samples of a Gaussian random field in a fixed domain. In this regard, a modified version of the estimator studied in this paper can provide a powerful tool for interpolation of Gaussian processes.
\end{enumerate}

\section{Proofs of the main results}\label{Proofs}

We first introduce a few notation to simplify the algebra in the forthcoming sections. For any strictly positive scalar $r$ and any $\theta_0\in\Theta$, the ball of radius $r$ (with respect to the Euclidean norm) centered at $\theta_0$ and its complement are defined by
\begin{equation*}
\Theta_{\theta_0}\paren{r} \coloneqq \set{\theta\in\Theta:\; \LpNorm{\theta-\theta_0}{2}\leq r},\quad \Theta^c_{\theta_0}\paren{r}\coloneqq \Theta\setminus \Theta_{\theta_0}\paren{r}.
\end{equation*}
Furthermore, for any $\theta_1,\theta_2\in\Theta$ define
\begin{equation}\label{MandH}
M_{\theta_1,\theta_2}\coloneqq \frac{K^{1/2}_n\paren{\theta_1}K_n\paren{\theta_2}K^{1/2}_n\paren{\theta_1}}{\LpNorm{K_n\paren{\theta_2}}{2}^{2}} ,\quad H_{\theta_1,\theta_2}\coloneqq \LpNorm{K_n\paren{\theta_2}}{2}M_{\theta_1,\theta_2}.
\end{equation}

\begin{proof}[Proof of Theorem \ref{GlobMinRate}]

Our proof has two major parts. In the first part, the consistency of $\hat{\theta}_n$ (correlation function's parameters) will be substantiated. In the second part, we establish the consistency of $\hat{\phi}_n$, which has a closed form solution in terms of $Y$, correlation function and $\hat{\theta}_n$, by conditioning on the consistency of $\hat{\theta}_n$. To this end, various types of concentration inequalities regarding the quadratic forms (and their supremum over a bounded space) of Gaussian processes are of the indispensable importance. Such results will be presented in the Appendix \ref{AuxRes}.
	
Let $Z$ be a standard Gaussian vector in $\bb{R}^n$. As $Y$ and $\sqrt{\phi_0}K^{1/2}_n(\theta_0)Z$ have the same distribution, \eqref{InvFreeOptProb} can be equivalently written by 
\begin{equation}\label{InvFreeOptProbEqForm}
\paren{\hat{\phi}_n,\hat{\theta}_n} = \argmax_{\paren{\phi,\theta}\in\cc{I}\times \Theta} \set{\phi\phi_0 Z^\top K^{1/2}_n(\theta_0)K_n\paren{\theta}K^{1/2}_n(\theta_0)Z -\frac{\phi^2}{2}\LpNorm{K_n\paren{\theta}}{2}^2 }.
\end{equation}
The objective function in \eqref{InvFreeOptProbEqForm} is quadratic in terms of $\phi$ and its maximizer $\hat{\phi}_n$ has a simple closed form. Replacing $\hat{\phi}_n$ to \eqref{InvFreeOptProbEqForm} gives a surrogate form for $\hat{\theta}_n$. Omitting the cumbersome algebra, the final results are given by
\begin{equation}\label{OptValInvFreeOptProb}
\frac{\hat{\phi}_n}{\phi_0} = Z^\top M_{\theta_0,\hat{\theta}_n}Z,\quad\quad \hat{\theta}_n = \argmax_{\theta\in\Theta} Z^\top H_{\theta_0,\theta} Z.
\end{equation}
We first show (as Claim $1$) the consistency of $\hat{\theta}_n$, which is the supremum of a generalized chi-square random variable. The purpose of Claim $2$ is to find the estimation rate of $\phi_0$.
	
\begin{clawithinpf}\label{Claim1Thm1}
Choose $\xi>\paren{m-1}$ and let $r_n\coloneqq C_{\min}\sqrt{\frac{\ln n}{n}}$ for some bounded positive scalar $C_{\min}$ (See Lemma \ref{Hanson-WrightConcIneq} for its exact form). Then,
\begin{equation*}
\Pr\set{\hat{\theta}_n\in \Theta^c_{\theta_0}\paren{r_n}}\rightarrow 0,\quad\mbox{as}\; n\rightarrow \infty.
\end{equation*}
\end{clawithinpf} 
	
\begin{proof}[Proof of Claim \ref{Claim1Thm1}]
Consider the sequence $r'_n = n^{-1}\sqrt{\ln n},\;\forall n$. The boundedness of $\Theta$ guarantees the existence of some $R_0>0$ such that a ball of radius $R_0$ contains $\Theta$. So, the classical volume argument implies that
\begin{equation}\label{CovNumUppBnd}
\abs{\cc{N}_{r'_n}\paren{\Theta^c_{\theta_0}\paren{r_n}}}\leq \abs{\cc{N}_{r'_n}\paren{\Theta}}\lesssim \paren{\frac{R_0}{r'_n}}^m = o\paren{n^m}.
\end{equation}
So, there is $n_1\in\bb{N}$ such that $\cc{N}_{r'_n}\paren{\Theta^c_{\theta_0}\paren{r_n}}\leq n^m$ for any $n\geq n_1$. It follows from \eqref{OptValInvFreeOptProb} that
		
\begin{equation*}
\Pr\set{\hat{\theta}_n\in \Theta^c_{\theta_0}\paren{r_n}} \leq \RHS\coloneqq \Pr\paren{A_n} \coloneqq \Pr\paren{ Z^\top H_{\theta_0,\theta_0} Z\leq \sup_{\theta\in \Theta^c_{\theta_0}\paren{r_n}} Z^\top H_{\theta_0,\theta} Z }.
\end{equation*}
Thus, it suffices to control $\RHS$ from above. For a properly chosen positive scalar $C_2$, define the event $\pi_n$ by
		
\begin{equation*}
\pi_n \coloneqq \set{ \sup_{\theta\in \Theta^c_{\theta_0}\paren{r_n}} Z^\top H_{\theta_0,\theta} Z \leq \sup_{\theta\in \cc{N}_{r'_n}\paren{\Theta^c_{\theta_0}\paren{r_n}}} Z^\top H_{\theta_0,\theta} Z + C_2\sqrt{\frac{\ln n}{n}} }.
\end{equation*}
Notice that $r_n\sqrt{n} = \cc{O}\paren{\sqrt{n^{-1}\ln n}}$. According to Lemma \ref{DiscritizationLemma}, there is a bounded $C_2>0$ for which $\tau_n \coloneqq \Pr\paren{\pi^c_n}\rightarrow 0$. We refer the reader to Lemma \ref{DiscritizationLemma} for the closed form expression of $C_2$. An upper bound on $\RHS$ is obtained by conditioning $A_n$ on $\pi_n$. 
\begin{eqnarray}\label{ModalIneq2}
\RHS &=& \Pr\paren{A_n \cap \pi_n} + \tau_n \Pr\paren{A_n \mid \pi^c_n} \leq \tau_n + \Pr\paren{A_n \cap \pi_n}\nonumber\\
&\RelNum{\paren{A}}{\leq}& \tau_n + \Pr\paren{ Z^\top H_{\theta_0,\theta_0} Z\leq \sup_{\theta\in \cc{N}_{r'_n}\paren{\Theta^c_{\theta_0}\paren{r_n}}} Z^\top H_{\theta_0,\theta} Z + C_2\sqrt{\frac{\ln n}{n}}}\nonumber\\
&\RelNum{\paren{B}}{\leq}& \tau_n + n^m \sup_{\theta\in \cc{N}_{r'_n}\paren{\Theta^c_{\theta_0}\paren{r_n}}}\Pr\paren{Z^\top H_{\theta_0,\theta_0} Z\leq Z^\top H_{\theta_0,\theta} Z+C_2\sqrt{\frac{\ln n}{n}} }
\end{eqnarray}
The way that $\pi_n$ and $A_n$ have been defined trivially justifies inequality $\paren{A}$. Furthermore $\paren{B}$ is inferred from the combination of \eqref{CovNumUppBnd} and the union bound. Applying Lemma \ref{Hanson-WrightConcIneq} guarantees the following result for any $\theta\in \cc{N}_{r'_n}\paren{\Theta^c_{\theta_0}\paren{r_n}}$.
\begin{equation}\label{ModalIneq3}
\Pr\paren{Z^\top H_{\theta_0,\theta_0} Z\leq Z^\top H_{\theta_0,\theta} Z+C_2\sqrt{\frac{\ln n}{n}} } \leq n^{-\paren{1+\xi}}.
\end{equation}
Finally, substituting \eqref{ModalIneq3} into \eqref{ModalIneq2} yields that
\begin{equation*}
\RHS \leq \paren{ \tau_n + n^{m-\paren{1+\xi}} }  \rightarrow 0,\quad \mbox{as}\; n\rightarrow \infty.
\end{equation*}
\end{proof}
	
\begin{clawithinpf}\label{Claim2Thm1}
There exists a bounded scalar $C>0$, depending on $\cc{D}_n$, $K$ and $\Omega$, such that
\begin{equation*}
\pi'_n\coloneqq\Pr\paren{\abs{\frac{\hat{\phi}_n}{\phi_0}-1}\geq r''_n\coloneqq C\sqrt{\frac{\ln n}{n}}}\rightarrow 0,\quad \mbox{as}\;n\rightarrow\infty.
\end{equation*}
(Here $C = 2\paren{\ff{D}_{\max}C_{\min}+C'\Lambda_{\max}\sqrt{m}}$ in which $C'$ is a large enough positive universal constant and $C_{\min}$ has been defined in Claim \ref{Claim1Thm1}. $\Lambda_{\max}$ and $\ff{D}_{\max}$ are given in the Proposition \ref{UnifBndEigValCorrMatrix}.)
\end{clawithinpf}
	
\begin{proof}[Proof of Claim \ref{Claim2Thm1}]
Recall $r_n$ and $r'_n$ form Claim \ref{Claim1Thm1}. Obviously,
\begin{equation*}
\pi'_n \leq T_1+T_2 \coloneqq \Pr\set{\hat{\theta}_n\in \Theta^c_{\theta_0}\paren{r_n}} + \Pr\set{\paren{\abs{\frac{\hat{\phi}_n}{\phi_0}-1}\geq C\sqrt{\frac{\ln n}{n}}} \bigcap \paren{\hat{\theta}_n\in \Theta_{\theta_0}\paren{r_n}} }.
\end{equation*}
Since $T_1$ tends to zero (via Claim \ref{Claim1Thm1}), it suffices to show that $T_2$ is a diminishing sequence as $n\rightarrow\infty$. Let $\beta_{\hat{\theta}_n}$ be the closest point in $\cc{N}_{r'_n}\paren{\Theta_{\theta_0}\paren{r_n}}$ (which is a deterministic set) to $\hat{\theta}_n$. Based upon identity \eqref{OptValInvFreeOptProb}, we have
\begin{equation}\label{PhiHatnPhi0Ratio}
\abs{\frac{\hat{\phi}_n}{\phi_0}-1}= \abs{ Z^\top M_{\theta_0,\hat{\theta}_n}Z - 1}
\end{equation}
Given that $\hat{\theta}_n$ belongs to $\Theta_{\theta_0}\paren{r_n}$, applying the triangle inequality on the right hand side of the identity \eqref{PhiHatnPhi0Ratio} yields that, almost surely
	
\begin{eqnarray}\label{UppBndT2}
\abs{\frac{\hat{\phi}_n}{\phi_0}-1} &\leq& \abs{ Z^\top M_{\theta_0,\hat{\theta}_n}Z-Z^\top M_{\theta_0,\beta_{\hat{\theta}_n}}Z} + \abs{Z^\top M_{\theta_0,\beta_{\hat{\theta}_n}}Z -\tr\paren{M_{\theta_0,\beta_{\hat{\theta}_n}}} } + \abs{\tr\paren{M_{\theta_0,\beta_{\hat{\theta}_n}}}-1}\nonumber\\
&\RelNum{\paren{D}}{\leq}& \sup_{\theta\in\Theta_{\theta_0}\paren{r_n}} \set{\; \abs{\tr\paren{M_{\theta_0,\beta_{\theta}}-1}} + \abs{Z^\top M_{\theta_0,\beta_{\theta}}Z -\tr\paren{M_{\theta_0,\beta_{\theta}}} } + \abs{ Z^\top M_{\theta_0,\theta}Z-Z^\top M_{\theta_0,\beta_{\theta}}Z}\; }\nonumber\\
&\coloneqq& \sup_{\theta\in\Theta_{\theta_0}\paren{r_n}} \Bigset{ T_{21}\paren{\theta}+T_{22}\paren{\theta}+T_{23}\paren{\theta} }.
\end{eqnarray}
Replacing the random quantities $\hat{\theta}_n$ and $\beta_{\hat{\theta}_n}$ with the nonrandom parameters $\theta$ and $\beta_{\theta}$ is the key advantage of $\paren{D}$. Now we control the terms $T_{21},T_{22}$ and $T_{23}$ from above, uniformly over $\Theta_{\theta_0}\paren{r_n}$. Lemma \ref{DiscritizationLemma2} guarantees the existence of a scalar $C_0$, for which 
\begin{equation}\label{UppBndT23}
\lim\limits_{n\rightarrow\infty}\Pr\paren{\sup_{\theta\in\Theta_{\theta_0}\paren{r_n}} T_{23}\paren{\theta}\leq C_0r'_n} \rightarrow 1. 
\end{equation}
$C_0$ depends on $\Lambda_{\max}$ and $\ff{D}_{\max}$. See Lemma \ref{DiscritizationLemma2} for its exact formulation. For large enough $n$, we have
\begin{equation}\label{2UppBndT23}
C_0r'_n=\cc{O}\paren{n^{-1}\sqrt{\ln n}} < \frac{1}{2}\ff{D}_{\max}C_{\min}\sqrt{\frac{\ln n}{n}}.
\end{equation}
Now we control $T_{21}\paren{\theta}$ uniformly from above using Proposition \ref{UnifBndEigValCorrMatrix}. The goal is to show that
\begin{equation}\label{UppBndT21}
\sup_{\theta\in\Theta_{\theta_0}\paren{r_n}} T_{21}\paren{\theta} <  \frac{3}{2}\ff{D}_{\max}C_{\min}\sqrt{\frac{\ln n}{n}}.
\end{equation}
Applying the Cauchy-Schwartz inequality shows that (recalling $M_{\theta_1,\theta_2}$ from \eqref{MandH})
\begin{eqnarray}\label{1UppBndT21}
\sup_{\theta\in\Theta_{\theta_0}\paren{r_n}} T_{21}\paren{\theta} &=& \sup_{\theta\in\Theta_{\theta_0}\paren{r_n}} \abs{\frac{\InnerProd{K_n\paren{\beta_{\theta}}}{K_n(\theta_0)}}{\LpNorm{K_n\paren{\beta_{\theta}}}{2}^2}-1} = \sup_{\theta\in\Theta_{\theta_0}\paren{r_n}} \abs{\frac{\InnerProd{K_n\paren{\beta_{\theta}}}{K_n\paren{\beta_{\theta}}-K_n(\theta_0)}}{\LpNorm{K_n\paren{\beta_{\theta}}}{2}^2}}\nonumber\\
&\leq& \sup_{\theta\in\Theta_{\theta_0}\paren{r_n}} \frac{\LpNorm{K_n\paren{\beta_{\theta}}-K_n(\theta_0)}{2}}{\LpNorm{K_n\paren{\beta_{\theta}}}{2}}\leq \sup_{\theta\in\Theta_{\theta_0}\paren{r_n}}\frac{\LpNorm{K_n\paren{\beta_{\theta}}-K_n(\theta_0)}{2}}{\sqrt{n}}.
\end{eqnarray}
Furthermore, using the part $\paren{b}$ of the Proposition \ref{UnifBndEigValCorrMatrix}, we get
\begin{eqnarray}\label{2UppBndT21}
\sup_{\theta\in\Theta_{\theta_0}\paren{r_n}}\frac{\LpNorm{K_n\paren{\beta_{\theta}}-K_n(\theta_0)}{2}}{\sqrt{n}} &\leq& \ff{D}_{\max}\sup_{\theta\in\Theta_{\theta_0}\paren{r_n}} \LpNorm{\theta_0-\beta_{\theta}}{2}\nonumber\\
&\leq& \ff{D}_{\max}\sup_{\theta\in\Theta_{\theta_0}\paren{r_n}} \paren{\LpNorm{\theta-\beta_{\theta}}{2} + \LpNorm{\theta_0-\theta}{2}}\nonumber\\
&\leq& \ff{D}_{\max}\paren{r_n+r'_n} < \frac{3}{2}\ff{D}_{\max}C_{\min}\sqrt{\frac{\ln n}{n}}.
\end{eqnarray}
Note that the last inequality holds for large enough $n$. So \eqref{2UppBndT23} follows from replacing \eqref{2UppBndT21} into \eqref{1UppBndT21}. In the sequel we achieve a uniform upper bound on $T_{22}$. For brevity define $u_n \coloneqq \Lambda_{\max}\sqrt{mn^{-1}\ln n}$ and select a large enough universal constant $C'$. Recall that $\beta_{\theta}$, by its definition, is an element of the finite set $\cc{N}_{r'_n}\paren{\Theta_{\theta_0}\paren{r_n}}$. Thus, 
\begin{equation*}
\Pr\paren{\sup_{\theta\in\Theta_{\theta_0}\paren{r_n}} T_{22}\paren{\theta}\geq C'u_n}\leq \abs{\cc{N}_{r'_n}\paren{\Theta_{\theta_0}\paren{r_n}}}\sup_{\theta\in\Theta_{\theta_0}\paren{r_n}}\Pr\paren{T_{22}\paren{\theta}\geq C'u_n}.
\end{equation*}
The same trick as \eqref{CovNumUppBnd} leads to $\abs{\cc{N}_{r'_n}\paren{\Theta_{\theta_0}\paren{r_n}}} = o\paren{n^m}$. So, it is adequate to show that 
\begin{equation}\label{UppBndT22}
\Pr\paren{T_{22}\paren{\theta}\geq C'u_n}\leq n^{-m},\quad\forall\;\theta\in\Theta_{\theta_0}\paren{r_n}.
\end{equation}
We employ Hanson-Wright inequality (Theorem $1.1$, \cite{M.Rudelson}) for obtaining a probabilistic upper bound on $T_{22}\paren{\theta}$ (for a fixed $\theta$). 
\begin{equation*}
\Pr\set{ T_{22}\paren{\theta}\geq C'\sqrt{m\ln n}\paren{\LpNorm{M_{\theta_0,\beta_{\theta}}}{2} \vee \OpNorm{M_{\theta_0,\beta_{\theta}}}{2}{2}\sqrt{m\ln n}} }\leq n^{-m},\quad\forall\;\theta\in\Theta_{\theta_0}\paren{r_n}.
\end{equation*}
For simplifying the upper bound on $T_{22}\paren{\theta}$, we control the operator and Frobenius norms of $M_{\theta_0,\beta_{\theta}}$ from above. The following inequalities can be easily justified by Proposition \ref{UnifBndEigValCorrMatrix}.  
\begin{equation}\label{UppBndNormsM}
\LpNorm{M_{\theta_0,\beta_{\theta}}}{2}\leq \Lambda_{\max} n^{-1/2},\quad \OpNorm{M_{\theta_0,\beta_{\theta}}}{2}{2}\leq \Lambda^2_{\max}n^{-1}.
\end{equation}
Replacing \eqref{UppBndNormsM} into Hanson-Wright inequality justifies \eqref{UppBndT22}. Hence
\begin{equation}\label{2UppBndT22}
\Pr\paren{\sup_{\theta\in\Theta_{\theta_0}\paren{r_n}} T_{22}\paren{\theta}\geq C'u_n}\leq \abs{\cc{N}_{r'_n}\paren{\Theta_{\theta_0}\paren{r_n}}} n^{-m}\rightarrow 0,\quad\mbox{as}\; n\rightarrow\infty.
\end{equation}
Substituting inequalities \eqref{UppBndT23}, \eqref{2UppBndT23}, \eqref{UppBndT21} and \eqref{2UppBndT22} into \eqref{UppBndT2} concludes the proof by confirming that $T_2$ goes to zero as $n\rightarrow\infty$.
\end{proof}
Combining Claims \ref{Claim1Thm1} and \ref{Claim2Thm1} ends the proof.
\end{proof}

\begin{proof}[Proof of Theorem \ref{LocMinRate}]
Let $r_n = C\sqrt{n^{-1}\ln n}$ for some strictly positive $C$ whose exact form will be given shortly. Let $(\hat{\phi}_n,\hat{\theta}_n)$ be any stationary point of the optimization problem \eqref{InvFreeOptProb}. Due to the space constraint, we just show that $\hat{\theta}_n\in \Theta_{\theta_0}\paren{r_n}$. The same technique as Claim \ref{Claim2Thm1} in the proof of Theorem \ref{GlobMinRate} attains the convergence rate of $\hat{\phi}_n$. Notice that $\hat{\theta}_n$ is a stationary point of the optimization problem \eqref{OptValInvFreeOptProb}. We show that with a high probability there is no $\theta\in\Theta^c_{\theta_0}\paren{r_n}$ for which the gradient of the objective function in \eqref{OptValInvFreeOptProb} be exactly zero. In order to substantiate our claim, we prove that the absolute value of the inner product of the gradient and a fixed non-zero vector is uniformly greater than zero on $\Theta^c_{\theta_0}\paren{r_n}$.
	
Let us first give the closed form of the gradient function in \eqref{OptValInvFreeOptProb}, which will be denoted by $\brac{G_l\paren{\theta}}^m_{l=1}$. A few lines of direct algebra leads to 
\begin{eqnarray*}
G_l\paren{\theta}&\coloneqq& Z^\top P^l_{\theta_0,\theta}Z\coloneqq \frac{\partial}{\partial\theta_l} Z^\top H_{\theta_0,\theta}Z\\
&=& Z^\top K^{1/2}_n\paren{\theta_0}\set{ \frac{\partial}{\partial\theta_l}K_n\paren{\theta}- \InnerProd{\frac{\partial}{\partial\theta_l}K_n\paren{\theta}}{K_n\paren{\theta}} \frac{K_n\paren{\theta}}{\LpNorm{K_n\paren{\theta}}{2}^2 } } K^{1/2}_n\paren{\theta_0}Z,
\end{eqnarray*}
where $Z\in\bb{R}^n$ is a standard Gaussian vector and $\theta_l,\;l=1,\ldots,m$ is the $l^{\mbox{th}}$ component of $\theta$. Trivially, $G_l(\hat{\theta}_n) = 0$ for any $l=1,\ldots,m$. Choose an arbitrary unit norm vector $\lambda\in\bb{R}^{m}$ and define $Y \coloneqq K^{1/2}_n\paren{\theta_0}Z$. Observe that
\begin{equation}\label{ffFnTheta}
W\paren{\theta} \coloneqq \sum\limits_{j=1}^{m}\lambda_j G_j\paren{\theta} = Y^\top \set{ \sum\limits_{j=1}^{m}\lambda_j\frac{\partial}{\partial\theta_j}K_n\paren{\theta}- \InnerProd{\sum\limits_{j=1}^{m}\lambda_j\frac{\partial}{\partial\theta_j}K_n\paren{\theta}}{K_n\paren{\theta}} \frac{K_n\paren{\theta}}{\LpNorm{K_n\paren{\theta}}{2}^2 } } Y.
\end{equation}
We can conclude that $\hat{\theta}_n\in \Theta_{\theta_0}\paren{r_n}$ in probability, if we can prove that
\begin{equation}\label{SuffCondOptilStatPt}
\Pr\paren{\inf_{\theta\in \Theta^c_{\theta_0}\paren{r_n}} \abs{W\paren{\theta}} > 0}\rightarrow 1, \quad\mbox{as}\;n\rightarrow\infty.
\end{equation}
\eqref{SuffCondOptilStatPt} is immediate after we show the two following claims. 
	
\setcounter{clawithinpf}{0}
	
\begin{clawithinpf}\label{Claim1Thm2}
There exists a positive finite constant $C_0$ (depending on $K$, $\Theta$ and $\cc{D}_n$) such that 
\begin{equation}\label{Claim1Thm2Statement}
\lim\limits_{n\rightarrow\infty}\Pr\paren{\sup_{\theta\in\Theta}\abs{W\paren{\theta}- \sum\limits_{j=1}^{m}\lambda_j\tr \paren{P^{j}_{\theta_0,\theta}}}\geq C_0\sqrt{n\ln n}} = 0.
\end{equation}
\end{clawithinpf}
	
\begin{clawithinpf}\label{Claim2Thm2}
The succeeding inequality holds for large enough $n$ ($C_0$ is from the previous claim).
\begin{equation*}
\inf_{\theta\in \Theta^c_{\theta_0}\paren{r_n}}\abs{\sum\limits_{j=1}^{m}\lambda_j\tr \paren{P^{j}_{\theta_0,\theta}}}> C_0\sqrt{n\ln n}
\end{equation*}
\end{clawithinpf}
	
The Claim \ref{Claim1Thm2} provides a uniform concentration inequality regarding the random function $W\paren{\theta}$. In the Claim \ref{Claim2Thm2}, we obtain a uniform lower bound on the expected value of $W\paren{\theta}$ over $\Theta^c_{\theta_0}\paren{r_n}$.
	
\begin{proof}[Proof of Claim \ref{Claim1Thm2}]
For simplicity, define
\begin{equation}\label{Tn}
T_{n}\paren{\theta}\coloneqq  \set{\sum\limits_{j=1}^{m}\lambda_j\frac{\partial}{\partial\theta_j}K_n\paren{\theta}- \InnerProd{\sum\limits_{j=1}^{m}\lambda_j\frac{\partial}{\partial\theta_j}K_n\paren{\theta}}{K_n\paren{\theta}} \frac{K_n\paren{\theta}}{\LpNorm{K_n\paren{\theta}}{2}^2 }}.
\end{equation}
The matrix $T_{n}\paren{\theta}$ has appeared in the defining formula for $W\paren{\theta}$ in \eqref{ffFnTheta}. The basic facts of the directional derivative yield
\begin{eqnarray}\label{DirDeriv}
T_{n}\paren{\theta} &=& \lim\limits_{\gamma\rightarrow 0} \gamma^{-1}\set{ K_n\paren{\theta+\gamma\lambda} - K_n\paren{\theta}- \InnerProd{K_n\paren{\theta+\gamma\lambda} - K_n\paren{\theta}}{K_n\paren{\theta}} \frac{K_n\paren{\theta}}{\LpNorm{K_n\paren{\theta}}{2}^2 } }\nonumber\\
&=& \lim\limits_{\gamma\rightarrow 0} \gamma^{-1}\set{ K_n\paren{\theta+\gamma\lambda} - \InnerProd{K_n\paren{\theta+\gamma\lambda} }{K_n\paren{\theta}} \frac{K_n\paren{\theta}}{\LpNorm{K_n\paren{\theta}}{2}^2 } }.
\end{eqnarray}
Proposition \ref{ModalProp} is the essential tool in our proof. As inequality \eqref{Claim1Thm2Statement} is the same as \eqref{UnifDeviationBound}, it suffices to justify the three conditions of Proposition \ref{ModalProp}. We verify conditions $\paren{a}$ and $\paren{b}$ simultaneously. Choose $\theta,\theta'\in\Theta$ in an arbitrary way. The following inequalities are obvious implications of Proposition \ref{UnifBndEigValCorrMatrix} ($\Lambda_{\max}$, $\Lambda_{\min}$ and $\ff{D}_{\min}$ are introduced there).
\begin{align*}
&\OpNorm{\sum\limits_{j=1}^{m}\lambda_j P^{j}_{\theta_0,\theta}}{2}{2}\leq \Lambda_{\max}\OpNorm{T_{n}\paren{\theta}}{2}{2},\\
&\OpNorm{\sum\limits_{j=1}^{m}\lambda_j P^{j}_{\theta_0,\theta'}-\sum\limits_{j=1}^{m}\lambda_j P^{j}_{\theta_0,\theta}}{2}{2}\leq \Lambda_{\max}\OpNorm{T_{n}\paren{\theta'}-T_{n}\paren{\theta}}{2}{2},\\
& \OpNorm{\sum\limits_{j=1}^{m}\lambda_j P^{j}_{\theta_0,\theta}}{2}{2}\LpNorm{\sum\limits_{j=1}^{m}\lambda_j P^{j}_{\theta_0,\theta}}{2}^{-1}\leq \frac{\Lambda_{\max}}{\Lambda_{\min}}\frac{\OpNorm{T_{n}\paren{\theta}}{2}{2}}{\LpNorm{T_{n}\paren{\theta}}{2}}.
\end{align*}
Thus, it is sufficient to justify the conditions of Proposition \ref{ModalProp} for $\set{T_{n}\paren{\theta}}_{\theta\in\Theta}$. According to Lemma \ref{MaodlLem3}, $T_{n}\paren{\theta}$ admits conditions $\paren{a}$ and $\paren{b}$. As the operator norm of $T_{n}\paren{\theta}$ is uniformly bounded over $\Theta$, condition $\paren{c}$ holds if $\LpNorm{T_{n}\paren{\theta}}{2}$ grows faster than $\sqrt{\ln n}$. We aim to show that $\LpNorm{T_{n}\paren{\theta}}{2}=\cc{O}\paren{\sqrt{n}}$.
\begin{eqnarray}\label{LowBndFrobNormT}
\LpNorm{T_{n}\paren{\theta}}{2}&\RelNum{\paren{A}}{=}& \lim\limits_{\gamma\searrow 0} \gamma^{-1} \set{ \LpNorm{K_n\paren{\theta+\gamma\lambda}}{2}^2 - \abs{\frac{ \InnerProd{K_n\paren{\theta}}{K_n\paren{\theta+\gamma\lambda}} }{ \LpNorm{K_n\paren{\theta}}{2} }}^2  }^{1/2} \nonumber\\
&=& \lim\limits_{\gamma\searrow 0} \frac{ \LpNorm{K_n\paren{\theta+\gamma\lambda}}{2} }{\gamma\sqrt{2}} \set{ 2-2\abs{\InnerProd{ \frac{K_n\paren{\theta}}{\LpNorm{K_n\paren{\theta}}{2}} }{ \frac{K_n\paren{\theta+\gamma\lambda}}{\LpNorm{K_n\paren{\theta+\gamma\lambda}}{2}} }}^2 }^{1/2} \nonumber\\
&\geq& \lim\limits_{\gamma\searrow 0} \frac{ \LpNorm{K_n\paren{\theta+\gamma\lambda}}{2} }{\gamma\sqrt{2}} \set{ 2-2\abs{\InnerProd{ \frac{K_n\paren{\theta}}{\LpNorm{K_n\paren{\theta}}{2}} }{ \frac{K_n\paren{\theta+\gamma\lambda}}{\LpNorm{K_n\paren{\theta+\gamma\lambda}}{2}} }} }^{1/2} \nonumber\\
&\geq& \lim\limits_{\gamma\searrow 0} \paren{\gamma\sqrt{2}}^{-1} \LpNorm{ K_n\paren{\theta+\gamma\lambda} - \frac{ \LpNorm{K_n\paren{\theta+\gamma\lambda}}{2} }{\LpNorm{K_n\paren{\theta}}{2}} K_n\paren{\theta} }{2}\nonumber\\
&\RelNum{\paren{B}}{\geq}& \lim\limits_{\gamma\searrow 0}  \frac{\LpNorm{ K_n\paren{\theta+\gamma\lambda} - K_n\paren{\theta} }{2}}{ \Lambda_{\max}\gamma\sqrt{2} }\RelNum{\paren{C}}{\geq} \frac{\ff{D}_{\min}}{\Lambda_{\max}}\sqrt{\frac{n}{2}}.
\end{eqnarray}
		
Trivial calculations can prove identity $\paren{A}$. Moreover, $\paren{B}$ and $\paren{C}$ are immediate consequences of Lemma \ref{ModalLem2} and Proposition \ref{UnifBndEigValCorrMatrix}, respectively.
\end{proof}
	
\begin{proof}[Proof of Claim \ref{Claim2Thm2}]
For notational convenience, define $v_{n}\paren{\theta}\coloneqq \LpNorm{K_n\paren{\theta}}{2}^{-1} K_n\paren{\theta}$ for any $\theta\in\Theta$. Also, recall that $r_n = C\sqrt{n^{-1}\ln n}$ for some $C$ to be chosen. We first derive a lower bound on $\abs{{\rm E} \set{W\paren{\theta}}}$ in terms of $v_n$ vectors. Regrouping the terms in \eqref{DirDeriv} leads to
\begin{eqnarray*}
\abs{{\rm E} \set{W\paren{\theta}} } &=& \abs{\InnerProd{K_n\paren{\theta_0}}{T_n\paren{\theta}}}\\
&=&\lim\limits_{\gamma\rightarrow 0}\frac{\LpNorm{K_n\paren{\theta+\gamma\lambda}}{2}}{\gamma \LpNorm{K_n(\theta_0)}{2}^{-1}}\Bigl |\InnerProd{v_{n}(\theta_0)}{v_{n}\paren{\theta+\gamma\lambda}} -\InnerProd{v_{n}\paren{\theta+\gamma\lambda}}{v_{n}\paren{\theta}} \InnerProd{v_{n}\paren{\theta}}{v_{n}(\theta_0)} \Bigr|\\
&\geq& n \lim\limits_{\gamma\rightarrow 0} \abs{\frac{\InnerProd{v_{n}(\theta_0)}{v_{n}\paren{\theta+\gamma\lambda}} -\InnerProd{v_{n}\paren{\theta+\gamma\lambda}}{v_{n}\paren{\theta}} \InnerProd{v_{n}\paren{\theta}}{v_{n}(\theta_0)} }{\gamma}}.
\end{eqnarray*}
Hence, it suffices to show that 
\begin{eqnarray}\label{T}
A&\coloneqq&\inf_{\theta\in \Theta^c_{\theta_0}\paren{r_n}}\lim\limits_{\gamma\rightarrow 0} \abs{\frac{\InnerProd{v_{n}(\theta_0)}{v_{n}\paren{\theta+\gamma\lambda}} -\InnerProd{v_{n}\paren{\theta+\gamma\lambda}}{v_{n}\paren{\theta}} \InnerProd{v_{n}\paren{\theta}}{v_{n}(\theta_0)} }{\gamma}}\nonumber\\
&>& r'_n\coloneqq C_0\sqrt{\frac{\ln n}{n}}.
\end{eqnarray}
Define two unit norm vectors $u_1,u_2\in\bb{R}^m$ and positive scalars $\tau_1,\tau_2$ by
\begin{align*}
& v_{n}(\theta_0) = \InnerProd{v_{n}(\theta_0)}{v_{n}(\theta)}v_{n}(\theta) + \tau_1 u_1,\\
& v_{n}(\theta+\gamma\lambda) = \InnerProd{v_{n}(\theta+\gamma\lambda)}{v_{n}(\theta)}v_{n}(\theta) + \tau_2 u_2.
\end{align*}
Using the fact that both vectors $u_1$ and $u_2$ are perpendicular to $v_n\paren{\theta}$, we get
\begin{equation}\label{ClsdFrmExp}
\frac{\InnerProd{v_{n}(\theta_0)}{v_{n}\paren{\theta+\gamma\lambda}} -\InnerProd{v_{n}\paren{\theta+\gamma\lambda}}{v_{n}\paren{\theta}} \InnerProd{v_{n}\paren{\theta}}{v_{n}(\theta_0)}}{\gamma} = \frac{\tau_1\tau_2}{\gamma}\InnerProd{u_1}{u_2}.
\end{equation}
Achieving tight lower bounds on $\tau_1$ and $\tau_2$ is a crucial step of the proof. Without loss of generality we can assume that $\InnerProd{v_{n}(\theta_0)}{v_{n}(\theta)}$ is non-negative. Thus, for any $\theta\in \Theta^c_{\theta_0}\paren{r_n}$.
\begin{eqnarray}\label{Tau1LowBnd}
\tau_1 &=& \Bigl|v_{n}(\theta_0)-\InnerProd{v_{n}(\theta_0)}{v_{n}(\theta)}v_{n}(\theta)\Bigr| = \sqrt{1-\InnerProd{v_{n}(\theta_0)}{v_{n}(\theta)}^2}\geq \frac{1}{\sqrt{2}}\sqrt{2-2\InnerProd{v_{n}(\theta_0)}{v_{n}(\theta)}}\nonumber\\
&=&\frac{1}{\sqrt{2}}\LpNorm{v_{n}(\theta_0)-v_{n}(\theta)}{2} = \frac{1}{\sqrt{2}\LpNorm{K_n\paren{\theta}}{2} }\LpNorm{K_n\paren{\theta}-\frac{\LpNorm{K_n\paren{\theta}}{2}}{\LpNorm{K_n\paren{\theta_0}}{2}} K_n\paren{\theta_0}  }{2}\nonumber\\
&\RelNum{\paren{D}}{\geq}&\frac{\LpNorm{K_n\paren{\theta}-K_n\paren{\theta_0}}{2} }{ \sqrt{2}\LpNorm{K_n\paren{\theta}}{2}\Lambda_{max} }\RelNum{\paren{E}}{\geq} \frac{\ff{D}_{\min}}{\sqrt{2}\Lambda^2_{\max} }\LpNorm{\theta-\theta_0}{2}\geq \frac{\ff{D}_{\min}}{\sqrt{2}\Lambda^2_{\max} }r_n.
\end{eqnarray}
In above inequality, $\paren{D}$ and $\paren{E}$ has been concluded from Lemma \ref{ModalLem2} and Proposition \ref{UnifBndEigValCorrMatrix}. Applying similar tricks as \eqref{Tau1LowBnd} leads to
\begin{equation}\label{Tau2LowBnd}
\tau_2\geq \frac{\ff{D}_{\min}}{\sqrt{2}\Lambda^2_{\max} } \LpNorm{ \theta+\gamma\lambda-\theta }{2} = \gamma \frac{\ff{D}_{\min}}{\sqrt{2}\Lambda^2_{\max} }.
\end{equation}
Replacing the two inequalities \eqref{Tau1LowBnd} and \eqref{Tau2LowBnd} into \eqref{ClsdFrmExp} shows that
\begin{equation*}
\abs{\frac{\InnerProd{v_{n}(\theta_0)}{v_{n}\paren{\theta+\gamma\lambda}} -\InnerProd{v_{n}\paren{\theta+\gamma\lambda}}{v_{n}\paren{\theta}} \InnerProd{v_{n}\paren{\theta}}{v_{n}(\theta_0)}}{\gamma}}\geq  C\paren{\frac{\ff{D}_{\min}}{\sqrt{2}\Lambda^2_{\max} }}^2 \abs{\InnerProd{u_1}{u_2}} \sqrt{\frac{\ln n}{n}}.
\end{equation*}
		
Using Assumption \ref{Covar&ParamSpAssu2}, one can show that $\InnerProd{u_1}{u_2}$ is nonzero, i.e., $\abs{\InnerProd{u_1}{u_2}}\geq \alpha_0$ for some $\alpha_0>0$, universally over $\theta\in\Theta$ and $\theta_0\in\Theta^c_{\theta}\paren{r_n}$. Finally, choosing $C> C_0/\alpha_0\paren{\sqrt{2}\ff{D}^{-1}_{\min}\Lambda^2_{\max} }^2$ implies that
\begin{equation*}
A \geq C\alpha_0\paren{\frac{\ff{D}_{\min}}{\sqrt{2}\Lambda^2_{\max} }}^2\sqrt{\frac{\ln n}{n}}> r'_n = C_0\sqrt{\frac{\ln n}{n}},
\end{equation*}
which ends the proof.
\end{proof}
	
\end{proof}

\begin{proof}[Proof of Theorem \ref{MinMaxThm}]
We follow the standard techniques presented in the Chapter $2$ of \cite{AB.Tsybakov} for bounding the minimax risk from below. For any $\theta\in\Theta$, $\bb{P}_{\theta}$ stands for the associated distribution to a zero mean Gaussian vector with the covariance function $K_n\paren{\theta}$. Finding far enough (with respect to the Euclidean distance) pair of the correlation parameters, $\theta_i\in\Theta,\; i=1,2$, for which $D\paren{\bb{P}_{\theta_1} \parallel \bb{P}_{\theta_2}} < \alpha = 1/2$ lies at the heart of our proof. The two bounded positive scalars $\ff{D}_{\max}$ and $\Lambda_{\min}$ appearing here are defined in Proposition \ref{UnifBndEigValCorrMatrix}. To ease notation let $r_n \coloneqq \frac{\Lambda_{\min}}{8\ff{D}_{\max}\sqrt{n}}$ (Choose $n$ large enough so that $4r_n\leq \diam\paren{\Theta}$). Choose $\theta_1,\theta_2\in\Theta$ with $2r_n\leq \LpNorm{\theta_2-\theta_1}{2}\leq 4r_n$. The connectedness of $\Theta$ guarantees the existence of such pair of points. We first use the Proposition \ref{KLDivgncLem} to show that $D\paren{\bb{P}_{\theta_1} \parallel \bb{P}_{\theta_2}}\leq \alpha$.
\begin{equation*}
D\paren{\bb{P}_{\theta_1} \parallel \bb{P}_{\theta_2}}\leq 2n\paren{\frac{\ff{D}_{\max}}{\Lambda_{\min}} \LpNorm{\theta_2-\theta_1}{2}}^2\leq 32n\paren{\frac{\ff{D}_{\max}}{\Lambda_{\min}}r_n}^2 = \alpha = \frac{1}{2}.
\end{equation*}
As $\alpha\geq D\paren{\bb{P}_{\theta_1} \parallel \bb{P}_{\theta_2}}$, Theorem $2.2$ of \cite{AB.Tsybakov} yields
\begin{equation}\label{FanoIneq}
\inf_{\hat{\theta}_n}\sup_{\theta_0\in\Theta} \Pr\paren{\LpNorm{\hat{\theta}_n-\theta_0}{2}\geq r_n}\geq \paren{\frac{1}{4}e^{-\alpha} }\vee \paren{\frac{1-\sqrt{\frac{\alpha}{2}}}{2}} = \frac{1}{4}.
\end{equation}
The desired statement follows from the fact that $\LpNorm{\hat{\eta}_n-\eta_0}{2}\geq \LpNorm{\hat{\theta}_n-\theta_0}{2}$.
\end{proof}

\begin{proof}[Proof of Proposition \ref{IdentifProp}]
Recall our notation from Assumptions \ref{Covar&ParamSpAssu} and \ref{Covar&ParamSpAssu2}. The main objective is to verify the inequality \eqref{IdentifCond} (Assumption $\paren{A2}$). Let $\phi$ be any point in $\cc{I}$ and let $\phi_{\max} \coloneqq \sup\set{\phi':\phi'\in\cc{I}}$. We also choose $r_1=r_2=r$. Select small enough $\delta > 0$, which its value will be chosen later and an arbitrary $s\in\cc{D}_n$. First, choose two distinct $\theta_1,\theta_2\in\Theta$ whose mutual Euclidean distance is at least $\delta$ and define $\eta_i\coloneqq\paren{\theta_i,\phi},\;i=1,2$. Clearly, 
\begin{eqnarray*}
\abs{R\paren{s'-s,\eta_1}-R\paren{s'-s,\eta_2}} &=& \phi \abs{K\paren{s'-s,\theta_1}-K\paren{s'-s,\theta_2}}\\
&\leq& \phi_{\max}\abs{K\paren{s'-s,\theta_1}-K\paren{s'-s,\theta_2}},
\end{eqnarray*}
for any $s'\in \cc{D}_n\paren{s,r}$. Thus, according $\paren{A4.b}$, 
\begin{equation}\label{QKRel}
\max_{h\in\cc{D}_n\paren{s,r}-s} \abs{K\paren{h,\theta_1}-K\paren{h,\theta_2}} \geq C_{\Omega,K}\coloneqq  \phi^{-1}_{\max}\max_{h\in\cc{D}_n\paren{s,r}-s}\abs{R\paren{h,\eta_1}-R\paren{h,\eta_2}} > 0.
\end{equation}
In other words, $\paren{A2}$ holds by choosing $M = \delta^{-1}C_{\Omega,K}$. So, we need to verify condition \eqref{IdentifCond} when the distance $\theta_1$ and $\theta_2$ is smaller than $\delta$. In this case, there is small $\tau < \delta$ and $\lambda_0\in\cc{S}^{m-1}$ for which $\theta_2-\theta_1 = \tau\lambda_0$. Moreover, define the unit norm vector $\lambda\in\bb{R}^{m+1}$ by $\lambda = \paren{\lambda_0, 0}$. Because of the continuity of $K$, we can choose $\delta$ small enough for which there exists $M'\in\paren{0,1}$ such that 
\begin{eqnarray*}
\max_{h\in\cc{D}_n\paren{s,r}-s} \frac{\abs{K\paren{h,\theta_1}-K\paren{h,\theta_2}}}{\LpNorm{\theta_2-\theta_1}{2} }&\geq& M'\max_{h\in\cc{D}_n\paren{s,r}-s} \lim\limits_{r\rightarrow 0 }\abs{\frac{K\paren{h,\theta_1+r\lambda}-K\paren{h,\theta_1}}{r}}\\
&=& \vartheta \coloneqq \frac{M'}{\phi_{\max}} \max_{h\in\cc{D}_n\paren{s,r}-s} \abs{\sum\limits_{j=1}^{m+1} \lambda_j \frac{\partial}{\partial \eta_j} R\paren{h,\eta}\Big\rvert_{\eta=\eta_1}}.
\end{eqnarray*}
Condition $\paren{A4.a}$ assures the existence of a positive universal constant $M_2$, for which $\vartheta \geq C'_{\Omega,K}\coloneqq M_2M'/\phi_{\max} > 0$. Hence, $\paren{A2}$ holds by taking $M = C'_{\Omega,K} \wedge \delta^{-1}C_{\Omega,K}$.
\end{proof}

\begin{proof}[Proof of Proposition \ref{A5GeoAnisot}]
For a positive definite matrix $A\in\bb{R}^{d\times d}$ and $h\in\bb{R}^d$, the norm $\norm{h}{A}$ represents the quantity $\sqrt{h^\top Ah}$. Furthermore, define $\cc{B}_r\paren{\cc{D}_n} \coloneqq \set{h=s'-s: s,s'\in\cc{D}_n, \LpNorm{h}{2}\leq r }$, for any $r > 0$. Any geometrical anisotropic covariance function associated to the correlation function $K$ are of the form
\begin{equation*}
R\paren{s'-s,\eta} = \phi K\brac{ \norm{s'-s}{A} },\quad\forall\; s,s'\in\bb{R}^d.
\end{equation*}
Notice that $\eta = \paren{\phi,A}$ is a vector of size $m+1 = d^2+1$. Consider any  $\lambda\in\cc{S}^m$ and $\eta\in\Omega$. The chain rule for derivative leads to the following identity for any $\eta\in\Omega$ and $h\in\bb{R}^d$
\begin{equation}\label{Iofh}
J\paren{h} \coloneqq \sum\limits_{j=1}^{m+1} \lambda_j \frac{\partial}{\partial \eta_j} R\paren{h,\eta} = \lambda_1K\paren{ \norm{h}{A}} + \phi\sum_{i,j=1}^{d} \lambda_{ij}h_ih_j \frac{K'\paren{\norm{h}{A}}}{2\norm{h}{A}},\quad h\ne \zero_d.
\end{equation}
Slightly abusing the notation, $\lambda_{ij}$ denotes the entry of $\lambda$ corresponding to $A_{ij}$. Furthermore, let $\tilde{J}\paren{h}$ denotes the second term in the right hand side of \eqref{Iofh} and define $\Lambda \coloneqq \brac{\lambda_{ij}}^d_{i,j=1}$. One can easily show that $\lim\limits_{h\rightarrow\zero_d}\tilde{J}\paren{h} = 0$. Thus, identity \eqref{Iofh} can be rewritten as
\begin{align}\label{Iofh2}
&J\paren{\zero_d} = \lambda_1K\paren{0} = \lambda_1,\nonumber\\	& J\paren{h} = \lambda_1K\paren{ \norm{h}{A}} + \phi K'\paren{\norm{h}{A}} \frac{h^\top\Lambda h}{2\norm{h}{A}},\quad h\ne \zero_d.
\end{align}
It suffices to prove that $J$ is not a zero function of $h$ over $\cc{B}_r\paren{\cc{D}_n}$ for some appropriately chosen $r$. Assume toward contradiction that $J$ is a zero function over $\cc{B}_r\paren{\cc{D}_n}$. As $\zero_d\in \cc{B}_r\paren{\cc{D}_n}$, $\lambda_1$ equals zero from the first identity in \eqref{Iofh2}. Therefore, $\tilde{J}\paren{h} = 0$ for any $h\in\cc{B}_r\paren{\cc{D}_n}$. Since $K$ is a differentiable and strictly decreasing function in $\paren{0,\infty}$, we have $K'\paren{u}\ne 0$ for any strictly positive $u$. Thus,
\begin{equation}\label{Iofh3}
\tilde{J}\paren{h} = 0,\quad\forall\; h\in\cc{B}_r\paren{\cc{D}_n} \;\Rightarrow\; \InnerProd{\Lambda}{hh^\top} = 0,\quad h\in\cc{B}_r\paren{\cc{D}_n}.
\end{equation}
Now choose $r$ in such a way that $\abs{\cc{B}_r\paren{\cc{D}_n}} > \paren{d+1}$. Due to the non-atomicity of $\ff{P}$, there are almost surely $d$ non-zero and linearly independent vectors, $\set{h_1,\ldots,h_d}$, in $\cc{B}_r\paren{\cc{D}_n}$. According to \eqref{Iofh3}, $\InnerProd{\Lambda}{h_ih^\top_i} = 0$ for $i=1,\ldots,d$. Hence $\Lambda$ is a zero matrix, which leads to a clear contradiction to the fact that $\lambda\in\cc{S}^m$ (recall that $\LpNorm{\lambda}{2}^2 = \lambda^2_1 + \LpNorm{\Lambda}{2}^2$).
\end{proof}

\begin{proof}[Proof of Proposition \ref{A5GeoAnisotUnknwnNu}]
As the proof is akin to that of Proposition \ref{A5GeoAnisot}, the details are dropped for avoiding redundancy. We also keep the notation used there, for simplicity. Notice that $\eta = \paren{\phi,\theta}$ and here $\theta = \paren{\nu, A}$ is of size $m = d^2+1$. Choose any unit norm $\lambda\in\cc{S}^m$. We first consider the rational quadratic case. It is easy to verify the following identity for any $h\in\cc{B}_r\paren{\cc{D}_n}$.
\begin{equation*}
J\paren{h} \coloneqq \sum\limits_{j=1}^{m+1} \lambda_j \frac{\partial}{\partial \eta_j} R\paren{h,\eta} = \phi K_{\nu}\paren{ \norm{h}{A}} \set{\frac{\lambda_1}{\phi} - \lambda_2 \ln\paren{1+\norm{h}{A}^2} -\frac{d/2+\nu}{1+\norm{h}{A}^2}\sum_{i,j=1}^{d} \lambda_{ij}h_ih_j}.
\end{equation*}
It suffices to show that $J$ is a nonzero function on $\cc{B}_r\paren{\cc{D}_n}$ unless all $\lambda$ coefficients are zero. Assume that the premise is not true and $J$ vanishes in $\cc{B}_r\paren{\cc{D}_n}$. Thus, $J\paren{\zero_d} = \lambda_1 = 0$ and so $\tilde{J}$ defined by
\begin{equation*}
\tilde{J}\paren{h} \coloneqq \lambda_2 \ln\paren{1+\norm{h}{A}^2} +\paren{\frac{d/2+\nu}{1+\norm{h}{A}^2}}h^\top\Lambda h,
\end{equation*}
should be a zero function in $\cc{B}_r\paren{\cc{D}_n}$. If $\lambda_2 = 0$, then the rest of the proof is similar to the argument of Proposition \ref{A5GeoAnisot}. So assume that $\lambda_2$ is nonzero. Hence for $C = \phi/\lambda_2\paren{d/2+\nu}$ we have
\begin{equation*}
\paren{1+\norm{h}{A}^2}\ln\paren{1+\norm{h}{A}^2} = C\norm{h}{\Lambda}^2,\quad \forall\;h\in\cc{B}_r\paren{\cc{D}_n},
\end{equation*}
which is almost surely absurd for large enough $r$ and so terminates the proof. Apart from the algebraic form of $J$, our proof for the powered exponential family goes through without any change. In the sequel, we just express $J$ and skip the further details for page constraints.
\begin{equation*}
\frac{J\paren{h}}{\phi K_{\nu}\paren{ \norm{h}{A}}} = \frac{\lambda_1}{\phi} - \lambda_2 \norm{h}{A}^\nu \ln\paren{\norm{h}{A}} - \frac{\nu}{2} \norm{h}{A}^{-\paren{2-\nu}} \norm{h}{\Lambda}^2,\quad J\paren{\zero_d} = \lambda_1.
\end{equation*}
\end{proof}

\begin{proof}[Proof of Theorem \ref{AsympDistThm}]
Let $g:\Omega\mapsto\bb{R}^{m+1}$ represents the gradient of the objective function in \eqref{InvFreeOptProbEtaForm} with respect to $\eta$. Here $g_j, \; j=11,\ldots,\paren{m+1}$ stands for the $j^{\rm {th}}$ entry of $g$. Analyzing the exact second order Taylor expansion of $\sqrt{n}g\paren{\eta}$ around $\eta_0$ at $\eta = \hat{\eta}_n$ is the integral part of the proof. We argue that the second order term of the expansion, which involves the third order derivatives of the covariance function, converges to zero in probability as $n$ grows to infinity. We also show that the first term (zeroth order term) in the expansion converges weakly to a Gaussian random variable. These two ingredients lead to the desirable result by showing the asymptotic normality of the first order term in the expansion, which directly depends on $\sqrt{n}\paren{\hat{\eta}_n-\eta_0}$.
	
For simplicity, define $R^J_n\paren{\eta} = \frac{\partial R_n\paren{\eta}}{\partial\eta_{j_q}\partial\eta_{j_1}}$ for any $q\in\set{1,2}$ and $J\in\set{1,\ldots,m+1}^q$. In fact 
\begin{equation}\label{g_j}
ng_j\paren{\eta} = Y^\top R^j_n\paren{\eta} Y - \InnerProd{R^j_n\paren{\eta}}{R_n\paren{\eta}}.
\end{equation}
Let $\hat{\eta}_n$ be an arbitrary stationary point of optimization problem \eqref{InvFreeOptProbEtaForm}. Clearly, $g\paren{\hat{\eta}_n} = \zero_{m+1}$. The second order approximation of $g_j$ around $\hat{\eta}_n$ yields
\begin{equation*}
\sqrt{n}g_j\paren{\hat{\eta}_n} = \sqrt{n}g_j\paren{\eta_0} + \InnerProd{\sqrt{n}\paren{\hat{\eta}_n-\eta_0}}{\nabla_{\eta}g_j\paren{\eta}\Big\lvert_{\eta=\eta_0}} + \sqrt{n}\Delta_j\paren{\eta_0,\hat{\eta}_n},
\end{equation*}
for some residual function $\Delta_j\paren{\cdot,\cdot}$. Note that $\Delta_j\paren{\eta_0,\hat{\eta}_n}$ is given by
\begin{equation*}
\Delta_j\paren{\eta_0,\hat{\eta}_n} = \paren{\hat{\eta}_n-\eta_0}^\top \brac{ \frac{\partial g_j\paren{\eta}}{\partial\eta_{l_1}\partial\eta_{l_2}}\Big\lvert_{\eta=z_j} }^{m+1}_{l_1,l_2=1} \paren{\hat{\eta}_n-\eta_0}
\end{equation*}
in which $z_j$ lies on the line segment between $\eta_0$ and $\hat{\eta}_n$. Proposition $D.10$ of \cite{F.Bachoc} guarantees the statement \eqref{AsympNormEq} for $\Sigma = \Sigma^{-1}_2\Sigma_1\Sigma^{-1}_2$ and hence concludes the proof, if
\begin{enumerate}[label = (\alph*)]
\item The matrix $\Sigma_2$ defined as the following, is well defined and positive definite.
\begin{equation*}
V^n\coloneqq\brac{-\frac{\partial}{\partial\eta_l}g_k\paren{\eta}\Big\lvert_{\eta = \eta_0}}^{m+1}_{l,k=1}\cp{\Pr}\Sigma_2,\;\; \mbox{as}\; n\rightarrow\infty.
\end{equation*} 
\item $\sqrt{n}g\paren{\eta_0}\cp{d}\cc{N}\paren{\zero_{m+1},\Sigma_1}$ for some positive semidefinite matrix $\Sigma_1\in\bb{R}^{\paren{m+1}\times \paren{m+1}}$.
\item $\Pr\paren{\lim\limits_{n\rightarrow\infty}\sqrt{n}\Delta_j\paren{\eta_0,\hat{\eta}_n} = 0} = 1$, for any $j\in\set{1,\ldots,m+1}$.
\end{enumerate}
The reminder of the proof hinges on the following technicalities which verify conditions $\paren{a}-\paren{c}$.\\
	
\textbf{Validating condition $\paren{a}$.} The entries of $V^n$ has the following explicit form.
\begin{equation*}
V^n_{lk} = \frac{1}{n}\InnerProd{R^l_n\paren{\eta_0}}{R^k_n\paren{\eta_0}} + \frac{1}{n}\set{Y^\top R^{lk}_n\paren{\eta_0} Y - \InnerProd{R^{lk}_n\paren{\eta_0}}{R_n\paren{\eta_0}}}.
\end{equation*}
Now define 
\begin{equation*}
\Sigma_2 \coloneqq \brac{\lim\limits_{n\rightarrow\infty}\frac{\InnerProd{R^l_n\paren{\eta_0}}{R^k_n\paren{\eta_0}}}{n}}^{m+1}_{l,k=1}.
\end{equation*}
Notice that the entries of $\Sigma_2$ are well defined and bounded due to part $\paren{a}$ of Proposition \ref{UnifBndEigValCorrMatrix2}. The proof will be presented in two steps: First we show that $\Sigma_2$ is a positive definite matrix. Second, we prove that $\Phi^n_{lk}\coloneqq\set{Y^\top R^{lk}_n\paren{\eta_0} Y - \InnerProd{R^{lk}_n\paren{\eta_0}}{R_n\paren{\eta_0}}}/n$ converges to zero in probability. To substantiate the first claim, consider an arbitrary $\lambda\in\cc{S}^m$. It is required to show that $\lambda^\top\Sigma_2\lambda > c$, for some constant $c>0$. The condition $\paren{A4.a}$ guarantees the existence of positive scalars $M_2,r_2$ such that for any $s\in\cc{D}_n$,
\begin{equation}\label{Cond4.A}
\max_{s'\in\cc{D}_n\paren{s,r_2}}\abs{\sum\limits_{l=1}^{m+1}\lambda_l\frac{\partial}{\partial\eta_l}R\paren{s-s',\eta}\Big\lvert_{\eta=\eta_0}} \geq M_2.
\end{equation}
Thus,
\begin{eqnarray}\label{Sigma2PD}
\lambda^\top\Sigma_2\lambda &=& \lim\limits_{n\rightarrow\infty} \sum\limits
_{l,k=1}^{m+1}\lambda_l\lambda_k\frac{\InnerProd{R^l_n\paren{\eta_0}}{R^k_n\paren{\eta_0}}}{n} = \lim\limits_{n\rightarrow\infty}\frac{1}{n} \LpNorm{\sum\limits_{l=1}^{m+1}\lambda_lR^l_n\paren{\eta_0} }{2}^2\nonumber\\
&=& \lim\limits_{n\rightarrow\infty}\frac{1}{n}\LpNorm{ \brac{\sum\limits_{l=1}^{m+1}\lambda_l\frac{\partial}{\partial\eta_l}R\paren{s'-s,\eta}\Big\lvert_{\eta=\eta_0} }_{s,s'\in\cc{D}_n} }{2}^2 \RelNum{\paren{A}}{\geq} M^2_2.
\end{eqnarray}
Here, inequality $\paren{A}$ is an easy consequence of \eqref{Cond4.A}. The rest of the proof is devoted to prove the second claim. Choose an arbitrary strictly positive $\epsilon$. As $\Phi^n_{lk}$ is a zero mean random variable, using Chebyshev's inequality we get
\begin{eqnarray*}
\Pr\paren{\abs{\Phi^n_{lk}}\geq \epsilon }&\leq& \frac{\var \paren{\Phi^n_{lk}}}{\epsilon^2} = \frac{2\LpNorm{R^{1/2}_n\paren{\eta_0}R^{lk}_n\paren{\eta_0}R^{1/2}_n\paren{\eta_0}}{2}^2 }{n^2\epsilon^2}\RelNum{\paren{B}}{\lesssim} \paren{\frac{\phi_0}{n\epsilon}\LpNorm{R^{lk}_n\paren{\eta_0}}{2}}^2\\ &\RelNum{\paren{C}}{=}& \cc{O}\paren{n^{-1}}\rightarrow 0,
\end{eqnarray*}
in which $\paren{B}$ and $\paren{C}$ are implied by Propositions \ref{UnifBndEigValCorrMatrix} and \ref{UnifBndEigValCorrMatrix2}, respectively (See Appendix \ref{AuxRes} for more details about the constants).
\\
\textbf{Validating condition $\paren{b}$.} Define $Q^j_n \coloneqq n^{-1/2}R^{1/2}_n\paren{\eta_0}R^j_n\paren{\eta_0}R^{1/2}_n\paren{\eta_0}$ for $1\leq j\leq m+1$, and write $\Psi_{n,\lambda}\coloneqq\lambda_1 Q^1_n + \ldots + \lambda_{m+1} Q^{m+1}_n$ for any $\lambda = \paren{\lambda_1,\ldots,\lambda_{m+1}}\in\cc{S}^m$. Rewriting \eqref{g_j} yields
\begin{equation*}
\sqrt{n}g_j\paren{\eta_0} \eqd Z^\top Q^j_nZ - \tr\paren{Q^j_n}.
\end{equation*} 
The asymptotic normality of $\sqrt{n}g\paren{\eta_0}$ is justified if there is a positive semi-definite $\Sigma_1$ such that $\InnerProd{\lambda}{\sqrt{n}g\paren{\eta_0}}\cp{d} \cc{N}\paren{0,\lambda^\top\Sigma_1\lambda}$ for any $\lambda\in\cc{S}^m$. This statement trivially holds for zero $\Psi_{n,\lambda}$. So, without loss of generality assume that $\Psi_{n,\lambda}$ is non-zero. Observe that
\begin{equation*}
\InnerProd{\lambda}{\sqrt{n}g\paren{\eta_0}} = \frac{\set{Z^\top \Psi_{n,\lambda}Z - \tr\paren{\Psi_{n,\lambda} }}}{\LpNorm{\Psi_{n,\lambda}}{2} }\LpNorm{\Psi_{n,\lambda}}{2}.
\end{equation*}
We claim that $\lim\limits_{n\rightarrow\infty}2\LpNorm{\Psi_{n,\lambda}}{2}^2 = \lambda^\top\Sigma_1\lambda$ for a covariance matrix $\Sigma_1$. The construction of $\Psi_{n,\lambda}$ yields
\begin{equation}\label{Psinlambda}
2\LpNorm{\Psi_{n,\lambda}}{2}^2 = 2\lambda^\top \brac{\InnerProd{Q^k_n}{Q^l_n}}^{m+1}_{l,k=1}\lambda.
\end{equation}
Thus, it is enough to show that the matrix defined by $\Sigma_1 \coloneqq \lim\limits_{n\rightarrow\infty} 2\brac{\InnerProd{Q^k_n}{Q^l_n}}^{m+1}_{l,k=1}$ is well defined (with bounded entries) and positive semi-definite. Well definiteness of $\Sigma_1$ can be proved using the same techniques as the proof of Claim \ref{Claim1Thm4} and by employing Propositions \ref{UnifBndEigValCorrMatrix} and \ref{UnifBndEigValCorrMatrix2}. The positive semi-definite property of $\Sigma_1$ is an immediate consequence of \eqref{Psinlambda}. We conclude the proof by showing that  
\begin{equation*}
\frac{\set{Z^\top \Psi_{n,\lambda}Z - \tr\paren{\Psi_{n,\lambda} }}}{\sqrt{2}\LpNorm{\Psi_{n,\lambda}}{2} }\cp{d} \cc{N}\paren{0,1}.
\end{equation*}
According to Lemma \ref{AsymNormQForm}, this statement is valid if the following claim holds.
	
\setcounter{clawithinpf}{0}
	
\begin{clawithinpf}\label{Claim1Thm4}
$\LpNorm{\Psi_{n,\lambda}}{2}^{-1}\OpNorm{\Psi_{n,\lambda}}{2}{2} \leq C/\sqrt{n}$ for some positive scalar $C$. 
\end{clawithinpf}
	
\begin{proof}[Proof of Claim \ref{Claim1Thm4}]
We show that $C$ depends on $m$, $\Lambda_{\max}$, $\Lambda_{\min}$ and $\Lambda'_{\max}$ (Except $m$, all the constant are introduced in Propositions \ref{UnifBndEigValCorrMatrix} and \ref{UnifBndEigValCorrMatrix2}). Obviously $\Psi_{n,\lambda}$ can be rewritten as, 
\begin{equation*}
\Psi_{n,\lambda} = \frac{1}{\sqrt{n}}R^{1/2}_n\paren{\eta_0}\set{\sum\limits_{j=1}^{m+1}\lambda_jR^j_n\paren{\eta_0}}R^{1/2}_n\paren{\eta_0}
\end{equation*}
Applying Propositions \ref{UnifBndEigValCorrMatrix}, we get
\begin{equation}\label{Ratio}
\frac{\OpNorm{\Psi_{n,\lambda}}{2}{2}}{\LpNorm{\Psi_{n,\lambda}}{2}}\leq \frac{\OpNorm{R_n\paren{\eta_0}}{2}{2}\OpNorm{\sum_{j=1}^{m+1}\lambda_jR^j_n\paren{\eta_0}}{2}{2} }{\LpNorm{\sum_{j=1}^{m+1}\lambda_jR^j_n\paren{\eta_0}}{2}\lambda_{\min}\set{R_n\paren{\eta_0}} }\leq \frac{\Lambda_{\max}}{\Lambda_{\min}}\OpNorm{\sum_{j=1}^{m+1}\lambda_jR^j_n\paren{\eta_0}}{2}{2} \LpNorm{\sum_{j=1}^{m+1}\lambda_jR^j_n\paren{\eta_0}}{2}^{-1}
\end{equation}
Furthermore, using Proposition \ref{UnifBndEigValCorrMatrix2} leads to
\begin{equation*}
\OpNorm{\sum_{j=1}^{m+1}\lambda_jR^j_n\paren{\eta_0}}{2}{2} \leq \sum_{j=1}^{m+1} \abs{\lambda_j}\OpNorm{R^j_n\paren{\eta_0}}{2}{2}\leq \Lambda'_{\max}\LpNorm{\lambda}{1}\leq \Lambda'_{\max}\sqrt{m+1}.
\end{equation*}
From \eqref{Sigma2PD} we know that there is a scalar $C_0 \in\paren{0,\infty}$ for which $\LpNorm{\sum_{j=1}^{m+1}\lambda_jR^j_n\paren{\eta_0}}{2} \geq C_0\sqrt{n}$. Replacing the last two inequalities into \eqref{Ratio} ends the proof.
\end{proof}
	
In conclusion we state that the condition $\paren{c}$ can be proved using akin techniques as the proof of Proposition $3.2$ of \cite{F.Bachoc}. We omit the technical details due to the space constraints. 
\end{proof}



\appendix
\makeatletter
\def\@seccntformat#1{\csname Pref@#1\endcsname \csname the#1\endcsname\quad}
\def\Pref@section{~}
\makeatother

\renewcommand*{\appendixname}{}

\section{Appendix: Auxiliary results}\label{AuxRes}

We first state an elementary lemma regarding the basic properties of the operator norm. The proof is skipped since it is straightforward.

\begin{lem}\label{ModalLem1}
For any matrix $A\in\bb{R}^{n\times n}$, define its absolute value by $\abs{A}\coloneqq \brac{\abs{A_{ij}}}^n_{i,j=1}$. Then, $\OpNorm{A}{2}{2}\leq \OpNorm{\abs{A}}{2}{2}$. Moreover, the largest eigenvalue of $\abs{A}$ can be written as
\begin{equation*}
\OpNorm{\abs{A}}{2}{2} = \max_{v\in\cc{S}^{n-1}_{+}} v^\top\abs{A}v,
\end{equation*} 
in which $\cc{S}^{n-1}_{+}$ denotes the collection of unit norm vectors with non-negative entries in $\bb{R}^n$.
\end{lem}

The next result examines the perturbation of some norms of $K_n\paren{\theta}$ with respect to $\theta$. It appears to be of great importance for proving Theorems \ref{GlobMinRate}--\ref{AsympDistThm} in the Section \ref{Proofs}.

\begin{prop}\label{UnifBndEigValCorrMatrix}
Suppose that $\cc{D}_n$ admits Assumption \ref{ObsLocAssu}. Moreover, Assumption \ref{Covar&ParamSpAssu} holds for $\Theta$ and $K$. Construct $n \times n$ correlation matrix $K_n\paren{\theta} \coloneqq \brac{K\paren{s-s',\theta}}_{s,s'\in\cc{D}_n}$ for any $\theta\in\Theta$.
\begin{enumerate}[label=(\alph*),leftmargin=*]
\item There are bounded positive scalars $\Lambda_{\min}$ and $\Lambda_{\max}$ (depending on $K$, $\Theta$, $d$ and $\delta$) such that
\begin{equation*}
\Lambda_{\min}\leq \min_{n\in\bb{N}}\min_{\theta\in\Theta} \frac{1}{\OpNorm{K^{-1}_n\paren{\theta}}{2}{2}},\quad  \max_{n\in\bb{N}}\max_{\theta\in\Theta} \OpNorm{K_n\paren{\theta}}{2}{2}\leq \Lambda_{\max}.
\end{equation*} 
\item There exist scalars $\ff{D}_{\min},\ff{D}_{\max}\in\paren{0,\infty}
$ (depending on $K$, $\Theta$, $d$ and $\delta$) such that
\begin{equation}\label{OpNormofDiffUppBnd}
\OpNorm{K_n\paren{\theta_2}-K_n\paren{\theta_1}}{2}{2}\leq \ff{D}_{\max}\LpNorm{\theta_2-\theta_1}{2},
\end{equation}
\begin{equation}\label{L2NormofDiffUppBnd}
\frac{1}{\sqrt{n}}\LpNorm{K_n\paren{\theta_2}-K_n\paren{\theta_1}}{2}\leq \ff{D}_{\max}\LpNorm{\theta_2-\theta_1}{2},
\end{equation}
and
\begin{equation}\label{L2NormofDiffLowBnd}
\frac{1}{\sqrt{n}}\LpNorm{K_n\paren{\theta_2}-K_n\paren{\theta_1}}{2}\geq \ff{D}_{\min}\LpNorm{\theta_2-\theta_1}{2},
\end{equation}
for any $\theta_1,\theta_2\in\Theta$.
\end{enumerate}
\end{prop}

\begin{proof}
Part $\paren{a}$ is similar to Proposition $D.4$ and Lemma $D.5$ of \cite{F.Bachoc} and can be substantiated in an analogous way. So we skip its proof due to the space limit. Furthermore, Eq. \eqref{L2NormofDiffUppBnd} is an immediate consequence of Eq. \eqref{OpNormofDiffUppBnd}. Thus we only need to prove inequalities \eqref{OpNormofDiffUppBnd} and \eqref{L2NormofDiffLowBnd}. For brevity, define two $n\times n$ matrices $\ff{A}$ and $\ff{B}$ by 
\begin{equation*}
\ff{A}_{ij} = \abs{K\paren{s_i-s_j,\theta_2}-K\paren{s_i-s_j,\theta_1}},\quad\ff{B}_{ij} = \paren{1+ \LpNorm{s_i-s_j}{2}^{d+1} }^{-1}.
\end{equation*}
For distinct $\theta_1,\theta_2\in\Theta$, we have
\begin{eqnarray}\label{OpNormofDiffUppBndNrmlized}
\frac{\OpNorm{K_n\paren{\theta_2}-K_n\paren{\theta_1}}{2}{2}}{\LpNorm{\theta_2-\theta_1}{2}} &\RelNum{\paren{A}}{\leq}& \LpNorm{\theta_2-\theta_1}{2}^{-1}\OpNorm{\ff{A}}{2}{2} \RelNum{\paren{B}}{=} \LpNorm{\theta_2-\theta_1}{2}^{-1}\max_{v\in\cc{S}^{n-1}_+} \sum\limits^{n}_{i,j=1} v_i\ff{A}_{ij}v_j \nonumber\\
&\RelNum{\paren{C}}{=}& \max_{v\in\cc{S}^{n-1}_+}\sum\limits_{i,j=1}  v_iv_j\abs{\int\limits^{1}_{0} \InnerProd{\nabla_{\theta}K\paren{s_i-s_j,\theta_1+t\paren{\theta_2-\theta_1}}}{\frac{\theta_2-\theta_1}{\LpNorm{\theta_2-\theta_1}{2}}} dt} \nonumber\\
&\leq&  \max_{v\in\cc{S}^{n-1}_+}\sum\limits_{i,j=1}  v_iv_j \int\limits^{1}_{0} \LpNorm{\nabla_{\theta}K\paren{s_i-s_j,\theta_1+t\paren{\theta_2-\theta_1}}}{2} dt \nonumber\\
&\RelNum{\paren{D}}{\leq}& C_{K,\Theta}\OpNorm{\ff{B}}{2}{2}.
\end{eqnarray}
Notice that $\paren{A}$ and $\paren{B}$ follow from Lemma \ref{ModalLem1}. Identity $\paren{C}$ is implied by the fundamental theorem of calculus for line integrals and inequality $\paren{D}$ is a consequence of $\paren{A3}$ in the Assumption \ref{Covar&ParamSpAssu}. In the sequel we introduce an upper bound on $\OpNorm{\ff{B}}{2}{2}$. Using \emph{Gersgorin's circle theorem} (Theorem $6.1.1$, \cite{RA.Horn}), we get
\begin{equation}\label{GreshgorinCThm}
\OpNorm{\ff{B}}{2}{2}\leq 1+ \max_{i=1,\ldots,n} \sum\limits_{j\ne i} \frac{1}{1+\LpNorm{s_i-s_j}{2}^{d+1}}\RelNum{\paren{E}}{\leq} \frac{\ff{D}_{\max}}{C_{K,\Theta}}.
\end{equation}
Because of Assumption \ref{ObsLocAssu}, we can apply Lemma $D.1$ of \cite{F.Bachoc} to guarantee the existence of a bounded positive scalar $\ff{D}_{\max}$ for which inequality $\paren{E}$ holds. Combining the inequalities \eqref{OpNormofDiffUppBndNrmlized} and \eqref{GreshgorinCThm} concludes the proof of this part.
	
Now we turn to prove Eq. \eqref{L2NormofDiffLowBnd}. Recall $M$ and $r$ from $\paren{A2}$ in Assumption \ref{Covar&ParamSpAssu} and set $\ff{D}_{\min} = M$. Applying inequality \eqref{IdentifCond} terminates the proof by
\begin{eqnarray*}
\LpNorm{K_n\paren{\theta_2}-K_n\paren{\theta_1}}{2}^2&=& \sum\limits_{s,s'\in\cc{D}_n} \abs{K\paren{s-s',\theta_2}-K\paren{s-s',\theta_1}}^2\nonumber\\
&\geq& \sum\limits_{s\in\cc{D}_n} \max_{h\in\cc{D}_n\paren{s,r} } \abs{K\paren{h,\theta_2}-K\paren{h,\theta_1}}^2 
\geq M^2n\LpNorm{\theta_2-\theta_1}{2}^2 \nonumber\\
&=& \ff{D}^2_{\min}n\LpNorm{\theta_2-\theta_1}{2}^2.
\end{eqnarray*}
\end{proof}

For ease of reference, we present the following result as a standalone Proposition. Note that its proof is akin to that of Proposition \ref{UnifBndEigValCorrMatrix} and so it will be skipped to avoid repetition.

\begin{prop}\label{UnifBndEigValCorrMatrix2}
Suppose that $\cc{D}_n$ admits Assumption \ref{ObsLocAssu}. Moreover, $\Theta$ and $K$ satisfy Assumption \ref{Covar&ParamSpAssu2}. Construct the matrix $\partial K_n\paren{\theta}/\partial\theta_j  \coloneqq \brac{\partial K\paren{s-s',\theta}/\partial\theta_j }_{s,s'\in\cc{D}_n}$,for $\theta\in\Theta$ and $j=1,\ldots,m$.
\begin{enumerate}[label=(\alph*),leftmargin=*]
\item There is a bounded strictly positive scalar $\Lambda'_{\max}$ (depending on $K$, $\Theta$, $d$ and $\delta$) such that
\begin{equation*}
\max_{n\in\bb{N}}\max_{\theta\in\Theta} \OpNorm{\frac{\partial}{\partial\theta_j}K_n\paren{\theta}}{2}{2}\leq \Lambda'_{\max},
\end{equation*} 
\item There is $\ff{D}'_{\max}>0$ such that for any $\theta_1,\theta_2\in\Theta$
\begin{equation}\label{OpNormofDiffUppBnd2}
\OpNorm{\frac{\partial}{\partial\theta_j}K_n\paren{\theta_2}-\frac{\partial}{\partial\theta_j}K_n\paren{\theta_1}}{2}{2}\leq \ff{D}'_{\max}\LpNorm{\theta_2-\theta_1}{2}.
\end{equation}
\end{enumerate}	
\end{prop}

The bounded positive scalars $\ff{D}_{\max}, \ff{D}_{\min}$ and $\Lambda_{\max}$, which have been introduced in the Proposition \ref{UnifBndEigValCorrMatrix}, become frequently apparent in the subsequent results in this section. It is also proper to remind the reader that $\cc{N}_{\epsilon}\paren{\cc{A}}$ stands for the $\epsilon-$net of $\cc{A}$ with respect to the Euclidean distance. Furthermore, the matrices $H_{\theta_1,\theta_2}$ and $M_{\theta_1,\theta_2}$ have been formerly defined in \eqref{MandH} for any pair of the correlation function parameters $\theta_1,\theta_2$. The succeeding two Lemmas (\ref{DiscritizationLemma} and \ref{DiscritizationLemma2}), which come in handy in the proof of Theorem \ref{GlobMinRate}, establish a probabilistic upper bound on the maximum of a quadratic Gaussian expression over a uncountable set $\Theta$ in terms of its largest value over one of its finite subset. To avoid repetition, we omit the proof of the latter Lemma. Since it is akin to that of the former one (with a slightly different algebra).

\begin{lem}\label{DiscritizationLemma}
Let $Z\in\bb{R}^n$ be a standard Gaussian vector and suppose that $\cc{D}_n$ satisfies Assumption \ref{ObsLocAssu}. Furthermore, assume that $\Theta$ and $K$ admit Assumption \ref{Covar&ParamSpAssu}. For any vanishing positive sequence $\set{r_n}_{n\in\bb{N}}$, any non-empty $\bar{\Theta}\subset\Theta$ and each $\theta_0\in\Theta$,
\begin{equation}\label{DiscIneq}
\LHS\coloneqq \lim\limits_{n\rightarrow\infty}\Pr\set{\paren{\sup_{\theta\in\bar{\Theta}} Z^\top H_{\theta_0,\theta}Z - \sup_{\theta\in\cc{N}_{r_n}\paren{\bar{\Theta}}} Z^\top H_{\theta_0,\theta}Z} \geq Cr_n\sqrt{n}}= 0,
\end{equation}
where $C = 2\Lambda_{\max}\paren{1+\ff{D}_{\max}}$. 
\end{lem}

\begin{proof}
For simplicity, define $Y = K^{1/2}_n\paren{\theta_0}Z$. Choose $\theta\in\bar{\Theta}$ arbitrarily. Trivially, there is $\beta_{\theta}\in\cc{N}_{r_n}\paren{\bar{\Theta}}$ such that $\LpNorm{\theta-\beta_{\theta}}{2}\leq r_n$. Thus, with probability one
\begin{eqnarray}\label{UppBndLHS}
\LHS &\leq& \sup_{\theta\in\bar{\Theta}} \paren{Z^\top H_{\theta_0,\theta}Z-Z^\top H_{\theta_0,\beta_{\theta}}Z} = \sup_{\theta\in\bar{\Theta}} \set{Y \paren{\frac{K_n\paren{\theta}}{\LpNorm{K_n\paren{\theta}}{2}}-\frac{K_n\paren{\beta_{\theta}}}{\LpNorm{K_n\paren{\beta_{\theta}}}{2}}}Y}\nonumber\\
&\leq& \LpNorm{Y}{2}^2 \sup_{\theta\in\bar{\Theta}}\OpNorm{\frac{K_n\paren{\theta}}{\LpNorm{K_n\paren{\theta}}{2}}-\frac{K_n\paren{\beta_{\theta}}}{\LpNorm{K_n\paren{\beta_{\theta}}}{2}}}{2}{2}.
\end{eqnarray}
Next we obtain an upper bound on the operator norm term in the inequality \eqref{UppBndLHS}. Observe that
\begin{eqnarray}\label{UppBndLHS2}
\OpNorm{\frac{K_n\paren{\theta}}{\LpNorm{K_n\paren{\theta}}{2}}-\frac{K_n\paren{\beta_{\theta}}}{\LpNorm{K_n\paren{\beta_{\theta}}}{2}}}{2}{2} &\RelNum{\paren{A}}{\leq}& \set{\frac{\OpNorm{K_n\paren{\theta}-K_n\paren{\beta_{\theta}}}{2}{2}}{\LpNorm{K_n\paren{\theta}}{2}} + \frac{\OpNorm{K_n\paren{\beta_{\theta}}}{2}{2}\LpNorm{K_n\paren{\theta}-K_n\paren{\beta_{\theta}}}{2}}{\LpNorm{K_n\paren{\theta}}{2}\LpNorm{K_n\paren{\beta_{\theta}}}{2}} }\nonumber\\
&\RelNum{\paren{B}}{\leq}& \set{\frac{\OpNorm{K_n\paren{\theta}-K_n\paren{\beta_{\theta}}}{2}{2}}{\sqrt{n}} + \frac{\OpNorm{K_n\paren{\beta_{\theta}}}{2}{2}\LpNorm{K_n\paren{\theta}-K_n\paren{\beta_{\theta}}}{2}}{n} }\nonumber\\
&\RelNum{\paren{C}}{\leq}& \paren{ \frac{\ff{D}_{\max}\LpNorm{\theta-\beta_{\theta}}{2}}{\sqrt{n}} + \frac{\Lambda_{\max}\ff{D}_{\max}\LpNorm{\theta-\beta_{\theta}}{2}\sqrt{n}}{n} }\nonumber\\
&\leq&\frac{\ff{D}_{\max}\paren{1+\Lambda_{\max}}}{\sqrt{n}} r_n.
\end{eqnarray}
Here $\paren{A}$ follows from the triangle inequality alongside some straightforward algebra. Also $\paren{B}$ can be deduced from the fact that $K_n\paren{\cdot}$ is correlation matrix and $\paren{C}$ is implied by Proposition \ref{UnifBndEigValCorrMatrix}. Substituting Eq. \eqref{UppBndLHS2} into Eq. \eqref{UppBndLHS} leads to
\begin{equation}\label{UppBndLHS3}
\LHS \leq \frac{\ff{D}_{\max}\paren{1+\Lambda_{\max}}}{\sqrt{n}} r_n \LpNorm{Y}{2}^2.
\end{equation}
We finish the proof by finding a tight probabilistic upper bound on $\LpNorm{Y}{2}^2$. Obvious calculations show that ${\rm E} \LpNorm{Y}{2}^2 = n$. Hanson-Wright inequality (Theorem $1.1$, \cite{M.Rudelson}) ensures the existence of a bounded universal constant $C>0$ such that
\begin{eqnarray*}
\Pr\paren{\abs{\LpNorm{Y}{2}^2-n}\geq C\LpNorm{K_n\paren{\theta_0}}{2}\sqrt{\ln n}}&\leq& \exp\set{-\paren{ \ln n \vee \sqrt{\ln n} \frac{\LpNorm{K_n\paren{\theta_0}}{2}}{ \OpNorm{K_n\paren{\theta_0}}{2}{2} } }}\\
&\RelNum{\paren{D}}{\leq}& \exp\set{-\paren{ \ln n \vee \frac{\sqrt{n\ln n}}{\Lambda_{\max}}}} = \frac{1}{n}.
\end{eqnarray*}
Notice that $\paren{D}$ is implied by the part $\paren{a}$ of Proposition \ref{UnifBndEigValCorrMatrix} and the last inequality holds whenever $n\geq n_0$ for some appropriately chosen $n_0$. Thus, with probability at least $1-n^{-\paren{1+\xi}}$, we have
\begin{equation}\label{UppBndNorm2Y}
\LpNorm{Y}{2}^2\leq n + C\LpNorm{K_n\paren{\theta_0}}{2}\sqrt{\paren{1+\xi}\ln n}\leq n+C\Lambda_{\max}\sqrt{\paren{1+\xi}n\ln n}\leq 2n.
\end{equation}
The last inequality in \eqref{UppBndNorm2Y} holds if $n\geq n'_0$ for some $n'_0$. Assuming that $n$ is at least $n_0\vee n'_0$, combining the inequalities \eqref{UppBndLHS3} and \eqref{UppBndNorm2Y} concludes the proof.
\end{proof}

\begin{lem}\label{DiscritizationLemma2}
Let $Z\in\bb{R}^n$ be a standard Gaussian vector. Suppose that Assumption \ref{ObsLocAssu} and Assumption \ref{Covar&ParamSpAssu} hold for $\cc{D}_n$, $\Theta$ and $K$. For any strictly positive vanishing sequence $\set{r_n}^{\infty}_{n=1}$, any non-empty $\bar{\Theta}\subset\Theta$ and arbitrary $\theta_0\in\Theta$, 
\begin{equation*}
\Pr\set{ \sup_{\theta\in\bar{\Theta}} \abs{Z^\top\paren{M_{\theta_0,\theta}-M_{\theta_0,\beta_{\theta}}} Z}\geq Cr_n}\rightarrow 0,\quad\mbox{as}\; n\rightarrow\infty.
\end{equation*}
Here $\beta_{\theta}$ represents the nearest element of $\cc{N}_{r_n}\paren{\bar{\Theta}}$ to $\theta$ and $C = 2\ff{D}_{\max}\paren{1+2\Lambda^2_{\max}}$.
\end{lem}

The next result not only appears in the proof of Theorem \ref{LocMinRate}, but is also required for substantiating some of the subsequent results in this section. Although it can be easily verified by Lemma $A.1$ of \cite{F.Bachoc}, we provide its proof for the sake of completeness.

\begin{lem}\label{ModalLem2}
Under the same notation and conditions as Proposition \ref{UnifBndEigValCorrMatrix},
\begin{equation}\label{LowBndModalIneq}
\LpNorm{K_n\paren{\theta_2}-\frac{\LpNorm{K_n\paren{\theta_2}}{2}}{\LpNorm{K_n\paren{\theta_1}}{2}}K_n\paren{\theta_1}}{2}\geq \frac{\LpNorm{K_n\paren{\theta_2}-K_n\paren{\theta_1}}{2}}{\Lambda_{\max}}.
\end{equation}
\end{lem}

\begin{proof}
Let $\varXi$ denote the left hand side of \eqref{LowBndModalIneq} and define $\lambda\coloneqq \LpNorm{K_n\paren{\theta_1}}{2}^{-1}\LpNorm{K_n\paren{\theta_2}}{2}$. Applying Lemma $A.1.$ of \cite{F.Bachoc} along obvious calculations imply that 
\begin{eqnarray*}
\varXi^2 &=& \sum\limits_{i=1}^{n} \set{\paren{\lambda-1}^2 + \sum\limits_{j\ne i} \Bigparen{ K\paren{s_i-s_j,\theta_2}-\lambda K\paren{s_i-s_j,\theta_1}}^2}\\
&\geq& \sum\limits_{i=1}^{n} \frac{\sum\limits_{j=1}^{n}\Bigparen{K\paren{s_i-s_j,\theta_2}-K\paren{s_i-s_j,\theta_1}}^2}{\sum\limits_{j=1}^{n}K^2\paren{s_i-s_j,\theta_1}}\geq \frac{\LpNorm{K_n\paren{\theta_2}-K_n\paren{\theta_1}}{2}^2}{ \max_{i=1,\ldots,n}\sum\limits_{j=1}^{n}K^2\paren{s_i-s_j,\theta_1} }.
\end{eqnarray*}
Notice that the operator norm of any matrix is no smaller than the largest $\ell_2$ norm of its columns. That is,
\begin{equation*}
\max_{i=1,\ldots,n}\sum\limits_{j=1}^{n}K^2\paren{s_i-s_j,\theta_1}\leq \OpNorm{K_n\paren{\theta_1}}{2}{2}^2\leq \Lambda^2_{\max}.
\end{equation*}
The combination of these two inequalities ends the proof.
\end{proof}

Now we state a lemma which plays a crucial role in the proof of Theorem \ref{GlobMinRate}.

\begin{lem}\label{Hanson-WrightConcIneq}
Let $Z\in\bb{R}^n$ be a standard Gaussian vector and let $C,\xi>0$. Suppose that $\Theta$ and $K$ satisfy Assumption \ref{Covar&ParamSpAssu}. Select 
$\theta_1,\theta_2\in\Theta$ such that
\begin{equation}\label{CondHWLem}
\LpNorm{\theta_2-\theta_1}{2}\geq C_{\min}\sqrt{\frac{\ln n}{n}},
\end{equation}
in which $C_{\min}\coloneqq 4\ff{D}^{-1}_{\min}\Lambda^2_{\max}\sqrt{C'\paren{1+\xi}}$ (Recall $\ff{D}_{\min}$ and $\Lambda_{\max}$, from the Proposition \ref{UnifBndEigValCorrMatrix}), for some appropriately chosen universal constant $C'>0$. There exists $n_0 = \cc{O}\paren{1}$ (depending on $C$, $\xi$, $K$, $\cc{D}_n$ and $\Theta$) such that for any $n\geq n_0$
\begin{equation}\label{Hanson-WrightIneq}
p\coloneqq\Pr\set{ Z^\top \paren{H_{\theta_2,\theta_2}-H_{\theta_2,\theta_1} }Z\leq  C\sqrt{\frac{\ln n}{n}}}\leq n^{-\paren{1+\xi}}.
\end{equation}
Refer to the identity \eqref{MandH} for the definition of $H_{\theta_2,\theta_1}$.
\end{lem}

\begin{proof}
We first obtain a lower bound on $\LpNorm{K_n\paren{\theta_2}-K_n\paren{\theta_1}}{2}$. Observe that,
\begin{equation}\label{ConseqCondHWLem}
\LpNorm{K_n\paren{\theta_2}-K_n\paren{\theta_1}}{2}\RelNum{\paren{A}}{\geq} \ff{D}_{\min}\sqrt{n}\LpNorm{\theta_2-\theta_1}{2} \RelNum{\paren{B}}{\geq} 4\Lambda^2_{\max}\sqrt{C'\paren{1+\xi}\ln n},
\end{equation}
where $\paren{A}$ follows from inequality \eqref{L2NormofDiffLowBnd} in Proposition \ref{UnifBndEigValCorrMatrix} and $\paren{B}$ is implied by condition \eqref{CondHWLem}. For brevity, let $\ff{A}\coloneqq H_{\theta_2,\theta_2}-H_{\theta_2,\theta_1}$. It is trivial from the identity \eqref{MandH} that
\begin{equation}\label{ffAform}
\ff{A} = K^{1/2}_n\paren{\theta_2}\paren{ \frac{K_n\paren{\theta_2}}{\LpNorm{K_n\paren{\theta_2}}{2}}-\frac{K_n\paren{\theta_1}}{\LpNorm{K_n\paren{\theta_1}}{2}} }K^{1/2}_n\paren{\theta_2}.
\end{equation}
Obtaining a closed form for the trace and \emph{Frobenius} norm of $\ff{A}$, which express the expected value and variance of $Z^\top\ff{A}Z$, plays a significant role in the proof. Straightforward algebra leads to
\begin{equation}\label{traceA}
\tr\paren{\ff{A}} = \LpNorm{K_n\paren{\theta_2}}{2} - \frac{\InnerProd{K_n\paren{\theta_1}}{K_n\paren{\theta_2}}{}}{\LpNorm{K_n\paren{\theta_1}}{2}} = \frac{\LpNorm{K_n\paren{\theta_2}}{2}}{2}\LpNorm{\frac{K_n\paren{\theta_2}}{\LpNorm{K_n\paren{\theta_2}}{2}}-\frac{K_n\paren{\theta_1}}{\LpNorm{K_n\paren{\theta_1}}{2}}}{2}^2.
\end{equation}
Applying Lemma \ref{ModalLem2} to the identity \eqref{traceA} shows that
\begin{eqnarray}\label{TrALowBnd}
\tr\paren{\ff{A}}&\geq& \frac{1}{2\LpNorm{K_n\paren{\theta_2}}{2}}\paren{\frac{\LpNorm{K_n\paren{\theta_2}-K_n\paren{\theta_1}}{2}}{\Lambda_{\max}}}^2\geq \frac{1}{2\Lambda_{\max}\sqrt{n}} \paren{\frac{\LpNorm{K_n\paren{\theta_2}-K_n\paren{\theta_1}}{2}}{\Lambda_{\max}}}^2\nonumber\\
&\RelNum{\paren{C}}{\geq}& 8C'\Lambda_{\max}\frac{\paren{1+\xi}\ln n}{\sqrt{n}}.
\end{eqnarray}
Note that $\paren{C}$ follows from the condition \eqref{ConseqCondHWLem}, and \eqref{TrALowBnd} assures the existence of some $n_2\in\bb{N}$ for which $\tr\paren{\ff{A}}\geq 2C\sqrt{n^{-1}\ln n}$, whenever $n\geq n_2$. Hence, $p\leq q\coloneqq\Pr\set{Z^\top\ff{A}Z\leq1/2\tr\paren{A}}$. We aim to show that $q$ is smaller than $n^{-\paren{1+\xi}}$. Based upon Hanson-Wright inequality there is a bounded universal constant $C'>0$ for which 
\begin{equation}\label{HWIneq}
\Pr\set{Z^\top\ff{A}Z \leq \tr\paren{\ff{A}} - \sqrt{C'\paren{1+\xi}\ln n} \paren{\LpNorm{\ff{A}}{2} \vee \OpNorm{\ff{A}}{2}{2}\sqrt{\paren{1+\xi}\ln n}} }\leq n^{-\paren{1+\xi}}.
\end{equation} 
Thus it suffices to show that for some $n_3\in\bb{N}$
\begin{equation}\label{LowBndTr}
\tr\paren{\ff{A}}\geq 2\sqrt{C'\paren{1+\xi}\ln n} \paren{\LpNorm{\ff{A}}{2} \vee \OpNorm{\ff{A}}{2}{2}\sqrt{\paren{1+\xi}\ln n}},\quad \forall\;n\geq n_3.
\end{equation}
\setcounter{clawithinpf}{0}
	
\begin{clawithinpf}\label{Claim1AuxRes}
Define $a_0\coloneqq \Lambda_{\min}\ff{D}_{\min}\set{ \Lambda^3_{\max}\ff{D}_{\max}\paren{1+\Lambda_{\max}} }^{-1}$. Then, $\LpNorm{\ff{A}}{2}\geq a_0\sqrt{n}\OpNorm{\ff{A}}{2}{2}$.
\end{clawithinpf} 
	
\begin{proof}[Proof of Claim \ref{Claim1AuxRes}]
In order to substantiate the claim, we require a lower bound on $\LpNorm{\ff{A}}{2}$ and an upper bound on $\OpNorm{\ff{A}}{2}{2}$. Applying the part $\paren{a}$ of Proposition \ref{UnifBndEigValCorrMatrix} (alongside the Cauchy-Schwartz inequality) on the identity \eqref{ffAform} shows that
\begin{align}\label{CondNumbCont1}
\Lambda^{-1}_{\max}\OpNorm{\ff{A}}{2}{2}&\leq L_{\mathop{up}}\coloneqq\OpNorm{\frac{K_n\paren{\theta_2}}{\LpNorm{K_n\paren{\theta_2}}{2}}-\frac{K_n\paren{\theta_1}}{\LpNorm{K_n\paren{\theta_1}}{2}}}{2}{2},\nonumber\\
\Lambda^{-1}_{\min}\LpNorm{\ff{A}}{2}&\geq L_{\mathop{lw}}\coloneqq \LpNorm{\frac{K_n\paren{\theta_2}}{\LpNorm{K_n\paren{\theta_2}}{2}}-\frac{K_n\paren{\theta_1}}{\LpNorm{K_n\paren{\theta_1}}{2}}}{2}.
\end{align}
We achieve an upper bound on $L_{\mathop{up}}$ by taking advantage of the same technique as \eqref{UppBndLHS2}. Furthermore, a tight lower bound on $L_{\mathop{lw}}$ can be obtained from Lemma \ref{ModalLem2}. The algebraic details are skipped to avoid redundancy.
\begin{equation}\label{CondNumbCont2}
L_{\mathop{up}}\leq \ff{D}_{\max}\paren{1+\Lambda_{\max}}\frac{\LpNorm{\theta_2-\theta_1}{2} }{\sqrt{n}},\quad
L_{\mathop{lw}}\geq \frac{\ff{D}_{\min}\LpNorm{\theta_2-\theta_1}{2} }{\Lambda^2_{\max}}.
\end{equation}
The combination of \eqref{CondNumbCont1} and \eqref{CondNumbCont2} finishes the proof.
\end{proof}
	
With Claim \ref{Claim1AuxRes} in hand, proving \eqref{LowBndTr} is equivalent to showing that
\begin{equation}\label{TrDivFrobNormffA}
\tr\paren{\ff{A}}\geq 2C'\sqrt{\paren{1+\xi}\ln n}\LpNorm{\ff{A}}{2},
\end{equation}
for large enough $n$. Obviously, $\LpNorm{\ff{A}}{2}\leq \Lambda_{\max} L_{\mathop{lw}}$. Thus, identity \eqref{traceA} implies that 
\begin{eqnarray}\label{TrDivFrobNormffA2}
\frac{\tr\paren{\ff{A}}}{\LpNorm{\ff{A}}{2}}&\geq& \frac{\LpNorm{K_n\paren{\theta_2}}{2}}{2\Lambda_{\max} L_{\mathop{lw}}}\LpNorm{\frac{K_n\paren{\theta_2}}{\LpNorm{K_n\paren{\theta_2}}{2}}-\frac{K_n\paren{\theta_1}}{\LpNorm{K_n\paren{\theta_1}}{2}}}{2}^2 \nonumber\\
&=& b_0\coloneqq \frac{1}{2\Lambda_{\max}}\LpNorm{K_n\paren{\theta_2}-\frac{\LpNorm{K_n\paren{\theta_2}}{2}}{\LpNorm{K_n\paren{\theta_1}}{2}}K_n\paren{\theta_1}}{2}
\end{eqnarray}
In the sequel, notice that
\begin{equation*}
b_0\RelNum{\paren{D}}{\geq}\frac{\LpNorm{K_n\paren{\theta_2}-K_n\paren{\theta_1}}{2}}{2\Lambda^2_{\max}}\RelNum{\paren{E}}{\geq}2\sqrt{C'\paren{1+\xi}\ln n},
\end{equation*}
in which $\paren{D}$ and $\paren{E}$ respectively follow from Lemma \ref{ModalLem2} and condition \eqref{ConseqCondHWLem}. Substituting the last inequality into \eqref{TrDivFrobNormffA2} verifies \eqref{TrDivFrobNormffA} and concludes the proof.
\end{proof}

The next proposition rigorously expresses the uniform concentration of the Euclidean squared norm of Gaussian vectors with the covariance matrix $K_n\paren{\theta},\;\theta\in\Theta$ around their mean. It is worthwhile to mention that such inequality is crucial for proving the Theorem \ref{LocMinRate}.

\begin{prop}\label{ModalProp}
Let $\Theta\subset\bb{R}^m$ be a bounded set. Consider the class of $n$ by $n$ matrices $\set{\Pi_n\paren{\theta}}_{\theta\in\Theta}$ parametrized by $\theta\in\Theta$. Suppose that the following conditions hold
	
\begin{enumerate}[label = (\alph*),leftmargin=*]
\item The operator norm of $\Pi_n\paren{\theta}$ is uniformly bounded in $\Theta$. Namely,
\begin{equation*}
M\coloneqq\sup_{n}\sup_{\theta\in\Theta} \OpNorm{\Pi_n\paren{\theta}}{2}{2} < \infty.
\end{equation*}
\item The mapping $\paren{\theta,\LpNorm{\cdot}{2}}\mapsto \paren{\Pi_n\paren{\theta},\OpNorm{\cdot}{2}{2} }$ is Lipschitz. Namely, there is $C>0$ for which
\begin{equation}\label{LipscCond}
\OpNorm{\Pi_n\paren{\theta_2}-\Pi_n\paren{\theta_1}}{2}{2} \leq C\LpNorm{\theta_2-\theta_1}{2},\quad\forall\;\theta_1,\theta_2\in\Theta.
\end{equation}
\item 
\begin{equation*}
\frac{\OpNorm{\Pi_n\paren{\theta}}{2}{2}}{\LpNorm{\Pi_n\paren{\theta}}{2}}= o\paren{\frac{1}{\sqrt{\ln n}}}, \quad\forall\;\theta\in\Theta.
\end{equation*}
\end{enumerate}
Then, there is a constant $C'>0$ such that
\begin{equation}\label{UnifDeviationBound}
\lim\limits_{n\rightarrow\infty}\Pr\paren{ \sup_{\theta\in\Theta}\abs{Z^\top \Pi_n\paren{\theta}Z - \tr\set{\Pi_n\paren{\theta}} } \geq C'\sqrt{n\ln n} } = 0.
\end{equation}
\end{prop}

\begin{proof}
Let $r_n = C^{-1}\sqrt{\ln n/ n}$ in which $C$ has been defined in \eqref{LipscCond} and let $\cc{N}_{r_n}\paren{\Theta}$ denote the $r_n$-covering set of $\Theta$. As before, for any $\theta$ let $\beta_{\theta}$ represents the closest element of $\cc{N}_{r_n}\paren{\Theta}$ to $\theta$. Observe that,
\begin{align*}
&\RHS \coloneqq \abs{Z^\top \Pi_n\paren{\theta}Z - \tr\set{\Pi_n\paren{\theta}} - Z^\top \Pi_n\paren{\beta_{\theta}}Z + \tr\set{\Pi_n\paren{\beta_{\theta}}}} \\
&= \abs{\InnerProd{\Pi_n\paren{\theta}- \Pi_n\paren{\beta_{\theta}} }{ZZ^\top+I_n}}
\leq \OpNorm{\Pi_n\paren{\theta}- \Pi_n\paren{\beta_{\theta}}}{2}{2}\SpNorm{ZZ^\top+I_n}{1}\\
&\RelNum{\paren{A}}{\leq} C\LpNorm{\theta-\beta_{\theta}}{2}\SpNorm{ZZ^\top+I_n}{1}
\leq Cr_n\SpNorm{ZZ^\top+I_n}{1} = \sqrt{\frac{\ln n}{n}}\paren{n+\LpNorm{Z}{2}},
\end{align*}
in which $\SpNorm{\cdot}{1}$ stands for the nuclear norm (absolute sum of eigenvalues). Note that the obtained upper bound does not depend on $\theta$ (uniform upper bound). Moreover, based upon Hanson-Wright inequality there is $c>0$ for which $\paren{n+\LpNorm{Z}{2}}\leq 3n$ with probability at least $1-\exp\paren{-cn}$. Thus, $\RHS\geq 3\sqrt{n\ln n}$ with probability at most $\exp\paren{-cn}$. Hence, as $n\rightarrow\infty$ we get
\begin{equation}\label{ModalIneq1}
\Pr\paren{ \sup_{\theta\in\Theta}\abs{Z^\top \Pi_n\paren{\theta}Z - \tr\set{\Pi_n\paren{\theta}} } \geq \sup_{\theta\in\cc{N}_{r_n}\paren{\Theta}}\abs{Z^\top \Pi_n\paren{\theta}Z - \tr\set{\Pi_n\paren{\theta}} }+3\sqrt{n\ln n} } \rightarrow 0.
\end{equation}
In the sequel we find a tight upper bound on $\sup_{\theta\in\cc{N}_{r_n}\paren{\Theta}}\abs{Z^\top \Pi_n\paren{\theta}Z - \tr\set{\Pi_n\paren{\theta}} }$. Applying condition $\paren{c}$ on Hanson-wright inequality and using union bound leads to
\begin{equation*}
\Pr\paren{\sup_{\theta\in\cc{N}_{r_n}\paren{\Theta}}\abs{Z^\top \Pi_n\paren{\theta}Z - \tr\set{\Pi_n\paren{\theta}} }\geq C_0\sup_{\theta\in\Theta}\LpNorm{\Pi_n\paren{\theta}}{2}\sqrt{\ln n}}\leq \abs{\cc{N}_{r_n}\paren{\Theta}}n^{-m},
\end{equation*}
for some constant $C_0>0$ depending on $m$. Notice that $\sup_{\theta\in\Theta}\LpNorm{\Pi_n\paren{\theta}}{2}\leq M\sqrt{n}$ according to condition $\paren{a}$. Moreover, as we argued in \eqref{UppBndT22}, $\abs{\cc{N}_{r_n}\paren{\Theta}}=o\paren{n^m}$. Thus, 
\begin{equation*}
\lim\limits_{n\rightarrow\infty} \Pr\paren{\sup_{\theta\in\cc{N}_{r_n}\paren{\Theta}}\abs{Z^\top \Pi_n\paren{\theta}Z - \tr\set{\Pi_n\paren{\theta}} }\geq C_0M\sqrt{n\ln n}} =0.
\end{equation*}
Replacing the last inequality into \eqref{ModalIneq1} concludes the proof.
\end{proof}

\begin{lem}\label{MaodlLem3}
Under the same notation and assumptions as Propositions \ref{UnifBndEigValCorrMatrix} and \ref{UnifBndEigValCorrMatrix2} and for any unit norm vector $\lambda$, the class of matrices defined by 
\begin{equation*}
T_n\paren{\theta}\coloneqq \lim\limits_{\gamma\searrow 0} \gamma^{-1}\set{ K_n\paren{\theta+\gamma\lambda} - \InnerProd{K_n\paren{\theta+\gamma\lambda} }{K_n\paren{\theta}} \frac{K_n\paren{\theta}}{\LpNorm{K_n\paren{\theta}}{2}^2 } }
\end{equation*} 
satisfy the conditions $\paren{a}$ and $\paren{b}$ in Proposition \ref{ModalProp}.
\end{lem}

\begin{proof}
The following inequality uniformly holds for any $\theta$ and $n$.
\begin{eqnarray*}
\OpNorm{T_{n}\paren{\theta}}{2}{2}
&\RelNum{\paren{A}}{\leq}& \lim\limits_{\gamma\searrow 0} \gamma^{-1}\set{ \OpNorm{K_n\paren{\theta+\gamma\lambda}-K_n\paren{\theta}}{2}{2}+ \OpNorm{K_n\paren{\theta}}{2}{2} \abs{1- \frac{  \InnerProd{K_n\paren{\theta+\gamma\lambda}}{K_n\paren{\theta}} }{\LpNorm{K_n\paren{\theta}}{2}^2} } }\\
&\RelNum{\paren{B}}{\leq}& \lim\limits_{\gamma\searrow 0} \gamma^{-1}\set{ \OpNorm{K_n\paren{\theta+\gamma\lambda}-K_n\paren{\theta}}{2}{2}+ \frac{\OpNorm{K_n\paren{\theta}}{2}{2}}{\LpNorm{K_n\paren{\theta}}{2}} \LpNorm{K_n\paren{\theta+\gamma\lambda}-K_n\paren{\theta}}{2} } \\
&\RelNum{\paren{C}}{\leq}& \lim\limits_{\gamma\searrow 0} \gamma^{-1}\paren{ \ff{D}_{\max}\gamma + \frac{\Lambda_{\max}}{\sqrt{n}} \sqrt{n} \ff{D}_{\max}\gamma } = \ff{D}_{\max}\paren{1+\Lambda_{\max}}<\infty.
\end{eqnarray*}
In above, $\paren{A}$ and $\paren{B}$ are simple applications of triangle and Cauchy-Schwartz inequalities, respectively. Furthermore, $\paren{C}$ follows from Propositions \ref{UnifBndEigValCorrMatrix}. Namely, $\sup_{n\in\bb{N}}\sup_{\theta\in\Theta} \OpNorm{T_{n}\paren{\theta}}{2}{2} < \infty$ which is same as condition $\paren{a}$. 
	
The next phase of the proof is devoted to justify condition $\paren{b}$. We have introduced an equivalent representation for $T_n\paren{\theta}$ in \eqref{Tn}. Define
\begin{equation*}
a_n\paren{\theta} \coloneqq \LpNorm{K_n\paren{\theta}}{2}^{-2} \InnerProd{\sum\limits_{j=1}^{m}\lambda_j\frac{\partial}{\partial\theta_j}K_n\paren{\theta}}{K_n\paren{\theta}}
\end{equation*}
It follows from Propositions \ref{UnifBndEigValCorrMatrix} and \ref{UnifBndEigValCorrMatrix2} that both matrices $K_n\paren{\theta}$ and $\sum\limits_{j=1}^{m}\lambda_j\partial K_n\paren{\theta}/\partial\theta_j $ admit conditions $\paren{a}$ and $\paren{b}$ in Proposition \ref{ModalProp}. Thus, according to Lemma \ref{ModalLem4}, it is sufficient to prove that $a_n\paren{\theta}$ is a uniformly bounded and Lipschitz function in $\Theta$. Notice that,
\begin{eqnarray}\label{DirDeriv1}
\sup_{\theta\in\Theta}\abs{a_n\paren{\theta}}&\leq& \sup_{\theta\in\Theta} \frac{\LpNorm{\sum\limits_{j=1}^{m}\lambda_j\frac{\partial}{\partial\theta_j}K_n\paren{\theta}}{2} }{\LpNorm{K_n\paren{\theta}}{2}}\leq \sup_{\theta\in\Theta} \sqrt{\frac{1}{n}}\sum\limits_{j=1}^{m}\abs{\lambda_j}\LpNorm{\frac{\partial}{\partial\theta_j}K_n\paren{\theta}}{2}\nonumber\\
&\leq& \frac{\Lambda'_{\max}\sqrt{n}\LpNorm{\lambda}{1} }{\sqrt{n}} \leq \Lambda'_{\max}\sqrt{m}, \quad\forall\;n\in\bb{N}.
\end{eqnarray}
In the sequel, we justify the Lipschitz property. Taking advantage of Proposition \ref{UnifBndEigValCorrMatrix2} and inequality \eqref{DirDeriv1}, one show that for any $\theta\in\Theta$,
\begin{align}\label{DirDeriv2}
&\LpNorm{\sum\limits_{j=1}^{m}\lambda_j\frac{\partial}{\partial\theta_j}K_n\paren{\theta}}{2}\leq \Lambda'_{\max}\sqrt{mn},\nonumber\\
&\LpNorm{\sum\limits_{j=1}^{m}\lambda_j\paren{\frac{\partial}{\partial\theta_j}K_n\paren{\theta_2}-\frac{\partial}{\partial\theta_j}K_n\paren{\theta_1}}}{2}\leq \ff{D}'_{\max}\sqrt{mn}\LpNorm{\theta_2-\theta_1}{2}.
\end{align}
Thus
\begin{eqnarray*}
\abs{a_n\paren{\theta_2}-a_n\paren{\theta_1}}&\RelNum{\paren{D}}{\leq}& \abs{\InnerProd{\sum\limits_{j=1}^{m}\lambda_j\paren{\frac{\partial}{\partial\theta_j}K_n\paren{\theta_2}-\frac{\partial}{\partial\theta_j}K_n\paren{\theta_1}}}{\frac{K_n\paren{\theta_1}}{\LpNorm{K_n\paren{\theta_1}}{2}^2 }}} \\
&+& \abs{\InnerProd{\sum\limits_{j=1}^{m}\lambda_j\frac{\partial}{\partial\theta_j}K_n\paren{\theta_2}}{{\frac{K_n\paren{\theta_2}}{\LpNorm{K_n\paren{\theta_2}}{2}^2}-\frac{K_n\paren{\theta_1}}{\LpNorm{K_n\paren{\theta_1}}{2}^2}}} }\\
&\leq& \frac{ \LpNorm{\sum\limits_{j=1}^{m}\lambda_j\paren{\frac{\partial}{\partial\theta_j}K_n\paren{\theta_2}-\frac{\partial}{\partial\theta_j}K_n\paren{\theta_1}}}{2} }{\LpNorm{K_n\paren{\theta_1}}{2}}\\
&+&\LpNorm{\sum\limits_{j=1}^{m}\lambda_j\frac{\partial}{\partial\theta_j}K_n\paren{\theta_2}}{2}\LpNorm{\frac{K_n\paren{\theta_2}}{\LpNorm{K_n\paren{\theta_2}}{2}^2}-\frac{K_n\paren{\theta_1}}{\LpNorm{K_n\paren{\theta_1}}{2}^2}}{2}\\
&\RelNum{\paren{E}}{\leq}& \frac{\sqrt{mn}\ff{D}'_{\max}\LpNorm{\theta_2-\theta_1}{2} }{\sqrt{n}}+\sqrt{mn}\Lambda'_{\max}\sqrt{n}\OpNorm{\frac{K_n\paren{\theta_2}}{\LpNorm{K_n\paren{\theta_2}}{2}^2}-\frac{K_n\paren{\theta_1}}{\LpNorm{K_n\paren{\theta_1}}{2}^2}}{2}{2}\\
&\RelNum{\paren{F}}{\leq}& \ff{D}'_{\max}\sqrt{m}\LpNorm{\theta_2-\theta_1}{2} + 3\Lambda'_{\max}\sqrt{m}\Lambda_{\max}\ff{D}_{\max}\LpNorm{\theta_2-\theta_1}{2} = C_0\LpNorm{\theta_2-\theta_1}{2},
\end{eqnarray*}
where $\paren{D}$ is implied by triangle inequality, $\paren{E}$ follows from \eqref{DirDeriv2}. Lastly, $\paren{F}$ can be proved using akin tricks as \eqref{UppBndLHS2}. Thus, $a_n\paren{\cdot}$ is a Lipschitz function with some constant $C_0$. 
\end{proof}

The following technical result, which comes in handy in the proof of Lemma \ref{MaodlLem3}, is an easy consequence of triangle inequality. The proof is skipped for brevity.

\begin{lem}\label{ModalLem4}
Suppose that $a_1,a_2:\Theta\mapsto\bb{R}$ are two uniformly bounded and Lipschitz functions with constant $L^1_a$ and $L^2_a$, respectively. Consider matrices $A_1,A_2:\Theta\mapsto\bb{R}^{n\times n}$ parametrized by $\theta\in\Theta$ such that
\begin{equation*}
\OpNorm{A_i\paren{\theta_2}-A_i\paren{\theta_1}}{2}{2}\leq L^i_A\LpNorm{\theta_2-\theta_1}{2},\quad\;\forall\;\theta_1,\theta_2\in\Theta,\;i=1,2
\end{equation*}
Let $\ff{A}\paren{\theta}\coloneqq a_1\paren{\theta}A_1\paren{\theta}+a_2\paren{\theta}A_2\paren{\theta}$ for any $\theta$. There is a scalar $L$, which depends on $L_a$, $L_b$, $L_A$ and $L_B$, for which the following inequality holds.
\begin{equation*}
\OpNorm{\ff{A}\paren{\theta_2}-\ff{A}\paren{\theta_1}}{2}{2}\leq L\LpNorm{\theta_2-\theta_1}{2},\quad\forall\;\theta_1,\theta_2\in\Theta.
\end{equation*}
\end{lem}

\setcounter{clawithinpf}{0}

The following result gives an upper bound on the \emph{Kullback-Leibler} divergence of two zero mean multivariate Gaussian distributions respectively associated with the two covariance matrices $K_n\paren{\theta_i},\; i=1,2$. Such upper bound is extremely useful for establishing Theorem \ref{MinMaxThm}.

\begin{prop}\label{KLDivgncLem}
Choose $\theta_1,\theta_2\in\Theta$ in such a way that $\LpNorm{\theta_2-\theta_1}{2}\leq \Lambda_{\min}/\paren{2\ff{D}_{\max}}$. Let $P_i, i=1,2$, denotes the associated probability distribution to a zero mean Gaussian vector with the covariance matrix $K_n\paren{\theta_i}\in\bb{R}^{n\times n}, \; i=1,2$. Then,
\begin{equation*}
D\paren{P_1 \parallel P_2}\leq 2n\paren{\frac{\ff{D}_{\max}}{\Lambda_{\min}}\LpNorm{\theta_2-\theta_1}{2}}^2.
\end{equation*}
\end{prop}

\begin{proof}
For any symmetric matrix $A\in\bb{R}^{n\times n}$, let $\lambda_i\paren{A},i=1,\ldots,n$, denotes its $i^{\mbox{th}}$ eigenvalue in decreasing order. The \emph{von Neumann's trace inequality} \cite{L.Mirsky} yields
\begin{eqnarray*}
D\paren{P_1 \parallel P_2}&=& \InnerProd{K^{-1}_n\paren{\theta_2}}{K_n\paren{\theta_1}} - n + \ln\paren{\frac{\det K_n\paren{\theta_2}}{\det K_n\paren{\theta_1}}}\\
&\leq& Q\coloneqq \sum\limits_{j=1}^{n}\set{\frac{\lambda_j\paren{K_n\paren{\theta_1}}}{\lambda_j\paren{K_n\paren{\theta_2}}} - 1 - \ln \frac{\lambda_j\paren{K_n\paren{\theta_1}}}{\lambda_j\paren{K_n\paren{\theta_1}}}}.	
\end{eqnarray*} 
We finish the proof by acquiring a proper upper bound on $Q$. Define  $f:\paren{0,\infty}\mapsto\bb{R}$ by $f\paren{x} = \abs{x-1-\ln x}$. Applying the second order Taylor's expansion around $x=1$ shows that $f\paren{x}\leq 2\paren{x-1}^2$ for $\abs{x-1}\leq 1/2$. 
	
\begin{clawithinpf}\label{Claim1PropB4}
The succeeding inequality holds for any $j=1,\ldots,n$.
\begin{equation*}
\abs{\frac{\lambda_j\paren{K_n\paren{\theta_1}}}{\lambda_j\paren{K_n\paren{\theta_2}}}-1}\leq \frac{\ff{D}_{\max}}{\Lambda_{\min}}\LpNorm{\theta_2-\theta_1}{2}\leq \frac{1}{2}.
\end{equation*}
\end{clawithinpf}
Claim \ref{Claim1PropB4} provides the key tool to control $Q$ from above.
\begin{eqnarray*}
Q&=&\sum\limits_{j=1}^{n} f\brac{\frac{\lambda_j\set{K_n\paren{\theta_1}}}{\lambda_j\set{K_n\paren{\theta_2}}}}\leq 2\sum\limits_{j=1}^{n}\brac{\frac{\lambda_j\set{K_n\paren{\theta_1}}}{\lambda_j\set{K_n\paren{\theta_2}}}-1}^2\leq 2\sum\limits_{j=1}^n \paren{\frac{\ff{D}_{\max}}{\Lambda_{\min}}\LpNorm{\theta_2-\theta_1}{2}}^2\\
&=& 2n\paren{\frac{\ff{D}_{\max}}{\Lambda_{\min}} \LpNorm{\theta_2-\theta_1}{2}}^2.
\end{eqnarray*}
In conclusion, we substantiate Claim \ref{Claim1PropB4}. The first inequality can be established using Proposition \ref{UnifBndEigValCorrMatrix} and the second one is obvious.
\begin{equation*}
\abs{\frac{\lambda_j\set{K_n\paren{\theta_1}}}{\lambda_j\set{K_n\paren{\theta_2}}}-1} = \abs{\frac{\lambda_j\set{K_n\paren{\theta_1}}-\lambda_j\set{K_n\paren{\theta_2}}}{\lambda_j\set{K_n\paren{\theta_2}}}}\leq \frac{\OpNorm{K_n\paren{\theta_2}-K_n\paren{\theta_1}}{2}{2} }{ \lambda_n\set{K_n\paren{\theta_2}} }\leq \frac{\ff{D}_{\max}\LpNorm{\theta_2-\theta_1}{2} }{ \Lambda_{\min} }.
\end{equation*}
\end{proof}

Now we demonstrate the asymptotic normality of the normalized quadratic Gaussian forms. We exploit this fact in the proof of Theorem \ref{AsympDistThm}.

\begin{lem}\label{AsymNormQForm}
For $n\in\bb{N}$, let $Z_n\in\bb{R}^n$ be a standard Gaussian vector and let $A_n\in\bb{R}^{n\times n}$. Then,
\begin{equation*}
\Psi_n\coloneqq\set{\frac{Z^\top_nA_nZ_n - \tr\paren{A_n}}{\LpNorm{A_n}{2}}}\cp{d} \cc{N}\paren{0,2},
\end{equation*}
provided that $\lim\limits_{n\rightarrow\infty} \LpNorm{A_n}{2}^{-1}\OpNorm{A_n}{2}{2} = 0$.
\end{lem}

\begin{proof}
	
Let $\Psi_{\infty}$ be a zero mean Gaussian random variable with variance $2$. So, $\ln{\rm E} \exp\paren{t\Psi_{\infty}} = t^2$ for any $t\in\bb{R}$. The basic properties of the quadratic forms of Gaussian vectors yields
\begin{eqnarray*}
\ln{\rm E} \exp\paren{t\Psi_n} &=& -\frac{1}{2} \ln\det\paren{I_n-2t\frac{A_n}{\LpNorm{A_n}{2}}} -\frac{t\tr\paren{A_n}}{\LpNorm{A_n}{2}} \\
&=& -\frac{1}{2}\sum\limits_{j=1}^{n} \set{\ln\paren{1-\frac{2t\lambda_j\paren{A_n}}{\LpNorm{A_n}{2}}} + \frac{2t\lambda_j\paren{A_n}}{\LpNorm{A_n}{2}}}\\
&\RelNum{\paren{A}}{=}&\sum\limits_{j=1}^{n}\brac{\paren{\frac{t\lambda_j\paren{A_n}}{\LpNorm{A_n}{2}}}^2 + o\set{ \paren{\frac{t\lambda_j\paren{A_n}}{\LpNorm{A_n}{2}}}^2} } \rightarrow t^2,\quad\mbox{as}\;n\rightarrow\infty.
\end{eqnarray*}
Here $\paren{A}$ follows from expanding $\ln\paren{1-x}$ around $1$ for infinitesimal $x$ (since $\lambda_j\paren{A_n}/\LpNorm{A_n}{2}$ vanishes as $n\rightarrow\infty$). Consequently, $\Psi_n$ converges in distribution to $\Psi_{\infty}$ by the continuity theorem of moment generating functions.
\end{proof}

The last result of this section studies the shrinkage behaviour of the partial derivatives of Matern covariance function with respect to its fractal index. In turns out to be useful for corroborating the part $\paren{a}$ of Remark \ref{Rem3.1}.

\begin{lem}\label{MaternDecayDiffwrtNu}
Let $K_{\nu}:\bb{R}^d\mapsto\bb{R}$ be a geometric anisotropic (Recall from Definition \ref{GeoAnisGPDefn}) Matern correlation function given by 
\begin{equation*}
K_{\nu}\paren{r} = \frac{2^{1-\nu}}{\Gamma\paren{\nu}} \ff{K}_{\nu}\paren{\sqrt{r^\top Ar}},
\end{equation*}
in which $\ff{K}_{\nu}$ stands for the modified Bessel function of the second kind and $A$ satisfies the condition \ref{EigValsUppLwBnd}. Then for any $\beta\in\bb{N}$ and $m\in\bb{N}$, there is a bounded constant $C_{\beta,A}$ such that
\begin{equation}\label{MatrDerivDecaywrtNu}
\abs{\frac{\partial^m}{\partial \nu^m} K_{\nu}\paren{r}} \leq \frac{C_{\beta,A}}{1+\LpNorm{r}{2}^{2\beta}},\quad\forall\;r=\paren{r_1,\ldots,r_d}\in\bb{R}^d.
\end{equation}
\end{lem}

\begin{proof}
Choose $\beta,m\in\bb{N}$ arbitrarily. The inequality \eqref{MatrDerivDecaywrtNu} obviously holds if we can show that 
\begin{equation}\label{EquivIneqMatrDerivDecaywrtNu}
\sup_{r\in\bb{R}^d}\abs{\paren{1+\sum_{i=1}^{d} r^{2\beta}_i } \frac{\partial^m}{\partial \nu^m} K_{\nu}\paren{r} } < \infty.
\end{equation}
Let $\cc{F}$ stands for the Fourier operator. For proving \eqref{EquivIneqMatrDerivDecaywrtNu} it suffices to show that 
\begin{equation}\label{LastClaim}
\cc{F}\set{\paren{1+\sum_{i=1}^{d} r^{2\beta}_i } \frac{\partial^m}{\partial \nu^m} K_{\nu}\paren{r}}\in L^1(\bb{R}^d),
\end{equation}
where $L^1(\bb{R}^d)$ stands for the class of absolutely integrable $d-$variate functions. For doing so we need a closed form expression for the Matern spectral density $\hat{K}_{\nu}$, which can be done using the integration by substitution technique.
\begin{equation*}
\hat{K}_{\nu}\paren{\omega} = \frac{\pi^{-d/2}}{\abs{\det\paren{A}}}\paren{1+\omega^\top A^{-1}\omega}^{-\paren{\nu+d/2}}
\end{equation*}
The condition \ref{EigValsUppLwBnd} says that $\hat{K}_{\nu}$ decays polynomially in terms of $\LpNorm{\omega}{2}$. We now evaluate the Fourier transform of $\frac{\partial^m K_{\nu}}{\partial \nu^m}$. Observe that
\begin{equation*}
\frac{\partial^m K_{\nu}}{\partial \nu^m} = \frac{\partial^m}{\partial \nu^m} \int_{\bb{R}^d} e^{-j\InnerProd{r}{\omega}}\hat{K}_{\nu}\paren{\omega} d\omega  = \paren{-1}^m \int_{\bb{R}^d} e^{-j\InnerProd{r}{\omega}}\hat{K}_{\nu}\paren{\omega} \ln^m\paren{1+\omega^\top A^{-1}\omega} d\omega.
\end{equation*}
In other words, $\cc{F}\paren{\frac{\partial^m K_{\nu}}{\partial \nu^m} } =   \hat{G}_{\nu}\paren{\omega} \coloneqq \paren{-1}^m \hat{K}_{\nu}\paren{\omega} \ln^m\paren{1+\omega^\top A^{-1}\omega}$. Thus,
\begin{equation*}
\cc{F}\set{ \paren{1+\sum_{i=1}^{d} r^{2\beta}_i } \frac{\partial^m}{\partial \nu^m} K_{\nu}\paren{r} } = \hat{G}_{\nu} + \paren{-1}^{\beta} \sum_{i=1}^{d} \frac{\partial^{2\beta} \hat{G}_{\nu}}{\partial \omega^{2\beta}_i}.
\end{equation*}
The absolutely integrability of $\hat{G}_{\nu}$ is clear. The reader can verify the same property for the functions $\frac{\partial^{2\beta} \hat{G}_{\nu}}{\partial \omega^{2\beta}_i},\; i=1,\ldots,d$ using simple differentiation. So \eqref{LastClaim} holds by the triangle inequality.
\end{proof}



\end{document}